\documentclass[12pt]{amsart}

\usepackage{tikz}
\usepackage{amsmath, verbatim}
\usepackage{amssymb,amsfonts,mathrsfs}
\usepackage[colorlinks=true,linkcolor=blue,citecolor=blue]{hyperref}
\usepackage[driver=pdftex,margin=3cm,heightrounded=true,centering]{geometry}

\usepackage{mllocal} 

\newcommand{\U}{\mathscr{U}}

\newcommand{\eh}{\mathrm{e-h}}

\newcommand{\HF}{H_{\mathscr{F}}}

\newcommand{\calG}{\mathcal G}

\newcommand\bbZ{\mathbb{Z}}

\newcommand\bbC{\mathbb{C}}

\DeclareMathOperator*{\res}{Res}
 \usepackage{color}
\definecolor{dg}{cmyk}{1,0,1,.2}
\definecolor{m}{rgb}{1,0.1,1}

\numberwithin{equation}{section}

\begin{document}

\title{Eta and rho invariants on manifolds with edges}

\author{Paolo Piazza}
\address{Sapienza Universit\`a di Roma \\ Dipartimento di Matematica\\ Roma, Italy.}
\email{piazza@mat.uniroma1.it}

\author{Boris Vertman}
\address{University of Oldenburg \\ Institute of Mathematics \\
Oldenburg, Germany.}
\email{boris.vertman@uni-oldenburg.de}


\date{This document was compiled on: \today}

\begin{abstract}
We establish existence of eta-invariants as well as of the Atiyah-Patodi-Singer and 
the Cheeger-Gromov rho-invariants for a class of Dirac operators on an incomplete edge space.  
Our analysis  applies in particular to the signature and the spin Dirac 
operator. We derive an analogue of the Atiyah-Patodi-Singer index theorem for incomplete
edge spaces and their non-compact infinite Galois coverings with edge singular boundary.
Our arguments employ microlocal analysis of the heat kernel asymptotics on incomplete
edge spaces and the classical argument of Atiyah-Patodi-Singer. As an application,  we discuss  
stability results for the two rho-invariants we have defined.
\end{abstract}

\maketitle

\tableofcontents

\section{Introduction and statement of the main result}\label{intro}

\subsection{Eta function and eta invariant}
Consider a compact Riemannian manifold $(M,g)$ with $\dim M = m$ 
and assume for the moment that $M$ is closed and smooth, i.e. without singularities. 
Consider a Hermitian vector bundle $(E,h)$ over $M$ and denote the space 
of smooth sections over $M$ with values in the vector bundle $E$ by $\Gamma(M,E)$.
Let $D$ be a linear elliptic self-adjoint
first order differential operator $D$ acting on the smooth sections $\Gamma(M,E)$. 
The eta invariant of $D$ was introduced by Atiyah, Patodi and Singer \cite{APSa}, 
see also \cite{APSb, APSc}, to measure the spectral asymmetry of $D$ in the following way. 
Consider an enumeration $\{\lambda_n\}_{n\in \N_0}$ of 
the non-zero eigenvalues of $D$, counted with their multiplicities and ordered e.g. in ascending order. 
Denote by $\textup{sign}(\lambda)$ the sign of an eigenvalue 
$\lambda$. Then the eta function of $D$ is defined by 
\begin{align*}
\eta(D,s) := \sum_{n=0}^\infty \textup{sign}(\lambda_n) \, | \lambda_n| ^{-s}, \quad \Re(s) > m.
\end{align*}
This series is absolutely convergent for $\Re(s) > m$ and, as a consequence 
of the short time asymptotics of the trace $\textup{Tr} \, De^{-tD^2}$, extends to a meromorphic 
function on the whole of $\C$ by the following integral expression
\begin{equation}\label{eta-heat}
\eta(D,s) = \frac{1}{\Gamma((s+1)/2)}
\int_0^\infty t^{(s-1)/2} \, \textup{Tr} \, De^{-tD^2} dt
\end{equation}
Regularity of the eta function at $s=0$ is a general phenomenon that 
has been established by Atiyah, Patodi and Singer \cite{APSa} 
by viewing $M$ as boundary of a half cylinder $M \times \R^+$ and studying 
the Dirac-type operator $\mathbb{D} := (\partial_u + D)$ on the cylinder with Atiyah-Patodi-Singer boundary
conditions at $u=0$. Here, $u\in \R^+$ denotes the variable along the cylinder.
If the eta-function is regular at $s=0$, its value at zero is defined as the eta-invariant of $D$.\medskip

Under additional assumptions on the operator $D$, namely $D$ is a Dirac operator associated to a unitary
Clifford action and to a Clifford connection, one can prove that the right hand side of \eqref{eta-heat}, evaluated
at $s=0$, namely
\begin{align*}
\frac{1}{\sqrt{\pi}}
\int_0^\infty \, \textup{Tr} \, De^{-tD^2} \frac{dt}{\sqrt{t}}\,,
\end{align*}
converges (see \cite{BiFr}) and this value is the eta invariant of $D$.

\medskip
We point out  that if $M$ is even dimensional and $D$ is an odd $\bbZ_2$-graded Dirac-type operator acting
on the sections of $E=E^+ \oplus E^-$
\begin{equation*} D = \left(
\begin{array}{cc}
0 & D^- \\ D^+ & 0
\end{array}\right),
\end{equation*}
then $De^{-tD^2}$ is an odd operator and thus its trace vanishes identically
for all $t>0$
\begin{align*}
\textup{Tr} \, De^{-tD^2}\equiv 0.
\end{align*}
Thus the integrand in the left hand side of \eqref{eta-heat} vanishes identically and consequently the eta invariant
of $D$ is equal to zero.
Needless to say, this argument breaks down if we perturb $D$ by a potential which is not odd, possibly leading to a non-trivial
eta-invariant in even dimensions. Another geometric example of a non-trivial eta invariant in even dimensions is given by the Dirac operator on pin and pin$_c$-manifolds; see for example 
Gilkey \cite{Gilkey-PIN} \cite{Gilkey-PIN-PIN}. 

\medskip

The eta function has been studied extensively in various
settings, for instance by Bismut and Freed \cite{BiFr}, Branson and Gilkey \cite{BG}, Botvinnik and Gilkey \cite{BoGi}, \cite{BoGi2}, 
Gilkey \cite{G}, Goette \cite{G1, G2, G3}. In case of manifolds with boundary, we mention
contributions by Bunke \cite{Bunke}, Gilkey and Smith \cite{GS}, Lesch and Wojciechowski \cite{LW}, 
M\"uller \cite{Mue} and Melrose \cite{Mel:TAP}. In the setting
of isolated conical singularities we refer the reader to the seminal work by 
Cheeger \cite{Che:SGS, Che-eta} and Lesch \cite{Les:OOF}. We also refer the reader to an 
in-depth overview article by Goette \cite{G4} and the references therein.

\subsection{Atiyah-Patodi-Singer index theorem}\label{APS-intro}
The eta invariant $\eta(D)$ appears as the correction 
term in the Atiyah-Patodi-Singer index theorem 
on manifolds with boundary. More precisely, consider 
an even dimensional compact Riemannian manifold $(X,G)$ with boundary
$(M,g)$ and a linear elliptic first order differential operator 
$\mathbb{D}$ acting between the sections of two Hermitian vector
bundles $E$ and $F$. Assume that over the collar 
$M\times [0,\varepsilon)$ of the boundary, 
$\mathbb{D}$ takes a special form 
\begin{align}\label{boundary-intro}
\mathbb{D} = \sigma \left( \frac{\partial}{\partial u} + D \right),
\end{align}
where $u\in [0,\varepsilon)$ is the inward normal 
coordinate and $\sigma$ is a bundle
isomorphism $E\restriction M \to F \restriction M$. 
The tangential operator $D$ is a self-adjoint 
operator acting on sections of $E\restriction M$. 
Impose Atiyah-Patodi-Singer
boundary conditions at $M$, defined in terms of 
the positive spectral projection of $D$. 
Then $\mathbb{D}$ is Fredholm  and for its index, 
$\textup{index} \, \mathbb{D} = \dim \ker \, \mathbb{D} 
- \dim \ker \, \mathbb{D}^*$,  the following remarkable formula holds:
\begin{align}\label{APS}
\textup{index} \, \mathbb{D} = \int_X a_0(p) \, \textup{dvol}_G(p) - \frac{\dim \ker D+ \eta(D)}{2}.
\end{align}
Here $a_0(p)$ is obtained as the constant term in the short time asymptotic 
expansion of the pointwise trace of $\exp (-t\mathbb{D}^*\mathbb{D}) - \exp (-t\mathbb{D}\mathbb{D}^*)$ at $p\in X$; it
is the same coefficient that would appear in the  boundaryless case. If $\mathbb{D}$ is the positive part of an odd
Dirac-type operator acting between the sections of a $\mathbb{Z}_2$-graded Clifford module, then $\alpha_0$ can be explicitly written 
down in terms of the curvature tensor of $M$ and of $E$. For example, in case  $\mathbb{D}$ is
the spin Dirac operator acting between the sections of the positive and the negative spinor bundles, then  $a_0$ can be  identified
as the $\widehat{A}$-polynomial applied to the curvature  tensor of $X$.

\subsection{Atiyah-Patodi-Singer index theorem for Galois coverings}
The Atiyah-Patodi-Singer index theorem has a non-compact analogue on infinite Galois coverings, 
a result due to Ramachandran \cite{Ram};
this result  is a generalisation to manifolds with boundary of the seminal result by Atiyah \cite{Atiyah-VN}.
Given a Galois covering $\widetilde{X}$
of the manifold $X$ with Galois group $\Gamma$, one may lift the Dirac 
operator $\mathbb{D}$ on $X$ to a $\Gamma$-invariant operator $\widetilde{\mathbb{D}}$ 
on the the covering and show that there is a well-defined 
$\Gamma$-index $\textup{index}_\Gamma$ for a 
suitable generalization  of the Atiyah-Patodi-Singer boundary conditions. The boundary operator of 
$\widetilde{\mathbb{D}}$ 
is denoted by  $\widetilde{D}$; it is a $\Gamma$-invariant  operator on $\partial\widetilde{X}$ descending to $D$ on $\partial X$.
Ramachandran proves the existence of a  $\Gamma$-eta invariant $\eta_\Gamma (\widetilde{D})$
and establishes the index formula
\begin{align}\label{APS-Galois}
\textup{index}_\Gamma \, \widetilde{\mathbb{D}} = \int_X a_0(p) \, \textup{dvol}_G(p) - 
\frac{\dim_\Gamma \ker \widetilde{D} + \eta_\Gamma(\widetilde{D})}{2},
\end{align}
where $a_0$ is again the usual coefficient from above.  Both the $\Gamma$-index, the $\Gamma$-dimension
of $\ker \widetilde{D}$ and the $\Gamma$-eta invariant make use of a trace on a suitable Von Neumann algebra.

\subsection{Cheeger-Gromov and Atiyah-Patodi-Singer rho-invariants}
While the eta invariant has a prominent role as the correction term 
in the index formula on manifolds with boundary, it can in fact be used
to study geometric questions on closed manifolds independently. 
Particularly interesting are the Cheeger-Gromov and the Atiyah-Patodi-Singer
rho invariants, defined in terms of eta invariants in the following way. 
The Cheeger-Gromov rho invariant is defined by
\begin{align}\label{CG-rho}
\rho_\Gamma(D) := \eta_\Gamma(\widetilde{D}) - \eta(D).
\end{align}
The Atiyah-Patodi-Singer rho invariant is associated to a pair of 
 unitary representations of the fundamental group $\pi_1(M)$ of the same dimension: $\alpha, \beta: \pi_1 (M) \to U(\ell)$.
Consider the flat vector bundles 
$$E_\alpha= \widetilde{M} \times_\alpha\bbC^\ell \,,\quad E_\beta= \widetilde{M} \times_\beta \bbC^\ell$$
We can  define twisted Dirac operators $D_\alpha$ and $D_\beta$, respectively.
One then defines the Atiyah-Patodi-Singer rho invariant associated to $(D,\alpha,\beta)$ by
\begin{align}\label{APS-rho}
\rho_{\alpha-\beta}(D) := \eta(D_\alpha) - \eta(D_\beta).
\end{align}

Rho invariants have very interesting applications. For example, the spin Dirac operators can be used in order to distinguish bordism 
classes of metrics of metrics of positive scalar curvature; this is work of Botvinnik and Gilkey \cite{BoGi2}  and, more
generally,  Piazza-Schick \cite{PS1}. Similarly,  rho invariants for the signature operator are employed in order to distinguish homotopy equivalent 
manifolds that are non-diffeomorphic; this is work of Chang-Weinberger \cite{ChWe}. \medskip

All these contributions assume that there is some torsion in the fundamental group of the manifold.
If, on the other hand, the fundamental group of the manifold is  torsion free and the (maximal) 
Baum-Connes map is an isomorphism, then these secondary invariants behave like primary invariants, i.e.
they are zero for the spin Dirac operator associated to a positive scalar curvature metric and they are homotopy invariant for the signature
operator. This is a theorem of Keswani \cite{Kes1, Kes2}, reproved by Piazza-Schick using bordism  techniques, see \cite{PS2}, and 
reproved lately by Higson-Roe \cite{HiRo} for the APS-rho invariant and by Benameur-Roy \cite{BeRo} for the Cheeger-Gromov rho invariant, 
using coarse index theory techniques. 

\subsection{Statement of the main results}
In this paper we study the classical statements from above in case
where the compact manifold $M$ admits an incomplete edge singularity. 
Index theory has been extended from manifolds with 
boundary to spaces with conical singularities by Cheeger e.g. in \cite{Che:SGS}, see also
Br\"uning and Seeley \cite{BruSee:ITF},
Lesch \cite{Les:OOF}, Lesch and Wojciechowski \cite{LW}, Fedosov-Schulze-Tarkhanov
\cite{FST} and Chou \cite{Chou}, to name a few. Singular analysis has been employed by Bismut and Cheeger
\cite{BiCh} in their families index theorem on manifolds with boundary, where they
assumed invertibility of the boundary Dirac operators. Let us also mention work
by Br\"uning \cite{Bru} for the signature operator on simple edge spaces of Witt type
and the work by Albin and Gell-Redman \cite{Albin-Jesse} for the spin Dirac operator
on simple edge spaces satisfying a geometric Witt condition (see also the recent
contribution  \cite{AGR17}); in these articles an explicit index formula is proved. \medskip

Hodge theory on singular spaces has been developed by Cheeger 
in his seminal papers e.g. \cite{Che:SGS} and by Lesch \cite{Les:OOF},
Hunzicker-Mazzeo \cite{HM}, Bei \cite{Bei},
Albin, Leichtnam, Mazzeo and the first named author in \cite{signature-package},
as well as in \cite{Cheeger-spaces}.

\medskip
A simple incomplete edge manifold is a smoothly stratified space of 
depth equal to one. Such a space admits a resolution 
given by a compact manifold $M_c$ with a boundary $\partial M_c$
which fibers $\phi:\partial M_c \to B$ over a closed manifold $B$, the edge singularity, and fibre $F$, 
a closed manifold as well. A precise definition of smoothly stratified spaces of any depth is given in 
\cite{BHS}, cf. also \cite{Novikov}. Given a Riemannian metric $g^B$ on the edge $B$, and a symmetric $2$-tensor 
$\kappa$ on $\partial M$ that restricts to a smooth family of Riemannian metrics on the fibres, 
the singular edge structure in an 
neighborhood $\U$ of the boundary is given by the Riemannian metric ($x$ denotes the defining 
function of the boundary)
$$
g \restriction \U = dx^2 + x^2 \kappa + \phi^* g^B + h,
$$
where $h$ is a lower order perturbation defined in Definitions \ref{d-edge} and 
\ref{def-admissible}. \medskip

In this paper we shall mainly concentrate on the depth-1 case, thus requiring 
 the links $F$ to be compact and smooth. However, for later use, we shall also allow special depth-2
 stratified spaces, with the most singular stratum equal to a discrete collection of points 
 $P$ and associated link equal to a depth-1 stratified space;
we shall also require the metric in a neighbourhood $\U$ of  $P$ to be exact, 
i.e. of the form $g \restriction \U = dx^2 + x^2 \kappa$.
In this specific situation, our analysis actually allows for the link $F$ to 
be any compact smoothly stratified space, as long as self-adjoint
extensions of certain geometric operators on $F$ admit discrete 
spectrum of finite multiplicitiy. \medskip

Our first main result discusses existence of eta invariants for the signature and spin Dirac 
operators under certain additional assumptions on the metric $g$ and under 
a geometric Witt condition which will be made explicit below, see
Assumption \ref{Witt}.

\begin{prop} Let $(M,g)$ be an incomplete edge space with an admissible
edge metric.  Let $D$ be either the signature or the spin Dirac operator, satisfying the geometric 
Witt condition. Then
\begin{enumerate}
\item[(i)]  if $m$ is even, the eta invariant $\eta(D)$ is identically zero,
\item[(ii)] if $m$ is odd, the eta function $\eta(D,s)$ has at most a first order pole singularity at $s=0$.
Then the residue at $s=0$ is local over the edge $B$ in the sense it is given by an integral over the base $B$ 
of a function which is local along $B$ and depends globally
on the fibres of the edge fibration $\phi: Y \to B$\footnote{We shall 
be more precise below in Theorem \ref{trace-coefficients-flat}.}.
\end{enumerate}
In case $M$ is odd dimensional with even dimensional edges, we can say even more.
\begin{enumerate}
\item[(i)] if the scalar curvature on $M$ is positive and bounded uniformly away from zero, 
the eta-invariant $\eta(D)$ of the spin Dirac operator $D$ is well-defined,
\item[(ii)] the eta invariant $\eta(D)$ of the signature operator $D$ is well-defined.
\end{enumerate}
Corresponding statements hold also for the $\Gamma$-eta
function $\eta_\Gamma(\widetilde{D},s)$ and the $\Gamma$-eta invariant $\eta_\Gamma(\widetilde{D})$ 
on Galois coverings of admissible edge spaces. 
\end{prop}

These results will be established in Propositions \ref{eta-main2} and \ref{eta-main-Galois2}, 
and in Corollary \ref{eta-exists-geometric}. \medskip

Our arguments extend to any formally self-adjoint perturbation of one of the 
three geometric Dirac operators, namely the Gauss-Bonnet, the signature 
and the spin Dirac operators, which is lower order in a suitable sense which will be made
precise later. We call such Dirac-type operators $D$ \emph{allowable}
(while we keep the word \emph{geometric} for the unperturbed operators \footnote{In this paper we shall care
to distinguish the signature operator in even and odd dimensions and for this reason we shall
often refer to them as the signature operator (in even dimensions) and the odd-signature operator
(in odd dimensions).}). Moreover we require $D$ to be even or odd with 
respect to some even subcalculus of the heat calculus (see Section \ref{even-odd-section} 
for the relevant notions). Eta-functions and eta-invariants for such allowable perturbations is discussed in 
our next main result.

\begin{thm}\label{main2}
Assume $(M,g)$ is an incomplete simple edge space with an admissible edge metric.
Let $D$ be an allowable Dirac-type operator satisfying the geometric Witt condition; 
let $m$ denote the dimension of $M$ and $b$ the dimension of the edge $B$.
Then we have: 
\begin{enumerate}
\item[(i)] Assume that $m$ is even and $(D,b)$ are of same parity, i.e. 
either $D$ and $b$ are both even, or $D$ and $b$ are both odd.
Then the eta invariant $\eta(D)$ is well-defined. \medskip

\item[(ii)] Assume that $m$ is odd, and $(D,b)$ are of same parity.
Then $\eta(D,s)$ admits a simple pole at $s=0$ with the residue coming from the interior only.
More precisely, the residue is an integral over $M$ of a function $f(p)$ which depends\footnote{
The dependence on the full symbol is functorial in the following sense: the full symbol is not a coordinate invariant
expression, however $f(p)$ is independent of a particular choice of coordinates, see e.g. \cite[Lemma 1.8.2]{Gilkey}.} 
pointwise on a finite number of jets of the full symbol of $D$ at a given point $p \in M$.  \medskip

\item[(iii)] Assume that $m$ is even and $(D,b)$ are of opposite parity, i.e. 
either $D$ is even and $b$ is odd, or $D$ is odd and $b$ is even.
Then $\eta(D,s)$ admits a simple pole at $s=0$ with the residue given by an 
interior term as in \textup{(ii)} and a boundary term 
coming from the edge. The contribution from the edge is an integral over $B$ of a function which is global
in the fibres of the edge fibration $\phi: Y \to B$ and local along the base $B$\footnote{We shall 
be more precise below in Theorem \ref{trace-coefficients-flat}.}.\medskip

\item[(iv)] Assume that $m$ is odd and $(D,b)$ are of opposite parity. Then the eta function $\eta(D,s)$ 
may admit a second order pole singularity at $s=0$. The Laurent coefficient of $s^{-2}$ 
in the Laurent expansion of $\eta(D,s)$ at $s=0$ is an interior term as in $\textup{(iii)}$. 
The residue is of the same structure as the residue in $\textup{(iii)}$.\medskip

\item[(v)] Assume $M$ is boundary of some admissible edge space $X$ with 
an even or odd Dirac operator $\mathbb{D}$, satisfying the geometric Witt condition and such that $D$ is the  
tangential operator of $\mathbb{D}$ as in \eqref{boundary-intro}. 
Moreover, assume that for the dimension $b$ of each edge in $X$, 
at least one of the numbers $(m+1-b)$ and $b$ is odd.
Then the eta-invariant $\eta(D)$ is well defined. 
\end{enumerate}

The case where the edge $B = \bigcup\limits_{i=1}^k B_i$ is a union of connected components $B_i$ of dimension $b_i$,
is dealt with a combination of the cases above. \medskip

Corresponding statements hold also for the $\Gamma$-eta
function $\eta_\Gamma(D,s)$ and the $\Gamma$-eta invariant $\eta_\Gamma(D)$ 
on Galois coverings of admissible edge spaces. 

\end{thm}

This theorem gathers together the results that will be established in Propositions \ref{eta-main} 
and \ref{Galois-main-eta}, as well as in Theorems \ref{eta-exists} and \ref{eta-exists-Galois}. \medskip

The eta-invariants of the geometric operators are trivially zero 
in case of dimension $m$ being even, due to symmetry of the spectrum.
However, when discussing general \emph{allowable} Dirac-type operators, the
geometric examples may well be perturbed in such a way to destroy the spectral symmetry.
Similarly, the eta-invariant of the Gauss-Bonnet operator is zero both
in even and in odd dimension; however, this is no longer true if we pass to an allowable
perturbation of the geometric operator.
\medskip

A remark on Theorem \ref{main2} (v) is in order. 
Let $D$ be equal to the signature operator and let $M$ be odd dimensional
with even dimensional edge singularities. 
We can always construct a bounding stratified pseudomanifold  $X$
by setting $X$ to be a finite cone over $M$: $X=\mathscr{C}(M)$. 
This defines an iterated cone edge space which may still be covered by the analytic
arguments given here; in particular the extension of Theorem \ref{main2} (v) to this case holds.
If $M$ is odd dimensional with even dimensional edges and satisfies the 
geometric Witt condition, the cone over $M$ also satisfies the geometric Witt condition (indeed, the link of the tip of the cone 
is odd dimensional) and so all the assumptions of Theorem \ref{main2} (v) are fulfilled and we can conclude that the eta
invariant of the odd-signature operator on $M$ is well defined.
The precise statement is given in Corollary \ref{eta-exists-geometric}. 
Similarly, if $M$ is spin, odd dimensional with an edge $B$ of even dimension and an edge metric of uniform positive scalar curvature 
then the cone over $M$ satisfies the geometric Witt condition (provided we suitably scale the metric) and so, 
by  Theorem \ref{main2} (v), the eta invariant of the
spin Dirac operator on $M$ is well defined.
These two examples are particularly important for applications and therefore
singled out in Corollary \ref{eta-exists-geometric}. 
\medskip

A central component in our argument is the even/odd heat calculus developed 
by the second author jointly with Mazzeo in \cite{MazVer}. In that reference the authors studied analytic
torsion on simple edge manifolds, which required a detailed analysis of the heat kernel. 
In the present paper we give more details and provide a more direct treatment of the even and odd subcalculi. 
The heat kernel on a manifold with isolated conical singularities has been studied 
by Cheeger in his seminal paper \cite{Che:SGS}; the general Witt case is also treated in \cite{Che:SGS}, 
see the iterative argument given in Section 7.4 there\footnote{Cheeger \cite[\S 7.4]{Che:SGS}, 
studies the heat kernel of the signature operator on  Witt spaces of arbitrary depth. 
In that respect his result is much more general than the Witt spaces of depth one, 
considered here. However, it should be remarked that Cheeger considers 
piecewise flat cone-edge metrics, which are locally 
isometric to products of smooth open subsets with exact
cones. Thus the edge metrics considered here 
are much more general, since we allow for fibrations of not necessarily 
exact cones and moreover do not require the fibrations to be locally 
isometric to products of smooth open subsets and cones.}.
Mooers \cite{Moo} gives a microlocal approach to the study of 
the heat kernel in the isolated case; this is the approach that is 
generalized by Mazzeo and the second author to the simple edge case.
\medskip

Going back to the contents of this paper, our second 
main result extends the Atiyah-Patodi Singer index formula
on compact manifolds with boundary $M$ and on their infinite Galois coverings to
the case of $M$ being an edge manifold as above.

\begin{thm}\label{main2-APS}
Consider an odd-dimensional incomplete edge space $M$ with an admissible edge metric $g$.
Let $D$ be an allowable Dirac-type operator, satisfying the geometric Witt condition.
Assume $M$ is boundary of some even-dimensional incomplete edge space $X$ with an admissible edge metric
 and an allowable Dirac-type operator $\mathbb{D}$ of the form 
\eqref{boundary-intro} near the boundary, satisfying the geometric Witt condition.
We denote by $\phi: Y \to B$ the fibration of links over the singular stratum of $X$. 
Assume that $\mathbb{D}^2$ is even and that the dimension $b$ of each edge singularity in $X$ is odd.
Then the Dirac operator $\mathbb{D}$ with APS boundary conditions is Fredholm, 
the eta invariant of $D$ is well-defined and is related to the Fredholm index of $\mathbb{D}$ by the 
index formula
\begin{align}\label{two-integrals}
\textup{index} \, \mathbb{D} = \left(\int_X a_0 + \int_{B} b_0 \right)- \frac{\dim \ker D+ \eta(D)}{2},
\end{align}
where we have the following characterization of the integrands $a_0$ and $b_0$.
\begin{enumerate}
\item The integrand $a_0$ is in fact the same as in the classical formula \eqref{APS}.\medskip

\item The integrand $b_0$ comes from the edge singularity 
in the sense that at each $p\in B$ the value $b_0(p)$ is global
in the fibres of the edge fibration $\phi: Y \to B$ and local along the base $B$\footnote{We shall 
be more precise below in Theorem \ref{trace-coefficients-flat}.}.
\medskip

\item If $\mathbb{D}$ is an allowable perturbation of a geometric Dirac operator 
twisted by a flat vector bundle $E$, then the coefficients $a_0$ and $b_0$ 
depend additionally only on the rank of $E$. 
\end{enumerate}

The corresponding statement carries over to the setting of Galois coverings, where 
index of $\mathbb{D}$, dimension of $\ker D$ and the eta-invariant $\eta(D)$ are 
replaced by their corresponding Galois covering versions.
\end{thm}

\begin{remark}
We point out that the presence of an integral along the edge singularity in the index 
formulae \eqref{two-integrals} associated to an edge manifold is a general phenomenon and in fact appears 
already in the signature index formula of Br\"uning \cite{Bru} and in the spin index formula
of Albin and Gell-Redman \cite{Albin-Jesse}. 
\end{remark}

Next we prove the important result that the rho invariants of Atiyah-Patodi-Singer and of Cheeger-Gromov
are well defined in the singular setting.

\begin{thm}\label{main1}
Assume $(M,g)$ is an incomplete simple edge space with an admissible edge metric.
Let $D$ be an allowable Dirac-type operator satisfying the geometric Witt condition.
Then the APS and the Cheeger-Gromov rho invariants 
are well-defined.
\end{thm}

We conclude the paper with the following stability results.

\begin{thm}\label{main3}
Assume $(M,g)$ is an odd-dimensional incomplete simple edge space with an admissible edge metric.
The following holds:
\begin{enumerate}
\item the APS and Cheeger-Gromov rho invariants of the spin Dirac operator associated to an edge
metric of positive scalar curvarure are bordism invariant\footnote{the bordisms are assumed to have edges of odd dimension.}
of  positive scalar curvature edge metrics;
\item the APS and Cheeger-Gromov rho invariants for the signature operator are invariant under variations of $g$ among 
admissible edge metrics satisfying the Witt condition. Consequently, the rho invariant for the signature operator is 
a stratified diffeomorphism invariant. 
\end{enumerate}
\end{thm}

\emph{Acknowledgements.} 
We thank Pierre Albin, Marcus Banagl, Matthias Lesch, Rafe Mazzeo and Jonathan Woolf for
valuable discussions. We are grateful to the anonymous referee for  a careful reading 
of the original manuscript and  for useful 
suggestions.
We thank Sapienza Universit\`a di Roma and M\"unster University for hospitality and financial
support. 

\section{Review of geometry on incomplete edge spaces}\label{geometry-section}

Consider a compact smoothly stratified space $\overline{M}$ of depth 1 and dimension $m$: 
we shall assume that $\overline{M}$ is the disjoint union of the top stratum $M$, a smooth manifold 
of dimension $m$, and finitely many lower dimensional strata $\{B_i\}, i\in I$, where each $B_i$ is a closed compact manifold
of dimension $b_i$. For notational simplicity we continue with the case of a single stratum $B$ of dimension $b$, the
general case is studied in an analogous way. The stratification hypothesis asserts the existence of an open 
neighbourhood $U\subset \overline{M}$ around the singular stratum $B$,
together with a distance function $x:U \to [0,1)$, such that $U\cap M$ 
is the total space of a smooth fibre bundle over $B$ with the fibre 
given by a truncated cone $\mathscr{C}(F)=(0,1)\times F$ 
over a compact smooth manifold $F$ of dimension $f\geq 1$.
The distance function $x$ restricts to a radial function of that cone on each fibre.
We shall refer to such a stratified space as a {\it simple edge space},
or, shortly, as an edge space.
\medskip

The stratified space $\overline{M}$ can be resolved to 
define a compact manifold $M_c$, with boundary $\partial M_c$
being the total space of a fibration $\phi: \partial M_c \to B$ with fibre $F$. 
Under the resolution, the neighborhood $U$ lifts to a collar neighborhood 
$\U \subset M_c$, which is a smooth fibration of cylinders 
$[0,1)\times F$ over $B$ with radial function $x$.
The open regular  stratum $M$ is identified with $M_c \backslash \partial M_c$.\medskip

\medskip
\noindent
{\bf Notation:} with a small abuse of notation we shall denote $\partial M_c$ by $\partial M$.

\medskip

An edge structure on $M$ is defined by a particular choice 
of a Riemannian metric.

\begin{defn}\label{d-edge}
An incomplete Riemannian simple edge space is a depth-one smoothly stratified space  $\overline{M}$ together
with a Riemannian metric $g$ on $M=M_c \backslash \partial M$ such that  over 
$\U\backslash \partial M$ the metric attains the form $g_0+h$ with
$$g_0 =dx^2+x^2 \kappa +\phi^*g^B,$$
where $g^B$ is a Riemannian metric on the closed manifold $B$, 
$\kappa$ is a symmetric 2-tensor on the fibration $\partial M$ restricting to a
smooth family of Riemannian metrics on fibres $F$, $|h|_{g_0}$ is smooth on $\U$ and 
vanishes at $x=0$. 
\end{defn}

Notice that such a Riemannian metric $g$ is incomplete. For this reason we
shall refer to $(M,g)$ as an incomplete simple edge space. On the other hand, the metric $x^{-2} g$ is
a complete Riemannian metric and $(M,x^{-2}g)$ is referred to as a complete edge space.
\medskip

We extend the class of edge singular manifolds slightly in \S \ref{iterated-section}. \medskip

Our analysis requires  $\phi: (\partial M, \kappa + \phi^*g^B) 
\to (B, g^B)$ to be a Riemannian submersion in the following sense. 
If $p\in \partial M$, then the tangent bundle $T_p\partial M$ splits into vertical and horizontal subspaces  
$T^V_p \partial M \oplus T^H_p \partial M$; the vertical subspace $T^V_p\partial M$ is the tangent space to the fibre of 
$\phi$ through $p$, and the horizontal subspace $T^H_p \partial M$ is the annihilator of the subbundle 
$T^V_p\partial M \lrcorner \kappa \subset T^*\partial M$ ($\lrcorner$ denotes contraction).  
$\phi$ is  a Riemannian submersion if the tensor $\kappa$ restricted to $T^H_p \partial M$ vanishes. 
\medskip

Finally, we require the tensor $h$ to be even, in the sense that $h$ admits an asymptotic
expansion as $x\to 0$ containing only even powers of $x$, up to $dx$-cross terms which 
are required to admits only odd powers of $x$ in their expansion. This condition was introduced
in \cite{MazVer} in order to achieve that the heat kernel of the corresponding Hodge 
Laplacian lies in a distinguished \emph{even} subcalculus. In fact similar evenness conditions 
have been employed in various geometric settings, cf. for example \cite{Scott}, \cite{Albin} and \cite{AR}.\medskip

Evenness of $h$ is well-defined only within a particular \emph{even} equivalence class of coordinate charts near the edge. 
A local coordinate system $(\wx, \wy, \wz)$ in $\U$ is said to be in the even equivalence class of a coordinate chart $(x,y,z)$
if the asymptotics of $\wx/x$, $\wy$ and $\wz$ near the edge admits only powers of $x^2$, with coefficients in the 
expansions depending smoothly on $y$ and $z$. We refer to coordinates within a fixed even equivalence class as \emph{special}
coordinates and will stay within the special coordinates henceforth.
\medskip

We summarize these conditions into the notion of \emph{admissible} edge metrics.

\begin{defn}\label{def-admissible}
Let $(M,g)$ be an incomplete Riemannian manifold with an edge. The edge metric $g=g_0+h$ is said to be 
admissible if $\phi: (\partial M, \kappa + \phi^*g^B) \to (B, g^B)$ is a Riemannian submersion 
with fibres $F$, and $h$ is even within a fixed 
choice of special coordinates in the sense above.
\end{defn} 

The condition on $\phi$ to be a Riemannian submersion
can be motivated here. \medskip

Pick local coordinates 
$y=(y_1,...,y_{b})$ on $B$ lifted to $\partial M$ and then extended inwards to $\U$. 
Let $z=(z_1,...,z_f)$ restrict to local coordinates on $F$ along each fibre of the boundary. 
Such a choice of $(x,y,z)$ defines local coordinates on $\U \cap M$. 
Consider the Hodge Laplace operator $\Delta$ on $(M,g)$
and define for any $y_0 \in B$ the normal operator 
$N(x^2\Delta)_{y_0}$ as the limit of $x^2\Delta$ with respect to conjugation by the local family of dilations 
$(x,y,z) \to (\lambda x, \lambda (y-y_0), z)$ as $\lambda \to \infty$. \medskip

If $\phi$ is a Riemannian submersion, then
$N(x^2\Delta_p)_{y_0}$ is naturally identified with $s^2$ times the Hodge Laplacian on the 
model edge $\R^+_s \times F \times \R^b$ with the model edge metric $g_{\textup{ie}} = 
ds^2 + s^2 \kappa_{y_0} + g^B_{y_0}$,
where we identified $T_{y_0}B=\R^b$. Similar decomposition holds for the spin Dirac and
the spin Laplace operators as worked out in \cite[Lemma 2.2]{Albin-Jesse}.
This structure of the normal operator is central to the microlocal construction of 
the heat kernel in \cite{MazVer}, as it allows a decomposition of the initial heat parametrix
into an explicit conical and euclidean part.

\subsection{Incomplete edge differential operators}

A  central element in our approach to singular edge spaces is the notion of edge vector fields.
Define the space of edge vector fields $\mathcal{V}_e$ to be the space of vector fields smooth 
in the interior of $M_c$ and tangent at the boundary $\partial M_c$ to the fibres of the fibration. 
$\mathcal{V}_e$ is closed under the ordinary Lie bracket of vector fields, and hence defines a Lie algebra. 
In local coordinates, the complete edge vector fields $\mathcal{V}_e$ are locally generated by 
\[
\left\{x\frac{\partial}{\partial x}, x\frac{\partial}{\partial y_1}, \dots, x \frac{\partial}{\partial y_b}, 
\frac{\partial}{\partial z_1},\dots, \frac{\partial}{\partial z_f}\right\}.
\]
These vector fields are of bounded length with respect to the complete edge metric 
$x^2 g$, where the radial function $x: \mathscr{U} \to (0,1)$ is extended smoothly to a nowhere
vanishing function in the interior of $M$.\medskip

By the Serre-Swan theorem there exists a finite rank vector bundle ${}^eTM$
such that the edge vector fields $\mathcal{V}_e$ 
form a spanning set of sections, i.e. $\mathcal{V}_e=\Gamma({}^eTM)$. 
We call ${}^eTM$ the complete edge tangent bundle.
Sections of the dual bundle ${}^eT^*M$, called the complete edge cotangent bundle, are spanned locally the one-forms
\begin{align}\label{triv}
\left\{\frac{dx}{x}, \frac{dy_1}{x}, \dots, \frac{dy_b}{x}, dz_1,\dots,dz_f\right\}.
\end{align}
Though singular in the usual sense, these one-forms are smooth as sections of ${}^eT^*M$
and are of bounded norm with respect to the complete edge metric $x^2g$ on differential forms.
We define the incomplete edge cotangent space ${}^{ie}T^*M$ 
by setting $\Gamma({}^{ie}T^*M) = x \Gamma({}^{e}T^*M)$, which amounts
to ${}^{ie}T^*M$ being spanned locally by 
\[
\left\{dx, dy_1, \dots, dy_b, x dz_1,\dots, x dz_f\right\}.
\]
We denote $p$-th order exterior power of ${}^{ie}T^*M$ by ${}^{ie}\Lambda^p M$. 
Smooth sections of the incomplete edge cotangent bundle, as well as its exterior powers
are of bounded norm with respect to the incomplete edge metric $g$ on differential forms.
\medskip

Let $E$ denote any flat Hermitian vector bundle over the edge manifold $\overline{M}$, 
defined by a unitary representation $\rho:\pi_1(\overline{M}) \to U(\ell)$. Such vector bundles
are defined on all of $\overline{M}$ and not only on the regular part of $M$. \medskip 

If $M$ is spin, we denote by $S$ the corresponding spinor bundle.
We define the spaces of edge differential operators
$\textup{Diff}_e^k(M,{}^{ie}\Lambda^*T^*M \otimes E)$
and $\textup{Diff}_e^k(M, S\otimes E)$ as spaces of 
differential operators over $M$ acting on compactly supported sections 
$\Gamma_0({}^{ie}\Lambda^*T^*M \otimes E)$ and $\Gamma_0(S\otimes E)$, 
respectively, and given locally over the singular neighborhood 
by a sum of products of elements of edge vector fields $\mathcal{V}_e$ 
with smooth (on $M_c$) matrix-valued coefficients, with respect to suitable local trivializations
of the vector bundles. Thus any $L\in \textup{Diff}_e^k$ is of the local form
\begin{equation}
L=\sum_{j+|\A|+|\beta|\leq k} a_{j,\A,\beta}(x,y,z)(x\partial_x)^j(x\partial_y)^{\A}\partial_z^{\beta}, 
\end{equation}
with each $a_{j,\alpha,\beta}$ matrix-valued and smooth up to $x=0$. 
As explained in the foundational work of Mazzeo, see \cite{Maz:ETO}, the analytic
properties of an edge differential operator are governed by the indicial operator,
a family of operators on $\RR^+ \times F_y$, $y\in B$ and the normal operator,
a family of operators on $\RR^+\times \RR^{\dim B} \times F_y$, $y\in B$.
The Mellin transform of the indicial operator computed at $z=0$ defines the vertical family
associated to $L$.
A differential operator $P$ of order $k$ is an \emph{incomplete} edge operator if $x^k P\in \textup{Diff}_e^k$,
i.e. if $P \in x^{-k} \textup{Diff}_e^k$. Moreover, we denote by $L^2$ the $L^2$-closure of compactly supported
sections of the respective vector bundles and define the maximal domain of $P$ by
\begin{align}
\dom_{\max}(P) := \{u \in L^2 \mid Pu \in L^2 \},
\end{align}
where $Pu \in L^2$ is understood in the distributional sense. 
The minimal domain $\mathscr{D}_{\min}(D)$ is defined as a subspace of 
$\mathscr{D}_{\max}(D)$, given by the graph closure of $\Gamma_0$, where the
graph norm is given in terms of the $L^2$ norm $\| \cdot \|_{L^2}$ by 
$\| u \|:= \| u \|_{L^2} + \| P u \|_{L^2}$. \medskip

We say that $P$ is \emph{locally independent of the twisting vector bundle}
$E$ up to conjugation if the following is true, cf. \S \ref{locality-section} for a more detailed
account. Consider another flat vector bundle $E'$  of the same rank, defined by a representation 
of the fundamental group $\pi_1(\overline{M})$. Fix after an 
eventual refinement a trivializing atlas for both $E$ and $E'$, for which the transition functions are locally constant.
Near the edge singularity we can take such a trivializing atlas made of distinguished neighbourhoods of the form 
$\mathcal{V}\times \overline{\mathscr{C}(F)}$ where $\mathcal{V} \subset B$ is a contractible open neighborhood.
Each such neighborhood comes with a local bundle isomorphism 
 \begin{align}
 \theta: E \restriction \mathcal{V}\times \overline{\mathscr{C}(F)} \to 
 E' \restriction \mathcal{V}\times \overline{\mathscr{C}(F)}.
\end{align}
Abusing notation, we denote the corresponding map between local sections of $E_1$ and $E_2$
by $\phi$ again.
We say that $P$ is locally independent of the twisting flat vector bundle up to conjugation, 
if $P \equiv P_E$ is associated by some geometric procedure to a flat vector bundle $E$ 
over $\overline{M}$ of fixed rank, such that for any two such vector bundles $E$ and $E'$
\begin{align}
P_E \restriction \mathcal{V}\times \overline{\mathscr{C}(F)} 
= \theta^{-1} \circ P_{E'} \circ \theta \restriction \mathcal{V}\times \overline{\mathscr{C}(F)}. 
\end{align}
A similar definition can be given for an operator 
$P \in \textup{Diff}_e^k$. Clearly if $P_E$ is of the form $Q\otimes \operatorname{Id}_E$ for some 
differential operator in $\textup{Diff}_e^k(M,{}^{ie}\Lambda^*T^*M)$ or $\textup{Diff}_e^k(M, S)$, 
then $P$ is independent of the twisting flat vector bundle. \medskip

Notice that the canonical flat connections of $E$ and $E'$ restricted to $M$ 
will have trivial connection one-forms in the fixed atlas (intersected with M). Using the vanishing of the 
connection one-forms in the trivializing atlas we have chosen, we conclude
that the three geometric Dirac operators, i.e. the signature, the Gauss Bonnet
and the spin Dirac operators, associated to an \emph{incomplete} edge metric and twisted by a flat unitary
connection on the flat vector bundle $E$, are all examples of incomplete edge operators of order $1$
which are locally independent of the twisting flat vector bundle up to conjugation.

\begin{defn}\label{allowable}
We call an elliptic first order formally self-adjoint incomplete edge differential operator $D \in x^{-1}\textup{Diff}_e^1$
{\bf allowable Dirac-type}, if $D=D^E_{{\rm geo}} + P$, with $D^E_{\rm geo} \in x^{-1}\textup{Diff}_e^1$ being one of the three
geometric Dirac operators twisted by a flat vector bundle $E$ and $P\in \textup{Diff}_e^1$ 
locally independent of the twisting flat vector bundle up to conjugation.\footnote{We may in principle extend the
class of possible perturbations $P$ to include certain pseudo differential operators, subject to the condition 
that we can construct the heat kernel of $D^2$ microlocally.}
\end{defn}

We shall impose the geometric Witt condition throughout the paper and we now proceed to carefully define it: 
we associate to each $D$ the vertical family $D_F \equiv \{D_{F}(y)\}$ 
of Dirac operators acting on the links $(F_y,\kappa_y)$ at each point $y\in B$ of the edge fibration $\phi:Y\to B$.
Since $F$ is a closed compact manifold, each $D_F(y)$ is essentially
self-adjoint and admits a discrete spectrum $\textup{Spec}(D_F(y))$ which is allowed to vary with 
$y \in B$. We write
\begin{equation}
\textup{Spec}(D_F) := \bigcup_{y\in B} \textup{Spec}(D_F(y)).
\end{equation}

\begin{Assumption}\label{Witt} (Geometric Witt condition)
If $D$ is an allowable perturbation of the spin Dirac operator, then we assume that 
\footnote{For the spin Dirac operator our notion of  geometric Witt condition 
comes from \cite{Albin-Jesse}.}
\begin{equation}\label{witt-spin}
\textup{Spec}(D_F) \cap \left(-\frac{1}{2}, \frac{1}{2}\right) = \varnothing.
\end{equation}
If $D$ is an allowable perturbation of the Gauss-Bonnet operator or the signature operator, then
we require that the 
\begin{equation}\label{witt-sign}
\textup{Spec}(D_F) \cap \left(-1, 1\right)  =\{0\}\,,\quad H^{k}(F, E)=0 \;\;\text{if}\;\;\dim F=2k .
\end{equation}
\end{Assumption}

\medskip
In case of the Gauss-Bonnet and the signature operators, as well as in the case of the spin Dirac operator 
satisfying that the scalar curvature is positive and uniformly bounded away from zero, the geometric
Witt condition for $D$ can be obtained assuming the geometric Witt condition for the untwisted geometric
Dirac operator $D_{\rm geo}$ and using a suitable scaling. See Subsection \ref{section-eta-bdry} for details.
\medskip

A comment about essential self-adjointness of $D$.  In case of the 
spin Dirac operator acting on sections of the spinor bundle $S$, essential self-adjointness
of $D$ as an unbounded operator in the space of square integrable sections 
$L^2(M,S)$ with core domain $\Gamma_c (M,S)$
is a consequence \footnote{In fact, in the setting of isolated conical singularities the geometric
Witt condition is equivalent to essential self-adjointness of the spin Dirac operator, see for instance 
\cite[Theorem 3.2]{Chou}} of the geometric Witt condition \eqref{witt-spin}, see  
Albin and Gell-Redman \cite[Theorem 1.1]{Albin-Jesse}.
Building on the work of Chou, Albin and Gell-Redman also prove in \cite[Section 7]{Albin-Jesse} that 
if the scalar curvature of $M$ is non-negative in an open neighborhood of the edge, 
then the geometric Witt condition \eqref{witt-spin}
is satisfied. Notice that for the argument of Albin and Gell-Redman to work we need to assume,
additionally, that the vertical tangent bundle to the fibration $F\to \partial M\to B$ is spin.
(Equivalently, we can assume that $B$ is spin: see Lawson-Michelson \cite[Ch. II, Section 1, Proposition 1.15]{LM}.) Indeed, the assumption that $M$ is spin only implies 
that $\partial M$ is spin; however, for the Lichnerowicz argument to work we need the spin assumption
on the links and this is not automatic.  
The spin assumption on the vertical tangent bundle of  $F\to \partial M\to B$, or, equivalently, on the singular stratum $B$, must therefore be added to the hypothesis of the Theorem 
stated in \cite[Section 7]{Albin-Jesse}. \medskip

The  geometric Witt condition  \eqref{witt-sign} in the Hodge - de Rham setting also  ensures essential
self-adjointness of the Gauss-Bonnet and the signature operators. This is implicit in the work of Cheeger
and addressed explicitly in the work of Albin, Leichtnam, Mazzeo and Piazza \cite[Proposition 5.11]{signature-package}.
\medskip

By inspection of the proofs in \cite[Theorem 1.1]{Albin-Jesse} and 
\cite[Proposition 5.11]{signature-package} \footnote{indeed, the indicial and normal operator at $y\in B$ do not see the lower order perturbation 
$P$.} one can 
establish the following 

\begin{proposition}\label{ess-sa-allowable}
If $D=D^E_{{\rm geo}}+P$ is an allowable Dirac operator and satisfies the geometric Witt condition, 
then $\dom_{\min}(D) = \dom_{\max}(D)
= \dom_{\min}(D^E_{{\rm geo}}) = \dom_{\max}(D^E_{{\rm geo}})$
and in particular, $D$ is essentially self-adjoint.
\end{proposition}

The geometric Witt condition for the Gauss-Bonnet and for the signature operator 
is obviously stronger than the cohomological Witt condition $H^{f/2}(F)=0$,
in the sense that it also requires absence of \emph{small} eigenvalues.
Similarly, condition \eqref{witt-spin} is obviously stronger  than 
the condition of invertibility of the vertical family,  in that it requires
also the vanishing of small eigenvalues.
If these small eigenvalues are present, the relevant geometric Dirac operators are not essentially self-adjoint; 
in  this case  one can impose algebraic boundary conditions as in  \cite{Ver-Mooers} and still obtain self-adjoint extensions
of our Dirac operators. However, the generalisation of the results of this paper requires
an extension of the microlocal heat kernel construction to these self-adjoint extensions. We leave all this 
to future investigations.

\subsection{A remark on iterated cone-edge singularities}\label{iterated-section}
A priori the links $F$ of an edge manifold are set to be smooth, compact and without boundary.
The underlying reason for that restriction is that our analytic arguments here require
a microlocal description of the heat kernel on such an edge manifold, as constructed in 
\cite{MazVer}. However, the analytic arguments of \cite{MazVer} directly apply 
to more general spaces $M$ (but still  particular cases of general stratified psedomanifolds): namely,  
the bottom stratum  is equal to a point, the link $(F, \kappa)$
is a stratified compact edge manifold with $D_F$ being essentially self-adjoint with discrete spectrum and the Riemannian metric 
in an open neighborhood $\U$ of the conical point is given by 
$g\restriction \U = dx^2 + x^2 \kappa$. Consequently, from here on, we include this case of 
depth two stratified spaces $M$ in our class of (admissible) edge manifolds.
For example, the signature operator on a  finite cone over an odd dimensional
Witt edge space of depth $\ell$ would satisfy this hypothesis.

\section{Microlocal analysis of the heat kernel on edges}\label{microlocal-section} 

\subsection{Heat kernel asymptotics for the Hodge Laplacian} 

In this subsection we consider the Friedrichs self adjoint extension of 
the Hodge Laplace operator on $(M,g)$ twisted with a flat connection on a flat Hermitian vector bundle $(E,\nabla, h)$, 
which we denote by $\Delta$ again, with domain $\dom (\Delta)$. 
The corresponding heat operator acts as an integral convolution operator 
on $u\in \dom (\Delta)$ 
\begin{equation}
e^{-t\Delta} u(p) = \int_M \HF \left( t, p,\widetilde{p} \right)
u(\widetilde{p}) \dv (\widetilde{p}),
\end{equation}
where the heat kernel $\HF$ is a distribution on $M^2_h=\R^+\times M_c^2$, 
taking values in $({}^{ie}\Lambda^* T^*M \otimes E)\boxtimes ({}^{ie}\Lambda^* T^*M \otimes E)^*$.
Equivalently we may view the heat kernel as a section of $({}^{ie}\Lambda^* T^*M \otimes E)\boxtimes ({}^{ie}\Lambda^* T^*M \otimes E)$
with the identification given by the metric. In this case the heat operator acts as follows
\begin{equation} \label{eqn:hk-on-functions}
e^{-t\Delta} u(p) = \int_M \left(\HF \left( t, p,\widetilde{p} \right),  
u(\widetilde{p})\right)_{g,h} \dv (\widetilde{p}).
\end{equation}
Its pointwise action on $u$ is given by taking the pointwise inner product defined by the Riemannian 
edge metric $g$ and the Hermitian bundle metric $h$.
Then $e^{-t\Delta} u$ solves the homogeneous heat problem
\begin{equation*}
(\partial_t + \Delta) \, \w(t,p)  = 0, \ \w(0,p)=u(p),
\end{equation*}
for any $u \in \dom (\Delta)$. We proceed with discussing the asymptotic properties
of $\HF$. Consider the local coordinates near the corner in $M^2_h$ given by $(t, (x,y,z), (\widetilde{x}, \wy, \widetilde{z}))$, 
where $(x,y,z)$ and $(\widetilde{x}, \wy, \widetilde{z})$ are two copies of coordinates on $M$ near the edge. 
The kernel $\HF(t, (x,y,z), (\wx,\wy,\wz))$ has non-uniform behaviour at the submanifolds
\begin{align*}
&A =\{ (t, (x,y,z), (\wx,\wy,\wz))\in M^2_h \mid t=0, \, x=\wx=0, \, y= \wy\}, \\
&D =\{ (t, p, \widetilde{p})\in M^2_h \mid t=0, \, p=\widetilde{p}\},
\end{align*}
which requires an appropriate blowup of the heat space $M^2_h$, 
such that the corresponding heat kernel lifts to a polyhomogeneous distribution 
in the sense of the following definition, which we cite from \cite{Mel:TAP} and \cite{MazVer}.

\begin{defn}\label{phg}
Let $\mathfrak{W}$ be a manifold with corners and $\{(H_i,\rho_i)\}_{i=1}^N$ an enumeration 
of its (embedded) boundary hypersurfaces with the corresponding defining functions. For any multi-index $s= (s_1,
\ldots, s_N)\in \C^N$ we write $\rho^s = \rho_1^{s_1} \ldots \rho_N^{s_N}$.  Denote by 
$\mathcal{V}_b(\mathfrak{W})$ the space of smooth vector fields on $\mathfrak{W}$ which lie
tangent to all boundary faces. A function $\w\in \rho^s L^\infty(\mathfrak{W})$ for 
some $s\in \C^N$ is said to be conormal, if $V_1 \ldots V_\ell \w \in \rho^s L^\infty(\mathfrak{W})$
for all $V_j \in \mathcal{V}_b(\mathfrak{W})$ and for every $\ell \geq 0$. An index set 
$E_i = \{(\gamma,p)\} \subset {\mathbb C} \times {\mathbb N_0}$ 
satisfies, by definition, the following hypotheses:

\begin{enumerate}
\item $\textup{Re}(\gamma)$ accumulates only at $+\infty$,
\item for each $\gamma$ there exists $P_{\gamma}\in \N_0$, such 
that $(\gamma,p)\in E_i$ for all $p < P_\gamma$,
\item if $(\gamma,p) \in E_i$, then $(\gamma+j,p') \in E_i$ for all $j \in {\mathbb N_0}$ and $0 \leq p' \leq p$. 
\end{enumerate}
An index family $E = (E_1, \ldots, E_N)$ is an $N$-tuple of index sets. 
Finally, we say that a conormal distribution $\w$ is polyhomogeneous on $\mathfrak{W}$ 
with index family $E$, we write $\w\in \mathscr{A}_{\textup{phg}}^E(\mathfrak{W})$, 
if $\w$ is conormal and if in addition, near each $H_i$ we have an asymptotic expansion
\[
\w \sim \sum_{(\gamma,p) \in E_i} a_{\gamma,p} \rho_i^{\gamma} (\log \rho_i)^p, \ 
\textup{as} \ \rho_i\to 0,
\]
where we require inductively that the coefficients $a_{\gamma,p}$ are 
conormal on $H_i$, i.e. polyhomogeneous with index $E_j$
at any other boundary face $H_j$ intersecting $H_i$. 
\end{defn}

Blowing up the submanifolds $A$ and $D$ is a geometric procedure of introducing polar coordinates on $M^2_h$, 
around the submanifolds together with the minimal differential structure which turns polar coordinates into smooth
functions on the blowup. A detailed account on the blowup procedure is given e.g. in \cite{Mel:TAP}. See also  \cite{Gr}. 
In this paper, the index sets are always subsets of $\R \times \N_0$, i.e. the exponents 
$\gamma$ are real.\medskip

First we blow up parabolically 
(i.e. we treat $\sqrt{t}$ as a smooth variable) the submanifold $A$.
This defines $[M^2_h, A]$ as the disjoint union of
$M^2_h\backslash A$ with the interior spherical normal bundle of $A$ in $M^2_h$,
equipped with the minimal differential structure 
such that smooth functions in the interior of $M^2_h$ and polar coordinates 
on $M^2_h$ around $A$ are smooth. The interior spherical normal bundle of $A$ defines a new boundary 
hypersurface $-$ the front face ff in addition to the previous boundary faces 
$\{x=0\}, \{\wx=0\}$ and $\{t=0\}$, which lift to rf (the right face), lf (the left face) and 
tf (the temporal face), respectively.  \medskip

The actual heat-space $\mathscr{M}^2_h$ is obtained by a second parabolic blowup of  
$[M^2_h, A]$ along the diagonal $D$, lifted to a submanifold of $[M^2_h, A]$. 
We proceed as before by cutting out the lift of $D$ and replacing it with its spherical 
normal bundle, which introduces a new boundary face $-$ the temporal diagonal td. 
The heat space $\mathscr{M}^2_h$ is illustrated in Figure 1. \medskip

\begin{figure}[h]
\begin{center}
\begin{tikzpicture}
\draw (0,0.7) -- (0,2);
\draw[dotted] (-0.1,0.7) -- (-0.1, 2.2);
\node at (-0.4,2) {t};

\draw(-0.7,-0.5) -- (-2,-1);
\draw[dotted] (-0.69,-0.38) -- (-2.05, -0.9);
\node at (-2.05, -0.6) {$x$};

\draw (0.7,-0.5) -- (2,-1);
\draw[dotted] (0.69,-0.38) -- (2.05, -0.9);
\node at (2.05, -0.6) {$\wx$};

\draw (0,0.7) .. controls (-0.5,0.6) and (-0.7,0) .. (-0.7,-0.5);
\draw (0,0.7) .. controls (0.5,0.6) and (0.7,0) .. (0.7,-0.5);
\draw (-0.7,-0.5) .. controls (-0.5,-0.6) and (-0.4,-0.7) .. (-0.3,-0.7);
\draw (0.7,-0.5) .. controls (0.5,-0.6) and (0.4,-0.7) .. (0.3,-0.7);

\draw (-0.3,-0.7) .. controls (-0.3,-0.3) and (0.3,-0.3) .. (0.3,-0.7);
\draw (-0.3,-1.4) .. controls (-0.3,-1) and (0.3,-1) .. (0.3,-1.4);

\draw (0.3,-0.7) -- (0.3,-1.4);
\draw (-0.3,-0.7) -- (-0.3,-1.4);

\node at (1.2,0.7) {\large{rf}};
\node at (-1.2,0.7) {\large{lf}};
\node at (1.1, -1.2) {\large{tf}};
\node at (-1.1, -1.2) {\large{tf}};
\node at (0, -1.7) {\large{td}};
\node at (0,0.1) {\large{ff}};
\end{tikzpicture}
\end{center}
\label{heat-incomplete}
\caption{The heat-space $\mathscr{M}^2_h$.}
\end{figure}

We now describe projective coordinates in a neighborhood of the front face in $\mathscr{M}^2_h$, which are used often 
as a convenient replacement for the polar coordinates. The drawback is that projective coordinates are 
not globally defined over the entire front face. Near the top corner of the front face ff, projective coordinates are given by
\begin{align}\label{top-coord}
\rho=\sqrt{t}, \  \xi=\frac{x}{\rho}, \ \widetilde{\xi}=\frac{\wx}{\rho}, \ u=\frac{y-\wy}{\rho}, \ z, \ \wy, \ \wz.
\end{align}
With respect to these coordinates, $\rho, \xi, \widetilde{\xi}$ are in fact the defining 
functions of the boundary faces ff, rf and lf respectively. 
For the bottom right corner of the front face, projective coordinates are given by
\begin{align}\label{right-coord}
\tau=\frac{t}{\wx^2}, \ s=\frac{x}{\wx}, \ u=\frac{y-\wy}{\wx}, \ z, \ \wx, \ \wy, \ \widetilde{z},
\end{align}
where in these coordinates $\tau, s, \widetilde{x}$ are
the defining functions of tf, rf and ff respectively. 
For the bottom left corner of the front face,
projective coordinates are obtained by interchanging 
the roles of $x$ and $\widetilde{x}$. Projective coordinates 
on $\mathscr{M}^2_h$ near the temporal diagonal are given by 
\begin{align}\label{d-coord}
\eta=\frac{\sqrt{t}}{\wx}, \ S =\frac{(x-\wx)}{\sqrt{t}}, \ 
U= \frac{y-\wy}{\sqrt{t}}, \ Z =\frac{\wx (z-\wz)}{\sqrt{t}}, \  \wx, \ 
\wy, \ \widetilde{z}.
\end{align}
In these coordinates, tf is defined as the limit $|(S, U, Z)|\to \infty$, where
we use the Euclidean norm on $\R^m$; 
ff and td are defined by $\widetilde{x}, \eta$, respectively. 
The blow-down map $\beta: \mathscr{M}^2_h\to M^2_h$ is in 
local coordinates simply the coordinate change back to 
$(t, (x,y, z), (\widetilde{x},\wy, \widetilde{z}))$. \medskip

The restriction of $\beta$ to $\td$ is a fibration, with each fibre being a closed `parabolic' hemisphere $S^m_+$ 
and base given by the lifted diagonal of $(M_c)^2$, which is diffeomorphic to a copy of $M_c$. Similarly, the restriction 
of $\beta$ to $\ff$ is a fibration over the diagonal of $\del M_c \times \del M_c$. \medskip

In the statement below we assume that the $2$-tensor
$\kappa$ restricts to a smooth family of isospectral Riemannian metrics
on the fibres $F$, i.e. that the vertical operators $D_F$ are isospectral.
The second named author has worked under this condition in 
\cite{MazVer} jointly with Mazzeo. This condition ensures that the 
heat kernel of $\Delta$ is polyhomogeneous
in the sense of Definition \ref{phg} on the heat space blowup $\mathscr{M}^2_h$,
with exponents in the asymptotic expansion at the right and left boundary face given in terms of
the spectrum of $D_F$. \cite{MazVer} asserts the following result.

\begin{thm}\label{heat-expansion}
Let $(M,g)$ be an incomplete Riemannian manifold with an 
edge singularity and an admissible edge metric $g$ with the $2$-tensor
$\kappa$ restricting to a smooth family of isospectral Riemannian metrics
on the fibres $F$. Consider a flat Hermitian vector bundle over $M$.
Then the lift $\beta^*\HF$ of the Friedrichs heat kernel for the Hodge Laplacian
is a polyhomogeneous function on $\mathscr{M}^2_h$ 
taking values in $({}^{ie}\Lambda^* T^*M \otimes E)\boxtimes ({}^{ie}\Lambda^* T^*M \otimes E)$
with the index set $(-m+\N_0, 0)$ at ff, $(-m+\N_0, 0)$ at td, vanishing to infinite order at tf, and with some discrete 
index set at rf and lf.  
\end{thm}

To be precise, \cite{MazVer} establishes the heat kernel asymptotics
under certain rescaling transformation, as employed by Br\"uning and Seeley in \cite{BruSee:ITF}.
In view of the transformation rules in \cite[(3.13)]{MazVer}
it is straightforward to derive the index set of the heat kernel at ff without using the rescaling. \medskip

We point out that the heat kernel construction extends to squares of 
allowable perturbations of Gauss-Bonnet or signature operators, in the sense of Definition \ref{allowable}.

\begin{remark}\label{not-isospectral}
In case of $\kappa$ not being isospectral, the heat kernel of $\Delta$ 
has a rather complicated behaviour at lf and rf, but still admits a 
polyhomogeneous expansion at the front face ff and the temporal diagonal td of same order.
Therefore the isospectrality assumption will be lifted once we restrict the heat kernel 
to the diagonal and consider its trace asymptotics. This is why the isospectrality assumption does not 
appear in the main theorems.
\end{remark}

\subsection{Heat kernel asymptotics for the spin Laplacian} \label{spin-section}

Assume $M$ is spin and consider the spin bundle $S$ on $M$. 
We denote the corresponding spin Dirac operator by $D$. Recall that $D$ is essentially
self-adjoint under the geometric Witt assumption \ref{Witt}. Assume 
for the moment that the higher order term $h$ in the Riemannian metric
$g$ on $M$ is zero. Then the explicit structure of $D$ is studied
in \cite{Albin-Jesse}. \medskip

Since by the admissibility assumption, $\phi: (\partial M, \kappa + \phi^*g^B) 
\to (B, g^B)$ is a Riemannian submersion, the tangent bundle $T_p\partial M$ splits into vertical and horizontal subspaces  
$T^V_p \partial M \oplus T^H_p \partial M$, where the vertical subspace $T^V_p\partial M$ is the tangent space to the fibre of 
$\phi$ through $p$, the horizontal subspace $T^H_p \partial M \cong \phi^* TB$ is the pull back of the space of tangent vectors of the base, lifted to $Y$, and the tensor $\kappa$ restricted to $T^H_p \partial M$ vanishes. 
The splitting of the tangent bundle over $\partial M$
yields a splitting of the tangent bundle of $M$ near the edge
\begin{align}\label{split}
{}^{ie} TM \cong \langle \partial_x \rangle \oplus x^{-1}T^V \partial M \oplus \phi^* TB.
\end{align}
Since $\kappa$ restricted to $\phi^* TB$ vanishes, there exists a local orthonormal frame
of vector fields $\{\partial_x, x^{-1}V_\A, U_\beta\}$, which respects the splitting
\eqref{split}. If $\nabla$ denotes the block-diagonal covariant derivative on the spin bundle, 
as used in \cite[Lemma 2.2]{Albin-Jesse} the Dirac operator $D$ decomposes 
locally as 
\begin{align}
D = c(\partial_x) \partial_x + \frac{f}{2x} c(\partial_x) + \frac{1}{x} \sum_{\A=1}^f
c(x^{-1}V_\A) \nabla_{V_\A} + \sum_{\beta = 1}^b c(U_\beta) \nabla_{U_\beta} + W,
\end{align}
where $c$ denotes the Clifford multiplication and $W$ a higher order term. 
Allowing for non-zero higher term $h$ in the 
Riemannian metric $g$ just introduces higher order terms in the formula for $D$. 
From here it is clear that the normal operator $N(x^2D)_{y_0}$ is $s^2$ times 
the spin Dirac operator on the model edge $\R^+_s \times F_{y_0} \times \R^b$ 
with the model edge metric $g_{\textup{ie}} = ds^2 + s^2 \kappa_{y_0} + g^B_{y_0}$, 
in a similar way as in the Hodge de Rham setting. The notation of $s$ as a defining
function of $\R^+$ alludes to the fact that $N(x^2D)_{y_0}$ acts on the fibre front face
at $y_0 \in B$, with e.g. local coordinates \eqref{right-coord}.
\medskip

We can now construct the heat kernel $H$ of the spin Laplacian $D^2$ as a polyhomogeneous function on 
the blowup space $\mathscr{M}^2_h$ taking values in $S\boxtimes S^*$ along the 
lines of \cite{MazVer}. The normal operator $N(H)_{y_0}$, which is the leading order term 
of the heat kernel at the $y_0$-fibre of the front face, is obtained as a product of the heat kernel $H_{\mathscr{C}(F)}$
on the model cone $(\R^+_s \times F_{y_0}, ds^2 + s^2 \kappa_{y_0})$ with the heat kernel 
on the euclidean factor $(\R^b, g^B_{y_0})$. 
\medskip

The spin Laplacian $\Delta_{\mathscr{C}(F)}$ and the spin heat kernel $H_{\mathscr{C}(F)}$
on the model cone has been identified explicitly by Chou \cite[(2.6), (2.8)]{Chou}. More precisely,
\cite[(2.19)]{Chou} asserts that any eigenspinor of $D$ is of the form $\w_\mu (\phi_\mu \pm c(\partial_s) \phi_\mu)$ with 
the scalar part $\w_\mu$ given in terms of Bessel functions of first kind, and $\phi_\mu$ being a $\mu$-eigenspinor 
of the Dirac operator $D_F = \sum c(V_\A) \nabla_{V_\A}$ on the fibres $(F_{y_0},\kappa_{y_0})$. We write $S^\pm_\mu := 
(\phi_\mu \pm c(\partial_s) \phi_\mu)$. The action of the spin Laplacian on spinors of the form
$\w S^\pm_\mu$ with $\w\in C^\infty_0(\R^+)$ is then given, cf. \cite[Prop. 2.5, (2.6), (2.7)]{Chou},
by a scalar operator
\begin{align}
\Delta_{\mathscr{C}(F)} \left(\w S^\pm_\mu \right)
= \left[\left( -\partial^2_s - \frac{f}{s}\partial_s + \frac{1}{s^2}
\left(\mu^2\mp \mu - \frac{f^2-2f}{4}\right)\right) \w\right]S^\pm_\mu.
\end{align}
The heat kernel of this scalar action is well-known and expressed in 
terms of modified Bessel functions in \cite[Proposition 2.3.9]{Les:OOF}. 
Correspondingly, the spin heat kernel $H_{\mathscr{C}(F)}$ is written out 
in \cite[(5.9)]{Chou} with $\nu^\pm(\mu):= |2\mu\mp 1|/2$ as
\begin{align}\label{Bessel-heat-spin}
H_{\mathscr{C}(F)} = \sum_\mu H^\pm_\mu  S^\pm_\mu \otimes S^\pm_\mu,
\ H^\pm_\mu(t, s,\widetilde{s}) = \frac{1}{2t}(s\widetilde{s})^{(1-f)/2}I_{\nu^\pm(\mu)}
\left(\frac{s\widetilde{s}}{2t}\right)e^{-\frac{s^2+\widetilde{s}^2}{4t}}. 
\end{align}
Hence the normal operator $N(H)_{y_0}$ can be set up exactly 
as in \cite[(3.10)]{MazVer}. In order to construct the exact heat kernel, 
the initial parametrix $N(H)_{y_0}$ has to be corrected, which involves composition 
of Schwartz kernels on $\mathscr{M}^2_h$ by \cite[Theorem 5.3]{MazVer}. 
Following the heat kernel construction as in \cite{MazVer} verbatim we arrive 
under the condition of isospectrality for $\kappa$ at the following result.

\begin{thm}\label{heat-expansion-spin}
Let $(M,g)$ be an incomplete Riemannian spin manifold with an 
edge singularity and an admissible edge metric $g$ with the $2$-tensor
$\kappa$ restricting to a smooth family of isospectral Riemannian metrics
on the fibres $F$. Assume the geometric Witt condition. 
Then the spin Dirac operator $D$ is essentially self-adjoint and the
lift $\beta^*H$ of the heat kernel for the spin Laplacian $D^2$ is a polyhomogeneous function on $\mathscr{M}^2_h$ 
taking values in $S \boxtimes S$ with the index set $(-m+\N_0, 0)$ at ff, $(-m+\N_0, 0)$ at td, 
vanishing to infinite order at tf, and with some discrete 
index set at rf and lf.  
\end{thm}

We point out that the heat kernel construction extends to squares of 
allowable Dirac-type operators satisfying the Witt condition in the sense of Definition \ref{allowable}.
Moreover, in case of $\kappa$ not being isospectral, the heat kernel for $D^2$ still admits a 
polyhomogeneous expansion at the front face ff and the temporal diagonal td of same order.
Therefore the isospectrality assumption is lifted once we restrict $H$ and $DH$
to the diagonal and consider its trace asymptotics, and does not appear in the main theorems anymore.

\subsection{On locality of heat kernel asymptotics along the edge} \label{locality-section}

The heat kernel construction for the Hodge and spin Laplacians and their higher order 
(allowable) perturbations carries over verbatim to the case where the operators are 
twisted with a flat vector bundle $E$. We consider flat vector bundles induced by  
unitary representations $\rho: \pi_1(\overline{M}) \to U(n)$, where $\overline{M}$ is the 
smoothly stratified space; notice that, therefore, the flat vector bundles here are defined on all
of $\overline{M}$ and not only on the regular part $M$. \medskip

Let $\overline{\U}$ be a tubular neighbourhood of a singular stratum $B$; 
from the axioms of  Thom-Mather we know that $\overline{\U}$ is a fiber bundle over $B$, 
$\phi: \overline{\U} \to B$, with fiber equal to $\overline{\mathscr{C}(F)}$.
As already remarked, since the fibres $\overline{\mathscr{C}(F)}$ are contractible , $E$ is trivial when restricted to a distinguished
neighbourhood $\mathcal{V}\times \overline{\mathscr{C}(F)}$ of a point $y_0\in \mathcal{V}\subset B$, 
with $\mathcal{V}$ a contractible open neighbourhood of $y_0$ in $B$. 
In particular, $E$ is trivial over each link $F$, but not necessarily along the 
edge singularity $B$. \medskip

We seek to clarify here to what extent the asymptotic behaviour of the 
heat kernel $H$ near the edge depends on the twisting bundle $E$. To this end, 
we consider two such flat vector bundles $E_1$ and $E_2$ of rank $r$ and fix after an eventual refinement a trivializing atlas 
for both $E_1$ and $E_2$, for which the transition functions are locally constant. We can take such a trivializing
atlas made of distinguished neighbourhoods of the form $\mathcal{V}\times \overline{\mathscr{C}(F)}$
 near $B$. Each such neighborhood comes with a local bundle isomorphism 
 \begin{align}
 \phi: E_1 \restriction \mathcal{V}\times \overline{\mathscr{C}(F)} \to 
 E_2 \restriction \mathcal{V}\times \overline{\mathscr{C}(F)}.
 \end{align}
Each $E_j, j=1,2,$ carries a flat connection, restricted to $M$. 
The corresponding connection one-forms associated to this atlas (intersected with M)
are zero. 
\medskip

Consider the front face of the heat space $\mathscr{M}^2_h$, which is a fibration over 
$\partial M$ and as such can also be viewed as a fibration over $B$. Fix a base point $y_0\in B$.
The intersection of of the distinguished neighborhood $\mathcal{V}\times \overline{\mathscr{C}(F)}$
of $y_0$ with the open smooth stratum $M$ is given by $\mathcal{V}\times \mathscr{C}(F)$.
We denote the Hodge or the spin Laplacian twisted with $E_j$ by $\Delta_j$ for $j=1,2$. We denote their 
heat kernels by the corresponding lower index $j$ as well. Using the vanishing of the 
connection one forms of $E_j$ in the trivializing atlas we have chosen, we conclude
that over $\mathcal{V}\times \mathscr{C}(F)$ the operators are related by 
 \begin{align}
\Delta_1 = \phi^{-1} \circ \Delta_2 \circ \phi.
 \end{align}
The lifts of $\Delta_1$ and $\Delta_2$ to $\mathscr{M}^2_h$
are again related via conjugation by the local bundle isomorphism $\phi$
over $\mathcal{V}\times \mathscr{C}(F)$. Hence, for $\wy \in \mathcal{V}$ the normal operators
$N(H_1)_{\wy}$ and $N(H_2)_{\wy}$, which are set up exactly as in \cite[(3.10)]{MazVer}, 
are related via conjugation by $\phi$ as well. The initial parametrices $H_0(\Delta_j)$ are obtained
by extending $N(H_j)_{\wy}$ smoothly off the front face to the interior of $\mathscr{M}^2_h$ for any $j=1,2$. 
By construction, the front face asymptotics of $H_0(\Delta_j)(\cdot, \wy)$ with $\wy \in \mathcal{V}$ is, 
up to conjugation, independent of $j$, i.e. writing $\rho_\ff$ for the defining function of the front face in 
$\mathscr{M}^2_h$ we have for any $N \in \N$ 
\begin{align}\label{conjugation}
\left(H_0(\Delta_1) -  \phi^{-1} \circ H_0(\Delta_2) \circ \phi\right) (\cdot, \wy) = O(\rho^N_\ff), 
\ \rho_\ff \to 0.
\end{align} 
The initial parametrix $H_0(\Delta_j)$ solves the heat equation only up to an error 
$R(\Delta_j):=(\partial_t + \Delta_j) H_0(\Delta_j)$. 
For $\wy \in \mathcal{V}$ the front face asymptotics of $R(\Delta_j)(\cdot, \wy)$ 
is again as in \eqref{conjugation} independent of $j$ up to conjugation by $\phi$, 
since the same holds for $H_0(\cdot, \wy)$ and $\Delta$ in the open neighborhood $\mathcal{V}\times \mathscr{C}(F)$. 
Now the parametrix $H_0(\Delta_j)$ is corrected to define an exact solution to the heat equation 
by adding convolutions $H_0(\Delta_j) * R(\Delta_j)^\ell$ for $\ell \in \N$. These convolutions involve integrations over 
the base $B$ and hence a priori might depend on the global structure of $E_1$ and $E_2$. \medskip

In order to address this issue, consider a kernel $P$ that lifts to a polyhomogeneous 
function $\beta^*P$ on $\mathscr{M}^2_h$. Consider any $y \in B$ with $|y-\wy|\geq \delta$
bounded away from zero. We write $u=(y-\wy)/\rho_\ff$ in correspondence with the projective coordinates
in \eqref{top-coord} and \eqref{right-coord}. Assume that $P$ admits a factor of the form 
$\exp(-\frac{1}{2}|u|^2)$. Then for any $N\in \N$ we may estimate
\begin{equation}\label{P-ff}
\begin{split}
\rho_\ff^{-N} \| P(\cdot, y, \wy, \rho_\ff) \|_\infty &= 
|u|^N \, \| P(\cdot, y, \wy, \rho_\ff) \|_\infty \, |y-\wy|^{-N} \\
&\leq |u|^N \, \| P(\cdot, y, \wy, \rho_\ff) \|_\infty \, \delta^{-N}. 
\end{split}
\end{equation} 
and the right hand side is bounded by a constant $C_N$ thanks to the factor 
$\exp(-\frac{1}{2}|u|^2)$ in $P$. Consequently, $\beta^*P$ is vanishing to infinite order at the front face, when evaluated
at two points $y$ and $\wy$ on the edge $B$ with their distance bounded away from zero.
\medskip

Consider now the convolutions $H_0(\Delta_j) * R(\Delta_j)^\ell$, whose definition involves integrations of the kernels
over the base $B$. Since the initial parametrix $H_0(\Delta_j)$ admits a factor of the form $\exp(-\frac{1}{2}|u|^2)$
by definition, this factor is present in the error term $R(\Delta_j)$ and the related convolutions $H_0(\Delta_j) * R(\Delta_j)^\ell$
as well. Then \eqref{P-ff} shows that the contributions to $H_0(\Delta_j) * R(\Delta_j)^\ell (\cdot, \wy)$ 
coming from integration on $B$ away from a $\wy$ are vanishing to infinite order at ff. 
Consequently, the front face asymptotics of $H_0(\Delta_j) * R(\Delta_j)^\ell (\cdot, \wy)$ depends only on 
the asymptotics of $H_0(\Delta_j)(\cdot, y)$ and $R(\Delta_j)(\cdot, y)$ for $y$ sufficiently close to $\wy \in \mathcal{V}$, 
and hence 
\begin{equation}\label{conjugation2}
\left(H(\Delta_1) -  \phi^{-1} \circ H(\Delta_2) \circ \phi\right) (\cdot, \wy) = O(\rho^N_\ff), 
\ \rho_\ff \to 0.
\end{equation} 
This proves that the front face asymptotics of
the lifted heat kernel $\beta^*H(\cdot, \wy)$ is up to conjugation of the coefficients,
independent of the choice of a flat vector bundle of fixed rank, for $\wy$ in an open neighborhood 
of $y_0$. A similar argument applies for the expansion of 
the heat kernel near the temporal diagonal. If we restrict the heat kernel to the diagonal
and take its pointwise trace, then dependence of the expansions on the conjugation by a bundle morphism
cancels and we arrive at the following theorem.

\begin{thm}\label{coefficients}
Let $E$ be a flat vector bundle induced by a unitary representation of the fundamental
group of $\overline{M}$. Consider the heat kernel $H$ of the square of an allowable perturbation of a
geometric Dirac operator $D$, twisted by the flat vector bundle $E$. Then the asymptotic expansion
of the pointwise traces $\tr H$ and $\tr D H$ at ff and td is up to the bundle rank independent of 
the particular choice of $E$. 
\end{thm}

\section{The even and odd heat calculus on edge spaces}\label{even-odd-section}

Admissibility of $g$, more precisely evenness of the higher order term $h$ in its 
behavior near the edge, in fact yields a more refined asymptotic information 
beyond Theorem \ref{heat-expansion}. This refinement is obtained by singling out
even and odd polyhomogeneous functions on the blowup space $\mathscr{M}^2_h$. 
The even/odd classification is slightly different for the Hodge Laplacian and for the 
spin Laplace operator and hence done here separately. As before, we assume in both 
cases the geometric Witt condition, so that the corresponding Dirac operators are essentially
self-adjoint and the heat kernels for their squares can be studied as polyhomogeneous
functions on the heat space $\mathscr{M}^2_h$.

\subsection{The even and odd heat calculus for the Hodge Laplacian}

We begin with the explicit definition of a heat calculus.

\begin{defn}\label{heat-calculus}  
Let $\calE = (E_{\lf}, E_{\rf})$ be an index family for the two side faces of $\mathscr{M}^2_h$. We define $\Psi^{\ell,p,\calE}_{\eh}(M)$ 
to be the space of all operators $P$ with Schwartz kernels $K_P$ which are pushforwards from polyhomogeneous functions 
$\beta^*K_P$ on $\mathscr{M}^2_h$, taking values in $({}^{ie}\Lambda^* T^*M \otimes E)\boxtimes ({}^{ie}\Lambda^* T^*M \otimes E)$
with index family $(-m-2+\ell +\N_0,0)$ at $\ff$, $(-m  + p + \N_0, 0)$ 
at $\td$, $\varnothing$ at $\tf$ and $\calE$ for the two side faces of $\mathscr{M}^2_h$. The subscript $\eh$ in $\Psi^{\ell,p,\calE}_{\eh}(M)$ 
stands for "edge-heat".
\end{defn}

The naming \emph{calculus} is due to composition formulae, which hold under 
certain restrictions on the index sets at the side faces
\[
\Psi^{\ell,p,\calE}_{\eh}(M) \circ \Psi^{\ell',p',\calE'}_{\eh}(M) \subset \Psi^{\ell+\ell', p + p', \calE''}_{\eh}(M),
\]
where the index set $\calE''$ is defined in terms of the index sets $(\calE, \calE')$ 
as in \cite[Theorem 5.3]{MazVer}. \medskip
 
We now introduce the definition of even and odd subcalculi, cf. \cite{MazVer}.
Note that we do not split $\Psi^{*}_{\eh}(M)$ into even and odd parts, but rather
define subspaces of the calculus.
Consider in an open neighborhood $\U$ of each edge local coordinates $(x,y,z)$ and 
the local differential forms $\{dx, x dz_1,.., xdz_f, dy_1, .., dy_b\}$
which generate the exterior algebra ${}^{ie}\Lambda^* T^*M$ near the edge.
We assign odd parity to the forms $\{dx, x dz_1,.., xdz_f\}$ and even parity to the 
forms $\{dy_1, .., dy_b\}$. Exterior product of two even differential forms is set to be even again.
Exterior product of two odd differential forms is set to be even as well. Exterior product of 
an odd and an even differential form is set to be odd. This decomposes the exterior algebra 
${}^{ie}\Lambda^* T^*M$ near the edge into an even $\Lambda_+$ and an odd subbundle $\Lambda_-$
\footnote{These should not be confused with the self-dual / anti self-dual decomposition of the signature operator, 
for which a similar notation with $\pm$ as upper script is usually employed. }
\begin{align}\label{lambda-pm}
{}^{ie}\Lambda^* T^*M \restriction \U = \Lambda_+ \oplus \Lambda_-.
\end{align}
Evenness and oddness of differential forms, is well-defined within a fixed choice of 
special coordinates, as introduced in \S \ref{geometry-section}. Similarly, evenness
of the admissible Riemannian metric $g$ also depends on the choice of special
coordinates, which we fix henceforth. \medskip

Consider the space $C^\infty_+(\U)$ of even functions on $\U_c$, i.e. functions 
that are smooth in $\U_c$ with a Taylor expansion at $x=0$ involving only even powers of $x$.
Consider the space $C^\infty_-(\U)$ of odd functions on $\U_c$, i.e. functions 
that are smooth in $\U_c$ with a Taylor expansion at $x=0$ involving only odd powers of $x$.
Clearly, $C^\infty_-(\U)=xC^\infty_+(\U)$.
\medskip

The assignment of the even/odd parity is not ad hoc and in no way arbitrary.
Recall that the evenness condition on the Riemannian edge metric asserts that the 
coefficients of $g$ are even functions of $x$, except for the coefficients corresponding to the
$dx$-cross terms which are required to be odd in $x$. One may check explicitly that the same
continues to hold for the inverse Riemannian metric on differential forms. Consequently, 
$g(dx,dx), g(dy_i, dy_j), g(xdz_k, xdz_\ell)$ as well as $g(dx,xdz_k)$ are always 
even functions of $x$.  Moreover, $g(dx,dy_j)$ as well as $g(dy_i,xdz_j)$ are always 
odd functions of $x$.
\medskip

Extending these relations to pairing of exterior products of cotangent vectors, we find
for the pointwise traces 
\begin{equation}\label{trace-point}
\begin{split}
&\textup{tr} (\Lambda_\pm \boxtimes \Lambda_\pm) = g (\Lambda_\pm, \Lambda_\pm) \subset C^\infty_+(\U), \\
&\textup{tr} (\Lambda_\pm \boxtimes \Lambda_\mp) = g (\Lambda_\pm, \Lambda_\mp) \subset C^\infty_-(\U).
\end{split}
\end{equation}
where the equalities in \eqref{trace-point} are justified by identifying $\Lambda_\pm \boxtimes \Lambda_\pm$ and
$\Lambda_\pm \boxtimes \Lambda^*_\pm$ through the metric and using that $\Lambda_\pm \boxtimes \Lambda^*_\pm |_{\Delta}=
\operatorname{Hom}(\Lambda_\pm, \Lambda_\pm)$.
We can now define the even/odd calculus by imposing conditions on the asymptotic expansions of elements in $\Psi_{\eh}^*$ at
the temporal diagonal $\td$ and the front face $\ff$. 

\begin{defn}\label{def-even}
Consider $P \in \Psi^{*}_{\eh}(M)$ with the Schwartz kernel 
$K_P$. Then $P$ is said to be even if the following conditions hold. 

\begin{enumerate}
\item In an open neighborhood of the temporal diagonal $\td$, in terms of the projective coordinates \eqref{d-coord}, 
valid uniformly up to the front face $\ff$
\footnote{The grading ${}^{ie}\Lambda^* T^*M \restriction \U = \Lambda_+ \oplus \Lambda_-$ is ignored 
when we specify the asymptotics near td.}
\begin{align*}
&\beta^*K_P \sim \eta^{-m-2} \sum_{n=0}^\infty \kappa^{\td}_n(S, U, Z; \wx, \wy, \wz) \eta^{n}, 
\\ &\textup{with} \ \kappa^{\td}_{n}(-S, -U, -Z; \wx, \wy, \wz) = (-1)^{n} \kappa^{\td}_{n}(S, U, Z; \wx, \wy, \wz)
\ \textup{for all} \ n.
\end{align*}
\item In an open neighborhood of the front face $\ff$ we write 
$K_P$ with respect to the decomposition 
${}^{ie}\Lambda^* T^*M \restriction \U = \Lambda_+ \oplus \Lambda_-$ as a matrix
\begin{equation*}
K_P = \left( \begin{split} K_{++} & \quad K_{+-} \\ K_{-+} & \quad K_{--}
\end{split}\right).
\end{equation*}
In terms of the projective coordinates \eqref{top-coord} (similarly in the 
projective coordinates \eqref{right-coord})
valid in the interior of ff
\begin{equation*}
\beta^*K_P \equiv 
\beta^*  \left( \begin{split} K_{++} & \quad K_{+-} \\ K_{-+} & \quad K_{--}
\end{split}\right) \sim \rho^{-m-2} \sum\limits_{n=0}^{\infty} 
\left( \begin{split} \kappa^{\ff}_{n,++} & \quad \kappa^{\ff}_{n,+-} 
\\ \kappa^{\ff}_{n,-+} & \quad \kappa^{\ff}_{n,--}
\end{split}\right)
(\xi, \widetilde{\xi}, u, z, \widetilde{y}, \widetilde{z}) \rho^{n}, 
\end{equation*}
where the coefficients satisfy the following parity conditions
\begin{align*}
& \kappa^{\ff}_{n, \pm , \pm}(\xi, \widetilde{\xi}, -u, z, \widetilde{y}, \widetilde{z})=
(-1)^{n}\kappa^{\ff}_{n, \pm , \pm}(\xi, \widetilde{\xi}, 
u, z, \widetilde{y}, \widetilde{z}) \ \textup{for all} \ n, \\
&\kappa^{\ff}_{n, \pm , \mp}(\xi, \widetilde{\xi}, -u, z, \widetilde{y}, \widetilde{z})=
(-1)^{n+1}\kappa^{\ff}_{n, \pm , \mp}(\xi, \widetilde{\xi}, 
u, z, \widetilde{y}, \widetilde{z}) \ \textup{for all} \ n.
\end{align*}
\item Near the intersection $\td \cap \ff$,
due to product polyhomogeneity in terms of projective coordinates \eqref{d-coord}, 
valid in an open neighborhood of $\td \cap \ff$
\begin{equation*}
\beta^*  \left( \begin{split} K_{++} & \quad K_{+-} \\ K_{-+} & \quad K_{--}
\end{split}\right) \sim (\wx \eta)^{-m-2} \sum\limits_{k=0}^{\infty}  \sum\limits_{\ell=0}^{\infty} 
\left( \begin{split} \kappa^{k\ell}_{++} & \quad \kappa^{k\ell}_{+-} 
\\ \kappa^{k\ell}_{-+} & \quad \kappa^{k\ell}_{--}
\end{split}\right) (-S, -U, -Z; \, \wy, \wz)  \, \wx^k \eta^\ell, 
\end{equation*}
where the coefficients satisfy the following parity conditions.
The following parity with respect to the coordinates $(S,U,Z)$ holds
\begin{align*}
& \kappa^{k\ell}_{\pm , \pm}(-S, -U, -Z) =
(-1)^{\ell}\kappa^{k\ell}_{\pm , \pm}(S, U, Z), \\
& \kappa^{k\ell}_{\pm , \mp}(-S, -U, -Z) =
(-1)^{\ell}\kappa^{k\ell}_{\pm , \mp}(S, U, Z).
\end{align*}
Moreover, the following parity with respect to the coordinate $U$ holds
\begin{align*}
& \kappa^{k\ell}_{\pm , \pm}(S, -U, Z) =
(-1)^{k}\kappa^{k\ell}_{\pm , \pm}(S, U, Z),  \\
& \kappa^{k\ell}_{\pm , \mp}(S, -U, Z) =
(-1)^{k+1}\kappa^{k\ell}_{\pm , \mp}(S, U, Z).
\end{align*}
\end{enumerate}

We denote the set of all even operators by $\Psi^{*}_{+}(M)$.
We define the odd subcalculus $\Psi^{*}_{-}(M)$ by exchanging the parities of
the coefficients $\kappa^{\ff}_{n, \pm , \pm}$ and 
$\kappa^{\ff}_{n, \pm , \mp}$. The parity at td remains unchanged. 
\end{defn}

Similar to \cite[Proposition 3.6]{MazVer}, the explicit the composition of even and odd operators
obey the following rule
\begin{equation}\label{even-odd-rule}
\begin{split}
&\Psi^{*}_{\pm}(M) \circ \Psi^{*}_{\pm}(M) \subset \Psi^{*}_{+}(M), 
\\ &\Psi^{*}_{\pm}(M) \circ \Psi^{*}_{\mp}(M) \subset \Psi^{*}_{-}(M).
\end{split}
\end{equation}
Note that while \cite[Proposition 3.6]{MazVer} only addresses explicitly the composition rule for
even operators, \eqref{even-odd-rule} still follows by the same arguments. \medskip

Note that the parity of coefficients in the definition of the even subcalculus is defined
precisely so that after restricting to the diagonal (the diagonal in $\mathscr{M}^2_h$
is again a manifold with corners, introduced in detail in \S \ref{trace-section} with $\eta$ 
still being the defining function of td, and $\rho$ the defining function of ff) and taking pointwise trace, one finds 
in view of \eqref{trace-point} for any $P \in \Psi^{*}_{+}(M)$ that $\textup{tr} \, \beta^* K_P$
is polyhomogeneous with
\begin{equation}\label{exp-point}
\begin{split}
&\textup{tr} \, \beta^* K_P \sim   \eta^{-m-2} \sum\limits_{n=0}^\infty 
\kappa^{\td}_{2n}(\wx, \wy, \wz) \eta^{2n}, \quad \eta \to 0, \\
&\textup{tr} \, \beta^* K_P \sim   \rho^{-m-2} \sum\limits_{n=0}^{\infty} 
\kappa^{\ff}_{2n}(\widetilde{\xi}, \widetilde{y}, \widetilde{z}) \rho^{2n}, 
\quad \rho \to 0.
\end{split}
\end{equation}
Similar expansions hold of course for $P \in \Psi^{*}_{-}(M)$ with 
the index $2n$ being replaced by $(2n+1)$ in the expansion near $\ff$. We should point out that 
the even odd classification can equivalently be introduced by studying 
the sign change under the map $(x,u)\mapsto (-x,-u)$, cf. \cite{Albin}, once 
evenness of the pointwise traces is taken into account. We now study the 
action of the geometric operators on $(M,g)$ with respect to the
even odd calculus. 

\begin{prop}\label{even-odd-examples}
Let $m$ be the dimension of the incomplete edge manifold $M$ with an admissible edge metric $g$.
Denote by $b$ the dimension of each edge in $M$.
Then the exterior derivative $d$ and the Hodge star operator $*$ acts as follows with respect 
to the even odd subcalculi
\begin{align*}
d \circ \Psi^*_{\pm} \subset \Psi^*_{\pm}, \quad
* \circ \Psi^*_{\pm} \subset \Psi^*_{\pm}, \ \textup{if $(m-b)$ is even}; \quad
* \circ \Psi^*_{\pm} \subset \Psi^*_{\mp}, \ \textup{if $(m-b)$ is odd}.
\end{align*}
In particular 
\begin{enumerate}
\item the Hodge Laplacian respects the even subcalculus $\Psi^*_{+}$,
\item the Gauss Bonnet operator respects the even subcalculus $\Psi^*_{+}$, 
\item the signature operator\footnote{on an even-dimensional edge manifold} 
respects the even subcalculus $\Psi^*_{+}$ only if $b$ is even \footnote{when 
$\dim F$ is odd and $M$ is therefore Witt. The case of $\dim F$ even is excluded in this statement.}.
\item the odd signature operator\footnote{on an odd-dimensional edge manifold}  respects the even subcalculus $\Psi^*_{+}$ 
if $b$ is odd, and acts as $D \circ \Psi^*_{\pm} \subset \Psi^*_{\mp}$ if $b$ is even.
\end{enumerate}
\end{prop}

\begin{proof}
Consider the individual components of the exterior derivative $d$. 
Its action leaves the parity of coefficients in the td asymptotics invariant. Indeed for instance
$\beta^*(\partial_x) = (\wx \eta)^{-1} \partial_S$, where the lowering of td asymptotics by $\eta^{-1}$ is offset by
the parity change for the coefficients $\partial_S \kappa^{\td}_n$. \medskip

At the front face,
differentiation in $x$ with $\partial_x (\cdot )\wedge dx$ does not change the parity. Indeed, 
in \eqref{d-coord} coordinates $\beta^*(\partial_x) = \rho^{-1} \partial_\xi$
and hence differentiation in $x$ lowers the front face behaviour by one and is an odd action.
However $dx$ is odd as well, so that the total action of $\partial_x (\cdot )\wedge dx$ is even.
The $z$-derivative component in $d$ comes as $x^{-1}\partial_z (\cdot )\wedge xdz$. 
The factor $x^{-1}$ decreases front face behaviour by one and hence is odd. However, 
$xdz$ is odd as well, so that the total action of $x^{-1}\partial_z (\cdot )\wedge xdz$ is even.
Finally, the action of $\partial_y (\cdot )\wedge dy$ is even also. Indeed, $dy$ is even and 
$\beta^*(\partial_y) = \rho^{-1}\partial_u$ respects the even odd parity as well. \medskip

Let us now consider the Hodge star operator $*$ of $(M,g)$. Here we just need to count
the number of odd components $\{dx, xdz\}$ and the even components $\{dy\}$. 
We find that $*:\Lambda_\pm \to \Lambda_\pm$ if $(m-b)$ is even, and 
$*:\Lambda_\pm \to \Lambda_\mp$ if $(m-b)$ is odd. From there it is 
clear that $*$ respects the even odd subcalculus if $(m-b)$ is even, while
it changes the parity if $(m-b)$ is odd.\medskip

It remains to check the actions of the various geometric Dirac and Laplace operators. 
Action of the Hodge Laplacian and the Gauss-Bonnet operator
follows directly from the properties of $d$ and $*$. Action of the odd signature operator 
$* d \pm d*$ (in case $m=\dim M$ is odd) also follows directly from the properties of $d$ and $*$.
\medskip

The statement on the signature
operator in case $m=\dim M$ is even is obtained as follows. Recall that the signature
operator is the restriction of the Gauss-Bonnet operator to the $(+1)$-eigenspace 
of an involution $\tau$ with $\tau \restriction \Lambda^p T^*M  = i^{p(p-1)+m/2} *$. 
Since for $m$ and $b$ even the Hodge star operator $*$ respects the
the even and odd subcalculi, so does $\tau$. Hence the $(+1)$-eigenspace of 
$\tau$ decomposes into a direct sum of $\Lambda_+$ and $\Lambda_-$
components. The arguments for the Gauss-Bonnet operator then carry over to the
signature operator. \medskip

This argument fails for the signature operator in case $b$ is odd. 
The $(\pm1)$-eigenspace of $\tau$ is spanned by elements of the form 
$(\w\pm\tau \w)$ where $\w$ is any differential form on $M$. Since for $b$ 
odd, $\tau$ interchanges the even and odd differential forms, its $(\pm 1)$
eigenspaces are comprised of elements of mixed parity and do not admit a 
direct sum decomposition into even and odd forms. In this case the signature
operator acting on the $(+1)$-eigenspace of $\tau$ does not any longer admit
an even or odd classification. 
 \end{proof}

Analysis of the even odd parity for the exterior derivative extends to the
twisted setting. As before, let $E$ be a flat vector bundle on $M$ corresponding
to a unitary representation of the fundamental group of the singular space $\overline{M}$. 
This flat vector bundle induces in a natural way a flat vector bundle on the resolved
manifold $M_c$. Let $Y=\partial M_c$. We can and we shall assume that 
\begin{equation}\label{flatbd}
E|_{\U}\cong \pi^*(E_Y)\cong (0,1]\times E_Y,
\end{equation}
where $E_Y$ is the restriction of $E$ to $Y$,
 and where $\pi\colon (0,1]\times Y\simeq \U\to Y$ is the canonical projection.
We can endow the flat vector bundle with an hermitian metric and a covariant derivative
which are compatible with  \eqref{flatbd}.
In particular  we shall assume that  the flat covariant derivative $\nabla$  has a product structure 
bear the boundary, i.e. there exists a  $\nabla_Y$ 
such that for $s\in C^{\infty}((0,1], C^{\infty}(E_Y))\cong
\Gamma(E|_{\U})$ we have
\begin{align}\label{ddy}
\nabla s= \frac{\partial s}{\partial x}\otimes dx + \nabla_Y s.
\end{align}
Extend the covariant derivative to a linear operator on twisted differential forms $\nabla: \Omega^k(M,E)
\to \Omega^{k+1}(M,E)$. In view of \eqref{ddy} we can now repeat the previous arguments 
to deduce that $\nabla$ acts as $\nabla \circ \Psi^*_{\pm} \subset \Psi^*_{\pm}$, like the 
exterior derivative in Proposition \ref{even-odd-examples}. In particular the twisted Hodge 
Laplacian respects the even subcalculus as well. Note that unitarity of the representation $\rho$
is irrelevant for \eqref{flatbd} and \eqref{ddy}; however, it is central for asserting that the twisted geometric Dirac operator are 
formally self-adjoint. \medskip

While the statements on evenness and oddness of the Hodge star operator 
$*$ in \cite{MazVer} differ slightly from the discussion here,
our arguments still imply that the Hodge Laplacian on an edge 
manifold with an admissible edge metric
is even, as asserted in \cite{MazVer}. Evenness of the Hodge 
Laplacian as well as the fact that even kernels compose to even 
kernels again (and indeed form a subcalculus) is used in \cite{MazVer} to establish evenness of the heat kernel.

\begin{thm}[\cite{MazVer}]\label{hkeven}
The heat kernel $\HF$ for the Friedrichs extension of the Hodge Laplacian 
on an incomplete edge space $(M,g)$ with an admissible edge metric is
an element of $\Psi^{2,0,\calE}_{+}$, for some index
set $\calE$ at rf and lf. 
\end{thm}

We point out that evenness of $\HF$ was derived in \cite{MazVer}
under the rescaling transformation $\Phi$. The 
precise relation between the asymptotics of Schwartz kernels with 
and without the rescaling has been worked out in \cite[(3.13)]{MazVer} 
and the statement here follows by an easy counting of the 
leading orders in the front face asymptotics.  \medskip

In view of the parity properties of the Gauss Bonnet and the odd signature 
operators, as established in Proposition \ref{even-odd-examples}, we
introduce the notion of even and odd differential operators.
Denote by $\textup{Diff}_e^q(M, {}^{ie}\Lambda^*T^*M \otimes E)$ the space of differential operators 
acting on sections of ${}^{ie}\Lambda^*T^*M \otimes E$ by smooth linear combinations of $\ell$-th order 
compositions of edge vector fields $\mathcal{V}_e$.

\begin{defn}\label{eo-hodge}
An operator $D\in x^{-q}\textup{Diff}_e^q(M, {}^{ie}\Lambda^*T^*M \otimes E)$ is 
\emph{even} if $D  \circ \Psi^*_{\pm}(M) \subset \Psi^*_{\pm}(M)$. 
Similarly, $D\in x^{-q}\textup{Diff}_e^q(M, {}^{ie}\Lambda^*T^*M \otimes E)$ is \emph{odd}
if $D \circ \Psi^*_{\pm}(M) \subset \Psi^*_{\mp}(M)$. 
\end{defn}

Following the microlocal heat kernel construction in \cite{MazVer}, we may 
similarly construct the heat kernel for the square of even or odd allowable Dirac-type operators,
thereby extending the statement of Theorem \ref{hkeven}

\begin{thm}
Let $D$ be an allowable essentially self-adjoint Dirac-type operator of either even or odd parity
on an incomplete edge space $(M,g)$ with an admissible edge metric. Then the heat kernel of $D^2$ is
an element of $\Psi^{2,0,\calE}_{+}$, for some index set $\calE$ at rf and lf. 
\end{thm}

We point out that both the even and the odd operators are set to preserve
the parity of coefficients in the asymptotics at the temporal diagonal td.
This is in fact a general property of a Dirac type operator,  $D$, well 
known in the classical setting of smooth manifolds. There, indeed, it is 
a known, see for example \cite[Lemma 1.9.1]{Gilkey}, that in the asymptotic 
expansion of the pointwise trace
\begin{align}\label{td-even}
\textup{tr} D e^{-tD^2}(p,p) \sim_{t\to 0+} \sum_{n=0}^\infty \kappa_n(p) t^{\frac{-m-1+n}{2}}
\end{align}
the even coefficients $\kappa_{2k}$ vanish identically. In the singular
setting this property of the expansion corresponds to the Schwartz kernel of $D e^{-tD^2}$ having an even 
asymptotics at the temporal diagonal.

\subsection{The even and odd heat calculus for the spin Laplacian}

We assume that $M$ is spin and consider the spinor bundle $S$. 
As before the heat calculus is defined as follows.

\begin{defn}\label{heat-calculus-spin}  
Let $\calE = (E_{\lf}, E_{\rf})$ be an index family for the two side faces of $\mathscr{M}^2_h$. 
We define $\Psi^{\ell,p,\calE}_{\eh}(M)$ 
to be the space of all operators $P$ with Schwartz kernels $K_P$ 
which are pushforwards from polyhomogeneous functions 
$\beta^*K_P$ on $\mathscr{M}^2_h$, taking values in 
$S \boxtimes S$ with index family $(-m-2+\ell +\N_0,0)$ at $\ff$, $(-m  + p + \N_0, 0)$ 
at $\td$, $\varnothing$ at $\tf$ and $\calE$ for the two side faces of $\mathscr{M}^2_h$. 
\end{defn}

Assume for the moment that $m$ is even. Then \cite[Proposition 3.19]{BGV} identifies 
the spinor bundle $S$ locally as follows. Given an oriented local orthonormal basis\footnote{We assume
for simplicity that the higher order term $h$ in the Riemannian metric is zero.}
$\{e_1, ..., e_m\} := \{\partial_x, x^{-1}V_\A, U_\beta\}$ of $TM$, a polarization
$P$ of the complexified tangent bundle $TM^\C = P \oplus P^*$ is defined in \cite[Definition 3.18]{BGV} 
as the span of $\{e_{2j-1} - i e_{2j} \mid j=1,.., m/2\}$.
The spinor bundle $S$ is then locally defined as the exterior algebra $S=\Lambda^{*} P$.
The spinor bundle $S$ decomposes into a direct sum of half spinor bundles $\oplus_k \Lambda^{2k} P$ and 
$\oplus_k \Lambda^{2k+1} P$. 
The Riemannian metric $g$, extended to $TM^\C$ by $\C$-linearity, now canonically 
defines an inner product on the fibres of $S$.
Now, if $b$ is even, the fibre dimension $f$ is odd and the polarization $P$ is spanned by elements of the form 
$\{\partial_x - i x^{-1}V_1, x^{-1}V_{2k} - x^{-1}V_{2k+1}, U_{2j-1}- i U_{2j} \mid k=1, ..., (f-1)/2, j=1,..., b/2\}$.
\medskip

As in the previous subsection for differential forms, we assign, in view of the evenness condition on $g$, 
even parity to the vector fields $U_\beta$, and odd parity
to vector fields $\partial_x$ and $x^{-1}V_\A$. The crucial observation is that in the elements spanning $P$ there are no differences of 
even and odd vector fields, so that the spanning elements above are either of even or of odd, but not of mixed parity.
This defines parity for sections of $P$.
\medskip

Parity of elements in $S=\Lambda^{*} P$ is now defined as follows: exterior product of two even or two odd elements 
is defined to be even, exterior product of an even and an odd element is defined to be odd. 
We can therefore decompose $S$ as in \eqref{lambda-pm} into even and odd subbundles
\footnote{This should not be confused with the decomposition of
the spinor bundle into half spinor bundles, for which a similar notation with $\pm$ as superscript is usually employed. } 
\begin{align}
S \restriction \mathscr{U} = S_+ \oplus S_-.
\end{align}

Now let us consider the case of $m$ odd. In this case the spinor bundle arises locally as the positive half spinor bundle
of $M\times \R$. The tangent vector coming from the additional $\R$ component is assigned even parity
and we may again decompose $S \restriction \mathscr{U} = S_+ \oplus S_-$ as long as $(m-b)$ is even.

\begin{defn}\label{def-even-spin}
Let $(m-b)$ be always even for $m= \dim M$ and $b$ being dimension 
of any edge singularity. Consider $P \in \Psi^{*}_{\eh}(M)$ with the Schwartz kernel 
$K_P$. Then $P$ is said to be even if the conditions in Definition \ref{def-even} hold verbatim with 
the grading of differential forms replaced by the grading $S \restriction \mathscr{U} = S_+ \oplus S_-$. 
We denote the space of all even operators by $\Psi^{*}_{+}(M)$ and the 
space of all odd operators by $\Psi^{*}_{-}(M)$.
\end{defn}

As in the previous subsection we classify even and odd 
edge differential operators. 

\begin{defn}\label{eo-spin}
An operator $D\in x^{-q}\textup{Diff}_e^q(M, S)$ is 
\emph{even} if $D  \circ \Psi^*_{\pm}(M) \subset \Psi^*_{\pm}(M)$. 
Similarly, $D\in x^{-q}\textup{Diff}_e^q(M, S)$ is \emph{odd}
if $D \circ \Psi^*_{\pm}(M) \subset \Psi^*_{\mp}(M)$. 
\end{defn}

We can now prove the following central theorem. 

\begin{thm}\label{hkeven-spin}
Assume that the spin Dirac operator $D$ on an incomplete edge space $(M,g)$ 
with an admissible edge metric satisfies the geometric Witt condition. 
Assume that for the dimension $m$ of $M$ and for the dimension $b$ of any edge 
singularity, we always have $(m-b)$ is even. Then $D$ and $D^2$
are both even operators and the heat kernel $H$ 
for the spin Laplacian $D^2$ is an element of $\Psi^{2,0,\calE}_{+}$ for some index
set $\calE$ at rf and lf. 
\end{thm}

\begin{proof}
Note that by the explicit formulae of \cite{Chou}, the normal operator 
$N(x^2D^2)_{y_0}$ defines an initial heat kernel parametrix inside the 
even subcalculus $\Psi^{2,0,\calE}_{+}$. If $D$ maps the even subcalculus
$\Psi^{*}_{+}(M)$ to itself, then so does the spin Laplacian $D^2$. 
Following the argument outlined in \cite{MazVer}, the exact heat kernel 
must be an element of the even subcalculus as well. \medskip

Hence it remains to study the action of $D$ with respect to the even and odd 
subcalculi. Recall from above
\begin{align}
D = c(\partial_x) \partial_x + \frac{f}{2x} c(\partial_x) + \frac{1}{x} \sum_{\A=1}^f
c(x^{-1}V_\A) \nabla_{V_\A} + \sum_{\beta = 1}^b c(U_\beta) \nabla_{U_\beta} + W,
\end{align}
where $\{\partial_x, x^{-1}V_\A, U_\beta\}$ is an orthonormal frame
of vector fields , which respects the splitting \eqref{split}; $\nabla$ 
the covariant derivative on the spin bundle, induced 
from the lift of the Levi-Civita connection to the spin principle bundle; 
$c$ denotes the Clifford multiplication and $W$ a higher order term, 
made explicit in the proof of \cite[Lemma 2.2]{Albin-Jesse}. \medskip

In order to clarify if $D$ respects or interchanges the even and odd subcalculi,
we need to make explicit how the pointwise trace on $S\boxtimes S$ is induced by the 
admissible edge metric $g$. We also need to identify the Clifford action in explicit
terms. Recall the constructions from the beginning of this subsection.
Consider a polarization $P$ of the complexified tangent bundle $TM^\C = P \oplus P^*$ as defined above 
and the spinor bundle $S$ is then locally defined as the exterior algebra $S=\Lambda^{*} P$.
We may write any given tangent vector $V$ as a direct sum of an element $\pi(V)$ in $P$ and an element 
$\overline{\pi}(V)$ in the complex conjugate $\overline{P}\cong P^*$.  Then for any $s\in S$
\begin{equation}
\begin{split}
&c(\pi(V)) \cdot s = \sqrt{2} \textup{ext}(\pi(V)) s, \\
&c(\overline{\pi}(V)) \cdot s = -\sqrt{2} \textup{int}(\overline{\pi}(V)) s.
\end{split}
\end{equation}
One checks explicitly that $\pi(\partial_x) = (\partial_x - i x^{-1}V_1)/2$ and $\overline{\pi}(\partial_x) = (\partial_x + i x^{-1}V_1)/2$
are both of odd parity. Similarly, $\pi(V)$ and $\overline{\pi}(V)$ are odd for $V\in \{x^{-1}V_\A\}$, and even for 
$V\in \{U_\beta\}$. Consequently, Clifford multiplication by odd vector fields is indeed odd, and Clifford 
multiplication by even vector fields is indeed even. Note that this line of argumentation breaks down completely
if $b$ is odd. In that case we cannot exclude elements of mixed parity from the set of generators of $P$
and hence cannot classify the parity of the Clifford action any longer. This is in accordance with the case of the signature 
operator treated above. \medskip

In view of the fact that in the spin Dirac operator action, the scalar action of each $e_j$ is coupled with the Clifford action of $c(e_j)$, 
we conclude that leading order terms of the spin Dirac operator $D$ are even if $m$ is even and $(m-b)$ is even. 
Evenness of the higher order terms $W$ follows from their explicit structure, cf. the proof of \cite[Lemma 2.2]{Albin-Jesse}.
Now let us consider the case of $m$ odd. In this case the spinor bundle arises locally as the half spinor bundle
of $M\times \R$. The tangent vector coming from the additional $\R$ component is assigned even parity.
Repeating the arguments verbatim yields evenness of the spin Dirac operator $D$ if $(m-b)$ is even as well. 
\end{proof}

We conclude the section with noting that the condition $(m-b)$ to be even in the even odd 
classification of the signature and the spin Dirac operators, is naturally satisfied in 
the setting of algebraic varieties with edge singularities. 
This is of course due to the fact that all singularities
must be of even real codimension, $\C^n\cong \R^{2n}$.

\section{Short time asymptotic expansions of the heat traces}\label{trace-section}

In this section we continue in the general setting of an
edge manifold $M$ with an admissible incomplete edge metric $g$; we do not impose any 
isospectrality condition, cf. Remark \ref{not-isospectral}.

\subsection{Trace as an integral of the Schwartz kernel}
This subsection is concerned with establishing the trace class property
for the operators $e^{-t D^2}$ and $D e^{-t D^2}$ and deriving an integral 
formula for their traces, a version of  Lidskii's theorem in this singular setting. 
We present an argument that 
can then be adapted in the non-compact setting of Galois coverings below. 
\medskip

Let $E$ be any Hermitian vector bundle and let, as usual, $L^2(M,E)$ denote the 
space of square-integrable sections of $E$ with an inner product defined
by the Riemannian metric $g$ on $M$ and the Hermitian fibrewise metric $h$ of $E$. 
We shall also consider the Hilbert space  $L^2(M\times M, E \boxtimes E^*)$.
We recall the following central definitions (cf. Shubin \cite[\S 2.20]{Shubin2})
\begin{defn}\label{HS-definition}
Let $\{e_\alpha\}_{\A \in \mathcal{J}}$ denote some orthonormal basis of $L^2 (M,E)$
and write $\| \cdot \|$ for the norm on $L^2 (M,E)$. 
A  bounded operator $A$ on $L^2(M, E)$ is said to be a Hilbert-Schmidt operator if
\begin{equation}\label{HS-condition}
\sum_{\A \in \mathcal{J}} \| A e_\alpha \|^2< +\infty
\end{equation}
We denote the set of Hilbert-Schmidt operators by $\calC_2(M,E)$.
\end{defn}

It is easy to see that the left hand side of 
\eqref{HS-condition} is independent of the
choice of the orthonormal basis.
The following are classic results:

\begin{itemize}
\item $\calC_2(M,E)$ is an ideal; \medskip

\item If  $K_A \in L^2(M\times M, E \boxtimes E^*)$ 
denotes the Schwartz kernel of 
a bounded operator $A$, then $A\in \calC_2(M,E)$ 
if and only if $K_A \in  L^2(M\times M, E \boxtimes E^*)$;
\medskip

\item $\calC_2(M,E)$ can be given the 
structure of a Hilbert space  by defining the
Hilbert-Schmidt inner product: if $B, C \in \calC_2(M,E)$, then 
\begin{align}
\langle B, C \rangle_{2} := \iint_{M\times M} \textup{tr}_p 
\left( K_B(p,q) \overline{K_{C^*}(q,p)} \right)
\, \textup{dvol}_g (q) \textup{dvol}_g (p).
\end{align}
\end{itemize}
We denote the corresponding Hilbert Schmidt norm by $\| A \|_2$ 
\begin{defn}\label{trace class-definition}
The ideal $(\calC_2(M,E))^2 := \calC_2(M,E) \circ \calC_2(M,E)$ is, by definition, the ideal of trace class operators. Thus 
a bounded operator $A$ is trace class if it is a finite linear combination of products of
Hilbert-Schmidt operators.
We denote the space of trace class operators by $\calC_1(M,E) \subset \calC_2(M,E)$ and 
equip it with the trace norm $\| A \|_1 := \Tr |A|$, where the trace functional $\Tr$ is defined 
on $A = B \circ C$,  $B,C\in \calC_2(M,E)$, as
\begin{equation}\label{trace-definition}
\begin{split}
\Tr(A) := \langle B, C^* \rangle_{2} 
&= \iint_{M\times M} \textup{tr}_p 
\left( K_B(p,q) \overline{K_{C^*}(q,p)} \right)
\, \textup{dvol}_g (q) \textup{dvol}_g (p) \\ 
&= \iint_{M\times M} \textup{tr}_p 
\left( K_B(p,q) K_{C}(q,p) \right)
\, \textup{dvol}_g (q) \textup{dvol}_g (p).
\end{split}
\end{equation}
\end{defn}

\begin{remark}\label{HS-remark}
The aim of this subsection is to establish an integral representation of the 
trace for the operators $e^{-t D^2}$ and $D e^{-t D^2}$ in terms of their 
Schwartz kernels at the diagonal. While this might seem obvious by trying to
replace the inner integral in \eqref{trace-definition}
by $K_A(p,p)$, some care is necessary as explained e.g. 
in \cite[\S 2.21]{Shubin2}. The important issue is that the Schwartz kernel of a trace class operator 
$A = B \circ C$ with $B, C \in \calC_2(M,E)$ is given only almost everywhere by
\begin{equation}\label{product}
K_A(p,q)= \int_{M} K_B(p,s) K_{C}(s,q)
\textup{dvol}_g (s).
\end{equation}
In order to claim this identity at the set $\{p=q\}$ of measure zero, 
one needs an additional assumption of continuity of $K_A$
at the diagonal. Otherwise we cannot in general replace the inner integral in \eqref{trace-definition}
by $K_A(p,p)$ and present the trace of $A$ solely by an integral 
of its Schwartz kernel over the diagonal.
\end{remark}

We proceed by establishing the trace class property for 
$e^{-t D^2}$ and $D e^{-t D^2}$.
 
\begin{proposition}\label{lidskii-theorem2}
Let $D$ be an allowable Dirac-type operator on an incomplete edge space $(M,g)$. 
Assume that $D$ satisfies the geometric Witt condition.
Let $f\in \mathcal{S}(\RR)$ be a rapidly decreasing function
and let $f(D)$ the operator defined through functional calculus by the
unique self-adjoint extension of $D$. Then $f(D)$ and in particular
$e^{-t D^2}$ and $D e^{-t D^2}$ are trace class.
\end{proposition}

\begin{proof} We proceed by proving the statement first for a 
general $f\in \mathcal{S}(\RR)$ and then provide an alternative argument
for the particular examples $e^{-t D^2}$ and $D e^{-t D^2}$.
Note by Proposition \ref{ess-sa-allowable} that symmetric lower order perturbations $D= D^E_{\rm geo} + P$ 
of essentially self-adjoint geometric operators $D^E_{\rm geo}$, satisfying the geometric 
Witt condition, are essentially self-adjoint as well with the same domain.
Hence, proceeding as in \cite[Theorem 7.6]{Cheeger-spaces} 
one proves, for any allowable Dirac-type operator $D$ 
\begin{align}\label{trace-class}
(\textup{Id} + D^2)^{-N} \ \textup{is trace class in} \ L^2
\end{align}
if  $N>\dim M/2$. For these $N$ we consider the function $g(z):= 
(1+z^2) f(z)$ and remark that it is a bounded function.
Thus $g(D)$ is a bounded operator.
Writing $f(z)=(1+z^2)^{-N} g(z)$ we then have that $f(D)= (1+D^2)^{-N} ((1+D^2)^{N}f(D))=
 (1+D^2)^{-N} g(D)
$ is the composition
of a bounded operator, $g(D)$, and of a trace class operator, $(1+D^2)^{-N} $. Thus $f(D)$ is trace class
and the statement is proved.
\medskip

For the particular examples $e^{-t D^2}$ and $D e^{-t D^2}$ 
the trace class property can alternatively be concluded, assuming for simplicity the isospectrality
condition for the metrics on the links, 
by studying their Schwartz kernels as follows.
First of all let us consider the index set $E=(E_\rf, E_\lf)$ for the heat kernel
of $e^{-t D^2}$, lifted to the heat space $\mathscr{M}^2_h$, at the right and left
boundary faces. By symmetry of the heat kernel it suffices to discuss $E_\rf$ only.
In view of the microlocal heat kernel construction as in \cite{MazVer}, the index set can be written as 
\begin{align*}
E_\rf \equiv E^0_\rf + \N_0 
:= \{ \gamma + n \mid \gamma \in E^0_\rf, n \in \N_0 \}
\end{align*} where the integer numbers in $\N_0$ are set to begin with $0$, 
and $E^0_\rf$ is the index set of the initial parametrix defined by the heat kernel 
$H_{\mathscr{C}(F)}$ of the model cone $\mathscr{C}(F)$. By construction, the asymptotic expansion 
of $H_{\mathscr{C}(F)}$ at rf gets annihilated under the Laplacian $\Delta_{\mathscr{C}(F)}$
and hence also under the normal operator $N(D)$, obtained by lifting $D$ to $\mathscr{M}^2_h$
and restricting its action to the front face ff. From here one easily concludes that the 
index set of $D e^{-t D^2}$ at rf is given again by $E_\rf$. \medskip

From the microlocal description of the Schwartz kernel for $e^{-t D^2}$ and $D e^{-t D^2}$,
lifted to the heat space $\mathscr{M}^2_h$, as well as from positivity of the index sets 
$E=(E_\rf, E_\lf)$, we conclude that both Schwartz kernels are 
$L^2$ integrable on $M\times M$ for fixed times $t>0$. We point out that for fixed 
$t>0$ we study the heat kernel asymptotics away from front face in $\mathscr{M}^2_h$.
Consequently, the Schwartz kernels of $e^{-t D^2}$ and $D e^{-t D^2}$ are $L^2$ on 
$M \times M$ and hence by Definition \ref{HS-definition} the operators are Hilbert-Schmidt.
Using this and the semigroup property we can write
\begin{equation}\label{semigroup}
\begin{split}
De^{-t D^2} &= (De^{-\frac{t}{2} D^2})( e^{-\frac{t}{2} D^2}), \\
e^{-t D^2} &= (e^{-\frac{t}{2} D^2})( e^{-\frac{t}{2} D^2}).
\end{split}
\end{equation}
showing that the two operators are in fact trace class.
\end{proof}

Knowing that the operators $e^{-t D^2}$ and $D e^{-t D^2}$
are trace class does not necessarily yield an integral formula for their 
traces in terms of their respective Schwartz kernels, as pointed out in 
Remark \ref{HS-remark}. However, by the microlocal construction of the 
heat kernel in \S \ref{microlocal-section}, the Schwartz kernels of these 
operators are smooth along the diagonal for each fixed $t>0$ so that we can conclude
with the following central result. 

\begin{proposition}\label{lidskii-theorem}
Let $D$ be an allowable Dirac-type operator on an incomplete edge space $(M,g)$. 
Assume that $D$ satisfies the geometric Witt condition.
Then the heat operator $e^{-t D^2}$  for the  unique self-adjoint extension of $D^2$
at a fixed time $t>0$ is a trace class operator and for its trace
the following formula holds:
\begin{equation}\label{lidskii}
\Tr (e^{-t D^2})=\int_M \tr_p \, \mathcal{H} (t) (p,p) \textup{dvol}_g(p),
\end{equation}
with $\mathcal{H} (t)$ equal to the heat kernel. Similarly, 
at a fixed time $t>0$, the operator $De^{-t D^2}$ is trace class and 
its trace can be computed by the formula
\begin{equation}\label{lidskii-bis}
\Tr (De^{-t D^2})=\int_M \tr_p \mathcal{K} (t) (p,p) \textup{dvol}_g(p),
\end{equation}
where $\mathcal{K} (t)$ equal to the Schwartz kernel of $De^{-t D^2}$.
\end{proposition}

\begin{proof}
By the microlocal description of the heat kernel in 
\S \ref{microlocal-section} the Schwartz kernels of $e^{-t D^2}$ and $D e^{-t D^2}$
are smooth in $M\times M$ for $t>0$, and in particular they are continuous at the 
diagonal of $M\times M$. We conclude\footnote{The trace integral 
formula for $e^{-t D^2}$ can be obtained differently
from the trace-class property and positivity of the operator using Mercer theorem.} 
in view of \eqref{semigroup}
and \eqref{product} (we write $dp=\textup{dvol}_g(p)$ to simplify notation)
\begin{equation*}
\begin{split}
\Tr e^{-t D^2} &= \int_M \int_M  \textup{tr}_p \, \left(e^{-\frac{t}{2} D^2}\right)(p,q) 
\left( e^{-\frac{t}{2} D^2}\right) (q,p) 
\, dq \, dp  \\ &=\int_M  \textup{tr}_p \, \left(e^{-t D^2}\right)(p,p) \, dp, \\
\Tr De^{-t D^2} &= \int_M \int_M  \textup{tr}_p \, \left(De^{-\frac{t}{2} D^2}\right)(p,q) 
\left( e^{-\frac{t}{2} D^2}\right) (q,p) 
\, dq \, dp \\ &=\int_M  \textup{tr}_p \, \left(De^{-t D^2}\right)(p,p) \, dp.
\end{split}
\end{equation*}
In both identities the first equality follows from the semigroup 
property in \eqref{semigroup} and Definition \ref{HS-definition}
of the trace, and the second identity follows from continuity of the Schwartz
kernels at the diagonal, which yields an identity of the form \eqref{product}
at $\{p=q\}$. 
\end{proof}

We will now proceed with a derivation of an asymptotic expansion for the traces 
 $\textup{Tr} (e^{-tD^2})$  and $\textup{Tr} (D e^{-tD^2})$ 
for allowable even or odd Dirac-type operators $D$ under the geometric Witt assumption.

\subsection{The push-forward theorem of Melrose and the asymptotic 
expansion of the heat traces.}\label{push-forward-subsection}

Consider the diagonal $D'_0$ in $M_c \times M_c$ and set $D_0 = \RR^+ \times D'_0$. 
We denote by $D_h$ its lift to $\mathscr{M}^2_h$,
as pictured in Fig. 2. Denote by $D_{\td} = D_h \cap \td$, 
$D_{\ff} = D_h\cap \ff$ and $D_{\textup{cf}} = D_h \cap \rf\cap \lf$ 
the boundary faces of the lifted diagonal $D_h$. 

\begin{figure}[h]
\begin{center}
\begin{tikzpicture}
\draw  (0,0) -- (0,2);

\draw  (0,0) .. controls (0.5,-0.5) and (1,-1) .. (1,-2);
\draw  (0,0) .. controls (0.5,-0.1) and (1.2,-0.5) .. (1.6,-1.5);
\draw  (0,0) .. controls (0.5,0) and (1.5,0) .. (2,-1.3);

\draw  (1,-2) -- (3,-3);
\draw  (2,-1.3) -- (4,-1.3);

\draw  (1.7,-1.5) .. controls (1.5,-1.5) and (1.4,-1.6) .. (1.4,-1.8);
\draw  (1.7,-1.5) -- (3.6,-1.9);
\draw  (1.4,-1.8) -- (3.4,-2.3);

\draw  (3.6,-1.9) .. controls (3.5,-1.9) and (3.4,-2.2) .. (3.4,-2.3);
\draw  (3.6,-1.9) .. controls (3.7,-1.9) and (3.8,-2) .. (3.8,-2.1);

\draw  (3.6,-1.9) -- (3.6,0.1);
\draw  (3.6,0.1) -- (0,0.7);

\draw  (1,-2) -- (1.4,-1.8);
\draw  (2,-1.3) -- (1.8,-1.5);

\node at (4.5,-2) {\large{td}};
\node at (-0.5,1) {\large{cf}};
\node at (0.1,-1) {\large{ff}};
\node at (4,1) {$D_h \subset M^2_h$};

\end{tikzpicture}
\end{center}
\label{diagonal-picture}
\caption{The diagonal hypersurface $D_h$ in $\mathscr{M}^2_h$.}
\end{figure}
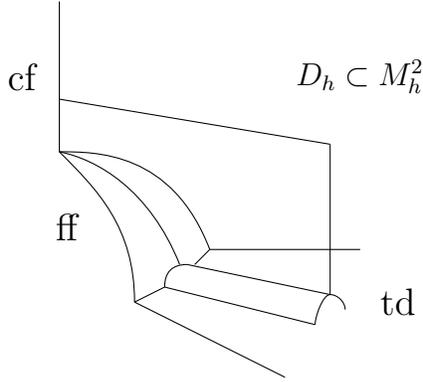

Let $\iota: D_h \hookrightarrow \mathscr{M}^2_h$ denote the natural embedding and $\iota^* H$ the 
restriction of $H\equiv e^{-tD^2}$ to the diagonal. 
If the fibres $(F_y, \kappa_y), y \in B$ are isospectral, $H \in \Psi^{2,0,\calE}_{+}
(M)$, the restriction $\iota^* H$ and its pointwise trace $\tr H$
are polyhomogeneous on $D_h$. Polyhomogeneity is straightforward 
and can be checked in local coordinates, but is also
a special case of the `push-forward theorem' in \cite{Mel2}. 
The index family of $\tr \iota^* H$ is given by 
\begin{align}\label{G}
G_{\td} = -m + 2\mathbb N_0,\ G_{\ff} = -m + 2\mathbb N_0,\ G_{\textup{cf}} = E_{\lf} + E_{\rf},
\end{align}
where the odd terms in $G_\td$ and $G_\ff$ vanish due to evenness conditions in Definitions 
\ref{def-even} and \ref{def-even-spin}. 
Note that functions on $\td$ which are odd with respect to the reflection $(S,U,Z) \mapsto -(S,U,Z)$ 
in coordinates \eqref{d-coord}, vanish on $D_h$ and do not contribute to the pointwise trace asymptotics.
Similarly, functions on $\ff$ which are odd with respect to $u \mapsto -u$ 
in coordinates \eqref{top-coord}, vanish on $D_h$ as well and thus also have vanishing trace. \medskip

In the general case, the fibres $(F_y, \kappa_y), y \in B$ are not isospectral. Still, the same statements
hold at ff and td, and the statement on the asymptotics near cf is replaced by integrability of $\iota^* H$ up to cf.
Note that the asymptotics at cf is irrelevant below, so we continue under the isospectrality assumption for notational
simplicity. \medskip

We now want to express the actual trace $\textup{Tr} H$ as a pushforward of some polyhomogeneous density.
Recall here that by Proposition \ref{lidskii-theorem} the heat operator is trace class and that 
the trace $\textup{Tr} H$ is given by the integral of the pointwise trace of the heat kernel along the diagonal.
\medskip

We employ a formalism for understanding pushforwards of polyhomogeneous distributions, known
as Melrose's Pushforward Theorem \cite{Mel2}. Let the map $\pi_c: D_h \to \RR^+$
be defined as the composition of the blowdown map $D_h \to \RR^+ \times M_c$ and the projection 
$\RR^+ \times M_c \to \RR^+$. Then, following straightforward computations outlined in \cite[\S 4]{MazVer}, 
$\Tr H$ is the pushforward of $\tr H$ by $\pi_c$:
\begin{equation}
\Tr H (t) \frac{dt}{t} = (\pi_c)_* \left( 2 \rho_\ff \rho_{\textup{cf}}\tr H \beta^*(x^f) \mu_b\right),
\label{bigtrace}
\end{equation}
where $\mu_b$ is a $b$-density on $D_h$ of the form $(\rho_{\td} \rho_{\ff} \rho_{\textup{cf}})^{-1}\mu_0$
for some density $\mu_0$ which is smooth up to all boundary faces and nowhere vanishing. 
The factor $\beta^*(x^f)$ comes from the volume form $\textup{dvol}(g)$ with equals 
$x^f dx dy dz$ up to some bounded function on $M$. Most importantly, 
Melrose's pushforward theorem asserts (cf. \cite[(4.3)]{MazVer}) that for the index family 
$\calG = (G_\td, G_\ff, G_{\textup{cf}})$ of $\rho_\ff \rho_{\textup{cf}}\tr \HF$,
the index family of the pushforward $\Tr H (t)$ is given by $F = \pi_c^{\flat}(\calG)$ which is defined as follows. 
If $G_{\td} = \{ (z_j,p_j)\}$ and $G_\ff = \{w_\ell, q_\ell)\}$, then $G'_{\td} = \{ (z_j/2,p_j)\}$, $G'_\ff = \{(w_\ell/2, q_\ell)\}$
and $F=G'_\td \overline{\cup} G'_{\ff}$ is the `extended union' of $G'_{\td}$ and $G'_\ff$ 
\begin{align}
\label{extended}
G'_\td \overline{\cup} G'_{\ff} = G'_\td \cup G'_\ff \cup \{(z, p + q + 1): \, (z,p) \in G'_\td,\ 
\mbox{and}\  (z,q) \in G'_{\ff} \}.
\end{align}
We refer to \cite[(42)]{Mel2} for a proof. Evaluating the index sets explicitly from \eqref{G},
we obtain in view of \eqref{extended}
\begin{align}\label{H-exp}
\Tr H(t) \sim \sum_{\ell=0}^{\infty}U_\ell t^{\ell-\frac{m}{2}}+
\sum_{\ell=0}^{\infty}V_\ell t^{\ell-\frac{b}{2}}+\sum_{\ell\in \mathfrak{I}}W_\ell t^{\ell-\frac{b}{2}}\log t, 
\end{align}
where $\mathfrak{I} = \varnothing$ if $(m-b)$ is odd and $\mathfrak{I}=\N_0$ if $(m-b)$ is even.
Here, the coefficients $U_\ell$ are defined as integrals over $M$ of pointwise traces of coefficients in the 
asymptotic expansion of $\beta^*H$ at td. These coefficients are functions on $M\times M$ and 
their pointwise traces are functions over $M$, depending pointwise on the jets of the Riemannian metric
and only on the rank of the twisting vector bundle. The coefficients $V_\ell$ and $W_\ell$ are 
integrals over $B$ of traces in $L^2(F_y \times F_y,E|_{F_y} \boxtimes E|_{F_y}; \kappa_y), y\in B$ of 
coefficients in the product type asymptotic expansion at the front face intersecting td. 
The $L^2(F_y \times F_y,E|_{F_y} \boxtimes E|_{F_y}, \kappa_y), y\in B$ traces of these coefficients are functions 
over $B$, which are global in the fibres of the edge fibration $\phi: Y \to B$ and local along the base $B$, i.e. depends
pointwise on a finite number of jets of the coefficients of $D$ in the radial and edge direction
at a given point at the edge singularity. All of this is explained in detail in 
\S \ref{coefficients-section}. \medskip

In case $M$ has several edge singularities $B_i, i\in I$ of dimension $b_i$, respectively, the asymptotic
expansion of $\Tr H(t)$ is of a similar form 
\begin{align}\label{H-exp}
\Tr H(t) \sim \sum_{\ell=0}^{\infty}U_\ell t^{\ell-\frac{m}{2}}+
\sum_{i \in I} \sum_{\ell=0}^{\infty}V_{\ell i} t^{\ell-\frac{b_i}{2}}+
\sum_{i \in I} \sum_{\ell\in \mathfrak{I}}W_{\ell i} t^{\ell-\frac{b_i}{2}}\log t. 
\end{align}

A similar analysis may be performed for the trace class operator $\textup{Tr} (D H)$.

\begin{thm}\label{exp1}
Let $D$ be an even or odd essentially self-adjoint allowable Dirac-type operator on an admissible edge manifold $(M,g)$
with $\dim M = m$ and a finite collection of edges $B_i, i\in I$ of dimension $b_i$, respectively. Then  \begin{equation}
 \begin{split}
&\Tr D H \sim_{t\to 0} \sum_{\ell = 0}^\infty A_\ell t^{\ell - \frac{m}{2}} 
+ \sum_{i \in I} \sum_{\ell = 0}^\infty B_{\ell i} t^{\ell - \frac{b_i}{2}}
+ \sum_{i \in I} \sum_{\ell \in \mathfrak{I}_i} C_{\ell i} t^{\ell - \frac{b_i}{2}} \log t, \ \textup{if $D$ is even}, \\
&\Tr D H \sim_{t\to 0} \sum_{\ell = 0}^\infty A_\ell t^{\ell - \frac{m}{2}} 
+ \sum_{i \in I} \sum_{\ell = 0}^\infty B_{\ell i} t^{\ell - \frac{b_i+1}{2}}
+ \sum_{i \in I} \sum_{\ell \in \N_0 \backslash \mathfrak{I}_i} C_{\ell i} t^{\ell - \frac{b_i+1}{2}} \log t, \ \textup{if $D$ is odd},
\end{split}
\end{equation}
where in both cases $\mathfrak{I}_i = \varnothing$ if $(m-b_i)$ is odd and $\mathfrak{I}_i=\N_0$ if $(m-b_i)$ is even.
\end{thm} 

\begin{proof} Assume for notational simplicity that there is a singe edge of dimension $b$ and that the metrics
on the links of the edge are isospectral. The general case is treated verbatim.
Let $D$ be even. Evenness of $D$ implies by definition $D H \in \Psi^{1,-1,\calE-1}_{+}(M)$. 
Consequently, the pointwise trace $\tr D H$ lifts to a polyhomogeneous
function on the diagonal $D_h$ with the index sets 
 \begin{equation}
 \begin{split}
&G_{\td} = (-1 -m + \N_0) \cap (-2-m+2\N_0) = - m+ 2 \N_0,\\
&G_{\ff} = (-1 -m + \N_0) \cap (-2-m+2\N_0) = -m +2\N_0,\\
&G_{\textup{cf}} = E_{\lf} + E_{\rf} - 1 >-1,
\end{split}
\end{equation}
Note that $G_\td$ corresponds to the interior expansion in \eqref{td-even}.
Let now $D$ be odd, so that by definition $D H \in \Psi^{1,-1,\calE-1}_{-}(M)$. 
Consequently, the pointwise trace $\tr D H$ lifts to a polyhomogeneous
function on the diagonal $D_h$ with the index sets 
 \begin{equation}
 \begin{split}
&G_{\td} = (-1 -m + \N_0) \cap (-2-m+2\N_0) = - m + 2 \N_0,\\
&G_{\ff} = (-1 -m + \N_0) \cap (-1-m+2\N_0) = -1-m+2\N_0,\\
&G_{\textup{cf}} = E_{\lf} + E_{\rf} - 1 >-1,
\end{split}
\end{equation}
The statement now follows by making the index set in 
\eqref{extended} explicit. 
\end{proof}

In the special case of $D$ being the Gauss-Bonnet, the signature or the spin Dirac operator,
we can obtain a much stronger statement.

\begin{thm}\label{exp2}
Let $D$ be an essentially self-adjoint geometric Dirac operator, i.e. either the Gauss-Bonnet,
the signature or the spin Dirac operator on an admissible edge manifold $(M,g)$ 
with $\dim M = m$ and a finite collection of edges $B_i, i\in I$ of dimension $b_i$, respectively. Then 
 $\Tr D H \equiv 0$ if $m$ is even, and the following short time 
asymptotic expansion holds if $m$ is odd
\begin{equation*}
\begin{split}
&\Tr D H \sim \sum_{\ell = \frac{(m+1)}{2}}^\infty A_\ell t^{\ell - \frac{m}{2}} 
+ \sum_{i \in I} \sum_{\ell = 0}^\infty B_{\ell i} t^{\ell - \frac{b_i}{2}}
+ \sum_{i \in I} \sum_{\ell \in \mathfrak{I}_i} C_{\ell i} t^{\ell - \frac{b_i}{2}} \log t, \ \textup{if $D$ is even}, \\
&\Tr D H \sim \sum_{\ell = \frac{(m+1)}{2}}^\infty A_\ell t^{\ell - \frac{m}{2}} 
+ \sum_{i \in I} \sum_{\ell = 0}^\infty B_{\ell i} t^{\ell - \frac{b_i+1}{2}}
+ \sum_{i \in I} \sum_{\ell \in \N_0 \backslash \mathfrak{I}_i} C_{\ell i} t^{\ell - \frac{b_i+1}{2}} \log t, \ \textup{if $D$ is odd}.
\end{split}
\end{equation*}
In both cases $\mathfrak{I}_i = \varnothing$ if $(m-b_i)$ is odd and $\mathfrak{I}_i=\{\ell \in \N_0\mid \ell \geq 
\frac{(m+1)}{2}\}$ if $(m-b_i)$ is even. 
Note that the class of even geometric Dirac operators $D$ 
on odd-dimensional manifolds includes the Gauss-Bonnet operator for any $b_i$, and the odd
signature as well as the spin Dirac operator if all $(m-b_i)$ are even. 
The class of odd geometric Dirac operators on odd dimensional
manifolds includes the odd signature operator if all $(m-b_i)$ are odd.
\end{thm} 

\begin{proof}
Let $E$ denote $S\otimes L$ in case $D$ refers to the spin Dirac operator, 
and $\Lambda^*T^*M \otimes L$ in case $D$ refers to the Gauss-Bonnet or the signature 
operator. Here $L$ is any flat Hermitian vector bundle induced by a unitary representation,
and $S$ the spinor bundle. Assume first that $m$ is even. 
Then, exactly as in the smooth compact setting, there exists 
a unitary $\phi: L^2(M,E) \to L^2(M,E)$, respecting the self-adjoint domain of $D$ 
such that $\phi D + D \phi = 0$. Consequently, the spectrum of $D$ is symmetric around 
zero and we conclude $\Tr D H = 0$. \medskip

Assume now that $m$ is odd. Then the spectrum of $D$ need not be symmetric any longer
and the trace asymptotics is non-trivial. Note that each coefficient $A_\ell$ comes from the heat 
kernel expansion at td, and hence is local in the sense that $A_\ell = \int_M a_\ell$, where at any given 
point $p\in M$ each $a_\ell(p)$ depends only on the jets of the Riemannian metric tensor $g$ at $p$. 
In particular, if an open neighborhood of $p\in M$ is isometrically identified with an 
open neighborhood of $p'$ in a smooth closed manifold $M'$, then $a_\ell(p)$ equals to the
corresponding coefficient $a'_\ell(p')$ in the heat trace asymptotics on $M'$. Using Getzler rescaling 
and local index theorem techniques on the closed manifold $M'$,
Bismut and Freed \cite{BiFr} have established vanishing of the local coefficients $a'_\ell$
for $(2\ell - m) < 1$, compare also Melrose \cite{Mel:TAP}. Consequently, $A_\ell$
vanish for $(2\ell - m) < 1$.\medskip

Now, by the definition \eqref{extended} of extended unions, coefficients
$C_{\ell i}$ arise if and only if the asymptotic expansion of $\Tr D H$ admits terms of the form
$A_{\ell+\frac{f}{2}}$ and $B_{\ell i}$ are both non-zero. In particular, $C_{\ell i}$ vanish for $(2\ell - m) < 1$
as well. This proves the statement. \end{proof}

We point out that without the assumption of $D$ being even or odd, all
statements in Theorems \ref{exp1} and \ref{exp2} continue to hold with the index $\ell$ 
being fractional, $\ell \in \N_0 / 2$.

\subsection{On locality of coefficients in the asymptotic expansions}\label{coefficients-section}

In this subsection we clarify to what extent the coefficients $A_\ell, B_\ell$ and $C_\ell$
for $\ell \in \N_0$ in the various asymptotic expansions of Theorems \ref{exp1} and
\ref{exp2} depend on the choice of a twisting flat vector bundle $E$. Consider $\iota^*(DH)$,
its asymptotics at the temporal diagonal and its product type asymptotic expansion near the lower corner of 
the front face in $D_h$
\begin{equation}\label{ptwise-asymptotics}
\begin{split}
&\iota^*(DH) \sim \sum a_{\ell} \, \rho_\td^{-m-1+2\ell} , \quad \rho_\td \to 0, \\
&\iota^*(DH) \sim \sum b_{kj} (\rho_\td \rho_\ff)^{-m-1} \rho_\td^{k} \rho_\ff^{j}, \quad \rho_\ff, \rho_\td \to 0,
\end{split}
\end{equation}
where the coefficients $a_{\ell}$ are smooth sections of $({}^{ie}\Lambda^*T^*M \otimes E) 
\otimes ({}^{ie}\Lambda^*T^*M \otimes E)^*$ in the Hodge de Rham setting, and sections of 
$(S \otimes E) \otimes (S \otimes E)^*$ in the spin setting. The coefficients $b_{kj}$ are smooth 
functions over $B$, coming from the intersection of ff and td, taking values in $C^\infty(F_y \times F_y,E|_{F_y} 
\boxtimes E^*|_{F_y})$ for each $y \in B$. Therefore each $b_{kj}(y)$ defines a trace class operator 
on $L^2(F_y,E|_{F_y}; \kappa_y)$. Taking pointwise traces in the fibres of $({}^{ie}\Lambda^*T^*M \otimes E) 
\otimes ({}^{ie}\Lambda^*T^*M \otimes E)^*$ yields smooth scalar functions $\tr a_{\ell}$ on $M$. 
Taking the functional analytic trace of trace class operators in $L^2(F_y,E|_{F_y}; \kappa_y)$ yields
smooth scalar functions $\tr b_{jk}$ on $B$. 
 \medskip

Short time asymptotics of $\Tr DH$ follows from either an application of the Pushforward theorem
by Melrose or in this specific situation by direct computations from the pointwise asymptotic expansions
in \eqref{ptwise-asymptotics}. In either approach it is clear that the coefficients $A_\ell$ are given as integrals 
over $M$ of multiples of $\tr a_{\ell}(p), p \in M$, which depend pointwise on a finite numbers of jets of the 
full symbol of $D$ at $p \in M$. We call such functions \emph{local over the interior $M$.}
\medskip

The coefficients $B_\ell$ and $C_\ell$ are given as integrals over $B$
of linear combinations of the traces $\tr b_{jk}(y), y\in B$, which are global
in the fibres of the edge fibration $\phi: Y \to B$ and local along the base $B$. This means that 
$\tr b_{jk}(y)$ depends in local coordinates $(x,y,z)$ near the edge singularity on 
a finite number of $x$-jets and $y$-jets of the coefficients of the operator $D$ 
at $y\in B$. We call such functions \emph{local over the edge $B$}.
\medskip

In both cases the dependence is functorial in the following sense: the full symbol
is not a coordinate invariant expression, however $\tr a_{\ell}(p)$
and $\tr b_{jk}(y)$ are independent of a particular choice of coordinates, since the 
Schwartz kernel $DH$ is invariantly defined and hence does not depend on the 
choice of local coordinates and local frames, compare \cite[Lemma 1.8.2]{Gilkey}.
 \medskip

Thus, the heat kernel expansion at td contributes coefficients to the short time asymptotics of 
$\Tr DH$ that are defined as integrals of local quantities over $M$. 
At the same time, the heat kernel expansion at ff contributes coefficients that are defined as 
integrals of local quantities over the edge singularity $B$.
This proves the following theorem 

\begin{thm}\label{trace-coefficients-flat}
Consider the short time asymptotic expansions in Theorems \ref{exp1} and
\ref{exp2}. Their coefficients are given by integrals of local quantities in the following sense.
\begin{enumerate}
\item the coefficient $A_\ell$ is an integral of multiples of the 
pointwise traces $\tr a_{\ell}(p)$ over $p \in M$, 
which are local over the interior $M$ in the following sense: $\tr a_{\ell}(p)$ depends 
functorially pointwise on a finite number of jets of the full symbol of $D$ at $p \in M$.
\item the coefficients $B_\ell$ and $C_\ell$ are integrals of linear combinations of the functional analytic traces $\tr b_{jk}(y)$
over $y\in B$, which are local over the edge $B$ in the following sense: each $\tr b_{jk}(y)$ is global
in the fibre $F_y$ of the edge fibration $\phi: Y \to B$ at $y\in B$ and local along the base $B$.
\end{enumerate}
\end{thm}

\subsection{Heat kernel on iterated cone-edge spaces}\label{iterated-heat-trace}

We remark that the microlocal  description of the heat kernel
carries over to the case of $D$ being a geometric Dirac operator 
on the higher depth iterated cone-edge spaces singled out 
in \S \ref{iterated-section}. More precisely, we assume that the bottom statum,
the most singular, is a point $P$ and the Riemannian metric in a neighbourhood of $P$ takes
the form 
$g\restriction \U = dx^2 + x^2 \kappa$ where the link $(F,\kappa)$ is a stratified
space with an iterated cone-edge metric $\kappa$ with the property that $D^2$ 
is the Hodge de Rham or the spin Laplacian and 
the tangential operator $D_F$ admits discrete spectrum (the Gauss Bonnet or the
signature operator associated to a Witt space $F$
of arbitrary depth with a rigid iterated edge metric would be an example).
\medskip

In this case, we may employ the discrete spectral decomposition of $D_F$ on the
cross section $F$ to write the heat kernel of $D^2$ explicitly in terms of Bessel functions.
The corresponding formula in the Hodge de Rham setting is provided in 
\cite[(3.9)]{MazVer}. The corresponding formula in the spin setting is provided 
above in \eqref{Bessel-heat-spin}. Since the Riemannian metric is assumed to be exact near the 
cone tip $P$, one does not require any microlocal calculus techniques to construct an exact 
heat kernel solution and may infer the heat kernel asymptotics straight from the explicit formulae.
In particular, \eqref{H-exp} continues to hold in that specific iterated cone-edge setting
for $D^2$ being the Hodge de Rham or the spin Laplace operator. Note that due to the explicit 
formulae, this cone edge singularity, the point $P$, contributes only the $B_0$ and does not contribute any $C_*$ coefficients.
\medskip

However an extension of an even and odd classification of geometric operators and 
definition of an even and odd subcalculus as in \S \ref{even-odd-section}, to higher
depth iterated cone-edge space is not obvious. While we may still try to mimic the even 
and odd subcalculus conditions at the front face ff, an extension of Theorems \ref{exp1} 
and \ref{exp2} depends on the explicit action of $D$ on the eigenspaces of the tangential 
operator $D_F$, which is now complicated due to $F$ being singular.

\section{Existence of eta invariants on singular edge spaces}

We are now in the position to discuss existence of eta invariants
in various configurations involving admissible edge singularities.
We prove existence of eta invariants for allowable essentially self-adjoint even or odd Dirac-type operators $D$
in the incomplete admissible edge setting. Essential self-adjointness is guaranteed by the geometric
Witt condition we have imposed throughout the paper. 
\medskip

In view of the asymptotic expansions obtained in the previous section, 
as well as the exponential decay of $\textup{Tr} \, (D e^{-tD^2})$ for large times, the eta-function 
\begin{align*}
\eta(D,s) = \frac{1}{\Gamma((s+1)/2)}
\int_0^\infty t^{(s-1)/2} \, \textup{Tr} \, (D e^{-tD^2}) dt
\end{align*}
extends to a meromorphic function on $\C$ with the residue at $s=0$
given by the $t^{-\frac{1}{2}}$ coefficient in the short time asymptotics of 
$\textup{Tr} \, (D e^{-tD^2})$, up to some universal proportionality factor. In particular, 
if the asymptotic expansion does not admit $t^{-\frac{1}{2}}$ terms, 
the eta invariant $\eta(D):=\eta(D,s=0)$ is well-defined. Hence just by looking at the 
asymptotics in Theorem \ref{exp1} and studying contributions to the $t^{-\frac{1}{2}}$
and the $t^{-\frac{1}{2}}\log t$ coefficients we obtain the following a priori statement.

\begin{prop}\label{eta-main} Let $(M,g)$ be an incomplete edge space with an admissible
edge metric.  Let $D$ be an allowable even or odd essentially self-adjoint Dirac-type operator. Then
\begin{enumerate}
\item[(i)] Assume that $m$ is even and $(D,b)$ are of same parity, i.e. 
either $D$ and $b$ are both even, or $D$ and $b$ are both odd.
Then the eta invariant $\eta(D)$ is well-defined\footnote{in these cases there are no
$t^{-1/2}$ coefficients in the asymptotics of $\Tr (DH)$.}. \medskip

\item[(ii)] Assume that $m$ is odd, and $(D,b)$ are of same parity. 
Then $\eta(D,s)$ has a simple pole at $s=0$ and the residue is an integral over $M$ of a 
term that is local over $M$ in the sense of Theorem 
\ref{trace-coefficients-flat} \footnote{in these cases there are no $B_*t^{-1/2}$ 
coefficients in the asymptotics of $\Tr (DH)$.}.
\medskip

\item[(iii)] Assume that $m$ is even and $(D,b)$ is of opposite parity, i.e. 
either $D$ is even and $b$ is odd, or $D$ is odd and $b$ is even.
Then $\eta(D,s)$ has a simple pole at $s=0$ and the residue 
is a sum of an integral over $M$ with an integrand that is local over $M$, 
and an integral over the edge $B$ with an integrand that 
is local over $B$ in the sense of Theorem \ref{trace-coefficients-flat}.\medskip

\item[(iv)] Assume that $m$ is odd and $(D,b)$ is of opposite parity. Then the eta function $\eta(D,s)$ 
may admit a second order pole singularity at $s=0$\footnote{in 
these cases we cannot exclude $C_*t^{-1/2}$ coefficients in the asymptotics of $\Tr (DH)$.}.
The Laurent coefficient of $s^{-2}$ in the Laurent expansion of $\eta(D;s)$ at $s=0$ 
is an integral over the edge $B$ with an integrand that is local over $B$ as in \textup{(iii)}.
The residue is of the same structure as the residue in \textup{(iii)}. 
\end{enumerate}

The case where the edge $B = \bigcup\limits_{i=1}^k B_i$ is a union of connected components $B_i$ of dimension $b_i$,
is dealt with a combination of the cases above. \medskip

\end{prop}

In the special case of $D$ being the signature or the spin Dirac operator,
we employ Theorem \ref{exp2} to obtain a much stronger statement; 
this improved statement will be crucial in discussing rho-invariants.

\begin{prop}\label{eta-main2} Let $(M,g)$ be an incomplete edge space with an admissible
edge metric.  Let $D$ be either the signature or the spin Dirac operator. Then
\begin{enumerate}
\item[(i)] the eta invariant $\eta(D)$ is identically zero if $m$ is even.
\item[(ii)] the eta function $\eta(D,s)$ has at most a first order pole singularity at $s=0$, if $m$ is odd.
Then the residue at $s=0$ is local over the edge $B$ in the sense of Theorem \ref{trace-coefficients-flat}.
\end{enumerate}
\end{prop}

\begin{proof}
The statement follows directly from Theorem \ref{exp2}. In case $m$ is even,
triviality of the eta function is a consequence of the spectrum being symmetric around zero.
In case $m$ is odd, we just observe that the  non-trivial contribution to the residue comes
only from $B_\ell$ coefficients that local over $B$ in the sense of Theorem \ref{trace-coefficients-flat}.
We may try to exclude the $B_*t^{-1/2}$ coefficient from the asymptotics of $\Tr DH$ by 
asking $(D,b)$ to be of the same parity. However, on an odd-dimensional manifold, 
the geometric Dirac operators are odd if $b$ is even and vice versa. Hence the residue
is a priori non-zero. 
\end{proof}

\section{Eta invariants on Galois coverings of admissible edge spaces}\label{Galois-edge-section}

\subsection{Geometric preliminaries}\label{subsect:geometric-pre}

Let $\overline{M}$ be a smoothly stratified pseudomanifold of arbitrary depth. 
Consider a Galois covering $\pi: \overline{M}_{\Gamma} \to \overline{M}$
with Galois group $\Gamma$ and  fundamental domain $\mathscr{F}_{\Gamma}$. There is a natural 
way to define a topological stratification on $\overline{M}_{\Gamma}$. 
Decompose $\overline{M}_{\Gamma}$ into the
preimages under $\pi$ of the strata in $ \overline{M}$. 
Surjectivity of $\pi$ ensures that each stratum in the covering is mapped 
surjectively onto the corresponding stratum in $ \overline{M}$. 
Since $\pi$ is  a local homeomorphism, it is 
straightforward to check that  $\overline{M}_{\Gamma}$ and 
its fundamental domain  are again topological 
stratified spaces. \medskip

In fact, more is true: by using these local homeomorphisms we can induce a \emph{smooth}
stratification on $\overline{M}_{\Gamma}$ by simply pulling it up from the base and this argument applies 
in the Whitney as well as the Thom-Mather cases.
It is important to point out that, by definition, the link of a point  
$\tilde{p}\in \overline{M}_{\Gamma}$ is
equal to the link of its image in the base.
Needless to say,  if  $\overline{M}_{\Gamma}$ is the universal covering 
space of $ \overline{M}$, the individual strata 
in $\overline{M}_{\Gamma}$ need 
not be the universal covering of the 
corresponding strata in the base. 
\medskip

We shall apply these remarks
to a depth-one smoothly stratified space. We denote by $\widetilde{M}$ the regular stratum of 
$\overline{M}_{\Gamma}$ and observe that it is a Galois covering of the regular stratum $M$ of 
$\overline{M}$ with fundamental domain $\mathscr{F}$ equal to the regular part of $\mathscr{F}_{\Gamma}$.
Let $g$ be an admissible incomplete edge metric on $M$. We can lift $g$ to the Galois covering
$\widetilde{M}$ where it becomes a $\Gamma$-invariant admissible incomplete 
edge metric $\widetilde{g}$. The edge singularity $\widetilde{B}$ of $\widetilde{M}$
is a Galois covering of $B$ with fundamental domain $\mathscr{F}_B$. 
Moreover, there are isometric embeddings of $\mathscr{F}$ into $M$, and of $\mathscr{F}_B$ into $B$ 
with complements of measure zero.

\subsection{$\Gamma$-Hilbert Schmidt and 
$\Gamma$-trace class operators}\label{subsect:von neumann}
  
In this section we introduce the notion of $\Gamma$-Hilbert Schmidt and 
$\Gamma$-trace class operators in full analogy to Definition \ref{HS-definition},
carried over to the present non-compact setting of Galois coverings.
Our primary sources here are Shubin \cite{Shubin2} and Atiyah \cite{Ati}.
\medskip

We first recall the notion of bounded equivariant linear operators on $L^2(\widetilde{M})$.
Let $\mathcal{N}_\Gamma$ be the Von Neumann algebra 
of the discrete group $\Gamma$; thus, by definition,
$\mathcal{N}_\Gamma$ is the weak closure of the algebra 
$\C\Gamma$ viewed as an algebra of bounded operators
on $\ell^2 (\Gamma)$ by convolution. There is a finite normal 
trace on $\mathcal{N}_\Gamma$ defined
as $\tau_{\, \Gamma} (A)= (A\delta_e,\delta_e)_{\ell^2}$. 
\medskip
  
The von Neumann algebra $\mathscr{A}_\Gamma (\widetilde{M})$ 
associated to the $\Gamma$-Galois covering 
$\widetilde{M}$ is defined, as usual,  as the 
algebra of bounded equivariant linear operators on $L^2(\widetilde{M})$. 
As in the closed case there is an isomorphism 
$L^2 (\widetilde{M})\simeq L^2 (\mathscr{F})\otimes \ell^2 (\Gamma)\equiv L^2 (M)\otimes \ell^2 (\Gamma)$.
Through this isomorphism the Von Neumann algebra $\mathscr{A}_\Gamma (\widetilde{M})$ 
can be shown to be naturally isomorphic
to $\mathcal{B}( L^2 (M))\otimes \mathcal{N}_\Gamma$ and 
we can define through this isomorphism  a  semifinite trace, denoted $\Tr_\Gamma$,
by  considering the tensor product $\Tr\otimes \tau_{\, \Gamma}$. 
Here, $\mathcal{B}( L^2 (M))$ denotes the space of bounded operators. 
Similar definitions can be given when there is a vector bundle 
$Q$ on $M$ or a $\Gamma$-equivariant vector bundle  $\widetilde{Q}$
on $\widetilde{M}$.
 We could now go on and give, abstractly, definitions for the ideal 
of $\Gamma$-Hilbert-Schmitd and $\Gamma$-trace class operators.

\medskip
However, in order to keep the theory of Von Neumann algebras to a minimum,
we prefer to use as  definitions
what are in fact characterizations of $\Gamma$-Hilbert Schmidt and 
$\Gamma$-trace class operators, using 
 Shubin \cite[\S 2.23 Theorems 1. and 3.]{Shubin2} and 
Atiyah \cite[\S 4]{Ati}.

\begin{defn}\label{HS-definition-galois}
Consider a $\Gamma$-equivariant Hermitian vector bundle $\widetilde{Q}$
on $\widetilde{M}$. Denote the space of square integrable sections of $\widetilde{Q}$
by $L^2(\widetilde{M}, \widetilde{Q})$. Consider the von Neumann algebra $\mathscr{A}_\Gamma 
(\widetilde{M}, \widetilde{Q})$ of bounded equivariant linear operators on 
$L^2(\widetilde{M}, \widetilde{Q})$. Then an operator $A \in \mathscr{A}_\Gamma 
(\widetilde{M}, \widetilde{Q})$ with Schwartz kernel $K_A$ is called \medskip
\begin{enumerate}
\item a $\Gamma$-Hilbert-Schmidt operator if $K_A \in L^2(\widetilde{M}\times \mathscr{F}, 
\widetilde{Q} \boxtimes \widetilde{Q}^*)$, or equivalently if $\phi A$ and $A \phi$
are Hilbert Schmidt in the sense of Definition \ref{HS-definition} for any bounded measurable
function $\phi$ with compact support in the resolution $\widetilde{M}_c$.
We denote the space of Hilbert-Schmidt operators by $\calC_2^\Gamma(\widetilde{M},\widetilde{Q})$.
One can show that 
$\calC_2^\Gamma (\widetilde{M},\widetilde{Q})$ is an ideal in $\mathscr{A}_\Gamma 
(\widetilde{M}, \widetilde{Q})$.

\medskip

\item a $\Gamma$-trace class operator if
$A = \sum B_j \circ C_j$ (finite sum) with 
$B_j, C_j \in \calC_2^\Gamma (\widetilde{M},\widetilde{Q})$. We denote the space 
of $\Gamma$-trace class operators by $\calC_1^\Gamma (\widetilde{M},\widetilde{Q}) \subset 
\calC_2^\Gamma (\widetilde{M},\widetilde{Q})$. Thus  $$\calC_1^\Gamma (\widetilde{M},\widetilde{Q}):=
(\calC_2^\Gamma (\widetilde{M},\widetilde{Q}))^2\subset 
\calC_2^\Gamma (\widetilde{M},\widetilde{Q})$$
We define for any $A \in \calC_1^\Gamma 
(\widetilde{M},\widetilde{Q})$ the $\Gamma$-trace in terms of the characteristic function $\phi$
of the fundamental domain $\mathscr{F}$ by 
\begin{equation}\label{trace-definition-galois}
\Tr_\Gamma(A) := \Tr (\phi A \phi) =
\iint\limits_{\mathscr{F} \times \widetilde{M}} \textup{tr}_p \, 
K_B(p,q) K_{C}(q,p) \, \textup{dvol}_{\widetilde{g}} (q) \, \textup{dvol}_{\widetilde{g}} (p).
\end{equation}
\end{enumerate}
These definitions are independent of the choice of the fundamental domain $\mathscr{F} $.
\end{defn}

\begin{remark}\label{HS-remark-galois}
As previously, one might want to replace the inner integral in \eqref{trace-definition-galois}
by $K_A(p,p)$ thus obtaining the $\Gamma$-trace of $A$ by integrating its Schwartz kernel at the
diagonal over the fundamental domain $\mathscr{F}$:
\begin{equation}\label{gammatrace}
\Tr_\Gamma(A) := \Tr (\phi A \phi)=\int_{\mathcal{F}} \textup{tr}_p \,  K_A (p,p)\textup{dvol}_{\widetilde{g}} (p)
\end{equation}
However, the Schwartz kernel of
$A = B \circ C$ with $B, C \in \calC_2(M,E)$ is given only almost everywhere by
\begin{equation}\label{product-galois}
K_A(p,q)= \int_{\widetilde{M}} K_B(p,s) K_{C}(s,q)
\textup{dvol}_{\widetilde{g}} (s).
\end{equation}
In order to claim this identity at the set $\{p=q\}$ of measure zero, 
one needs an additional assumption of continuity of $K_A$
at the diagonal. Otherwise we cannot generally replace the inner integral in 
\eqref{trace-definition-galois} by $K_A(p,p)$ and present the trace of 
$A$ solely by an integral of its Schwartz kernel.
\end{remark}

\subsection{Dirac operators }\label{subsect:dirac-covering}
Let $D$ be an allowable Dirac-type operator on $M$.
We denote the $\Gamma$-equivariant  lift of  $D$ to the Galois covering $\widetilde{M}$ by $\widetilde{D}$.
We have already observed that the link of a singular point
$\widetilde{p} \in \widetilde{M}$ is equal to the link of $\pi(\widetilde{p}) \in M$.
Consequently, the vertical operators induced by $D$ and $\widetilde{D}$ on the links coincide.
Since by Assumption \ref{Witt} the geometric Witt condition is satisfied for $D$, it is also satisfied for $\widetilde{D}$.
We briefly denote by $Q$ and $\widetilde{Q}$ the relevant bundles of Clifford modules
on which these operators act. For each $k\in \N$ we can consider the edge Sobolev spaces 
\begin{equation}\label{Gamma-sobolev}
H^k_e (\widetilde{M},\widetilde{Q})=\{u\in L^2_e (\widetilde{M},\widetilde{Q})\;:\; Pu\in L^2_e (\widetilde{M},\widetilde{Q})\;\;
\forall P\in {\rm Diff}^k_e (\widetilde{M},\widetilde{Q})^\Gamma\}
\end{equation}
where ${\rm Diff}^*_e (\widetilde{M},\widetilde{Q})^\Gamma$ 
is the algebra of $\Gamma$-equivariant edge differential operators
on $\widetilde{M}$ and where the subscript $e$ in $L^2_e$ refers 
to the complete edge metric $\widetilde{g}_e$ associated to $\widetilde{g}$.
Recall that the $L^2$-space for the incomplete edge metric $\widetilde{g}$, $L^2$, 
and $L^2_e$ are obtained one from the
other from the relation $L^2_e=x^{\dim M/2}L^2$.
The {\it incomplete} edge Sobolev spaces, denoted 
$H^k_{ie} (\widetilde{M},\widetilde{Q})$, are defined similarly: 
$H^k_{ie}:=x^{\dim M/2}H^k_e$. \medskip

Incomplete edge differential operators on stratified Galois coverings 
have been  already considered in \cite{signature-package}.
There it is proved that the signature operator  in the geometric Witt case, 
twisted by the Mishchenko bundle associated to $\pi: \overline{M}_{\Gamma} \to 
\overline{M}$ and to the reduced $C^*$-algebra $C^*_r \Gamma$, 
defines an index class in $K_* (C^*_r\Gamma)$, a result that should 
be regarded as a generalisation of the Fredholm property for $D$ itself.
The crucial remark in \cite{signature-package},
used again in \cite{Novikov}, is that the microlocal techniques used 
for $D$ applies equally well for the  Mishchenko-twisted version
of $D$ once we observe that the Galois covering over a 
distinguished neighbourhood $\R^b\times C(F)$ is trivial, see
the proof of Proposition 6.3 in \cite{signature-package}.
In this paper we are  interested in Von Neumann index theory but the crucial remark applies equally well. 
Thus, using the same ideas as in the proof of Proposition 6.3 in \cite{signature-package}
one can prove the following result (see also \cite{Cheeger-spaces}).
  
\begin{prop}\label{D-domains}
The Dirac operator $\widetilde{D}$ satisfies the following properties.
\begin{enumerate}
\item Its maximal domain is given by 
$\dom_{\max}(\widetilde{D}) \subset \bigcap\limits_{\epsilon>0} 
x^{1-\epsilon} H^1_{ie} (\widetilde{M},\widetilde{Q})$;
\item $\dom_{\min}(\widetilde{D})= \dom_{\max}(\widetilde{D})$ so 
that  $\widetilde{D}$ is essentially self-adjoint
\footnote{We will offer a different proof of essential 
self-adjointness of $\widetilde{D}$ below in \S \ref{ess-section}.}; \medskip
\item the domain of the unique self-adjoint  extension of $(\textup{Id} + \widetilde{D}^2)^{-N}$, \\ 
$N>\dim M$, is contained in $\bigcap\limits_{\epsilon>0} x^{1-\epsilon} 
H^{2N}_{ie} (\widetilde{M},\widetilde{Q})$
\end{enumerate}
\end{prop}

From there we conclude with the following 
statement on the trace class property. 

\begin{proposition}\label{tr-cl}
If $N>\dim M/2$, then $(1+\widetilde{D}^2)^{-N}$ is $\Gamma$-trace class. 
Consequently, for each rapidly decreasing function
$f$, the operator $f(\widetilde{D})$ is $\Gamma$-trace class. 
In particular, the operators $e^{-t \widetilde{D}^2}$ and 
$\widetilde{D} e^{-t \widetilde{D}^2}$ are $\Gamma$-trace class.
Moreover, if $\chi_x$ is the characteristic function of $[-x,x]\subset \R$
for any $x\geq 0$, then $E_x:= \chi_x(\widetilde{D})$ is $\Gamma$-trace class
as well.
\end{proposition}

\begin{proof}
From the third point above we deduce that
$\widetilde{D}$ has the property that 
\begin{align}\label{trace-class-gamma}
\phi (\textup{Id} + \widetilde{D}^2)^{-N})\psi \ \textup{is trace class in} \ L^2
\end{align}
for each $\phi$ and $\psi$ of compact support.
By the analogue in our context of \cite[Lemma 4.8]{Atiyah-VN} we infer that $(1+\widetilde{D}^2)^{-N}$ 
is $\Gamma$-trace class. The proof that $f(\widetilde{D})$ is $\Gamma$-trace class now proceed as in the 
proof of Proposition \ref{lidskii-theorem2}. Finally, note that we may write $\chi_x = f \circ  \chi_x$ for an 
appropriately chosen rapidly decreasing function $f$. Thus $E_x$ is $\Gamma$-trace class because 
$f(\widetilde{D})$ is $\Gamma$-trace class and $\chi_x (\widetilde{D})$ is bounded.
\end{proof}

As a consequence, the operators $e^{-t \widetilde{D}^2}$ and 
$\widetilde{D} e^{-t \widetilde{D}^2}$ are $\Gamma$-trace class.
However, we point out that the integral representation of their corresponding
traces requires an additional argument which is worked out below in \S 
\ref{heat-galois}.

\subsection{On essential self-adjointness}\label{ess-section}

Recall that by Assumption \ref{Witt} the geometric Witt condition is satisfied for $D$
and hence it is also satisfied for $\widetilde{D}$.
For the sake of completeness we offer a different proof of essential self-adjointness.

\begin{prop}
The operator $\widetilde{D}$ is essentially self-adjoint in $L^2(\widetilde{M}, Q)$.
\end{prop}

\begin{proof} We provide for brevity only a rough idea of the proof, since the arguments involved
are classical by now. First of all note that elements of $\dom_{\max}(\widetilde{D})$ admit a weak 
asymptotic expansion at the edges of $\widetilde{M}$. This can be done exactly as in Mazzeo 
\cite[Theorem 7.3]{Maz:ETO}, where the expansion is obtained using only Mellin transformation and
applies generally to a non-compact setting as well. \medskip

We then use the fact that elements in $\dom_{\max}(\widetilde{D})$ may be approximated 
in the graph norm by elements in the maximal domain that in addition are polyhomogeneous
in their asymptotics at the edge. This is done precisely as in the proof of Vertman 
\cite[Proposition 3.2]{Ver-Mooers}, cf. also Mazzeo-Vertman \cite{MazVer}. This in turn 
can be used to show that $\dom_{\min}(\widetilde{D}) = \dom_{\max}(\widetilde{D})$ assuming the
geometric Witt condition, cf. \cite[Lemma 2.4]{MazVer} and \cite[Proposition 3.2]{Ver-Mooers}.
\end{proof}

\subsection{Heat kernel on Galois coverings and 
$\Gamma$-trace integral formula}\label{heat-galois}

We define the heat operator $e^{-t\widetilde{D}^2}$ for $\widetilde{D}^2$ using functional calculus.
By Proposition \ref{tr-cl} we know that $e^{-t\widetilde{D}^2}$ is $\Gamma$-trace class.
Instead of developing a microlocal heat calculus for singular Galois coverings as in 
\cite{MazVer}, we employ the framework worked out by Lesch \cite{Lesch-singular}
and prove the following result.

\begin{prop}\label{trace-class-Galois}
Consider any functions $\phi, \psi$ that are smooth on the resolution of $M$, which is a
compact manifold $M_c$ with boundary $Y$. The resolution $\widetilde{M}_c$ of $\widetilde{M}$ is a
Galois covering of $M_c$. Lift $\phi, \psi$ to functions on $\widetilde{M}_c$ with compact support in the fundamental domain.
Then the operators $\phi e^{-t \widetilde{D}^2} \psi$ and $\phi (\widetilde{D} e^{-t \widetilde{D}^2}) \psi$ are trace 
class and, under the isometric embedding $\mathscr{F} \hookrightarrow M$, the trace norms satisfy the following
asymptotic behaviour for any $N \in \N$ and any $k\in \N$
\begin{equation}\label{trace-norm}
\begin{split}
&\| \phi (e^{-t \widetilde{D}^2} - e^{-tD^2}) \psi \|_{\textup{tr}} = O(t^N), \ t \to 0, \\
&\| \phi (\widetilde{D}^k e^{-t \widetilde{D}^2} - D^k e^{-tD^2}) \psi \|_{\textup{tr}} = O(t^N), \ t \to 0.
\end{split}
\end{equation}
\end{prop}

\begin{proof} 
We know that the fundamental domain $\mathscr{F}$ 
can be embedded  isometrically into the base $M$,
with complement of measure 0. Moreover  $D$ and 
$\widetilde{D}$ coincide over $M$ in the sense of \cite[Definition 2.8]{Lesch-singular}. 
Here, essential self-adjointness of $D$ and $\widetilde{D}$ is crucial in order to ensure that the condition 
\cite[(2.14)]{Lesch-singular} holds. Recall that 
\begin{align}
(\textup{Id} + D^2)^{-m-1} \ \textup{is trace class in} \ L^2(M, Q).
\end{align}
Thus $D$ satisfies the \emph{Schatten property} introduced in \cite[Lemma 2.5]{Lesch-singular}. 
We know that the lifted operator $\widetilde{D}$ satisfies a similar 
Schatten property, namely \begin{equation*}
\phi (\textup{Id} + \widetilde{D}^2)^{-m-1})\psi \ \textup{is trace class in} \ L^2(\widetilde{M}, Q)
\end{equation*}
Now the statement follows from \cite[Lemma 2.10]{Lesch-singular}.
\end{proof}

\begin{cor}\label{smooth-galois}
The Schwartz kernels of $e^{-t \widetilde{D}^2}$ and 
$\widetilde{D} e^{-t \widetilde{D}^2}$ are smooth\footnote{Note that we do not claim smoothness 
of the integral kernels if $\widetilde{M}$ is replaced with its resolution, which is 
a smooth manifold with boundary.} in $\widetilde{M}\times \widetilde{M}$.  
\end{cor}

\begin{proof}
The previous Proposition \ref{trace-class-Galois} in particular implies that for any 
$\phi, \psi \in C^\infty(\mathscr{F})$ with compact support away from the edge singularities,
the operator $\phi (\widetilde{D}^k e^{-t \widetilde{D}^2}) \psi$ is a bounded operator on
$L^2(U)$, where $U:=\supp \phi \cup \supp \psi$. Since $k\in \N$ is arbitrary we conclude
by elliptic regularity that $\phi (\widetilde{D}^k e^{-t \widetilde{D}^2}) \psi$ is smoothing and 
hence admits a smooth integral kernel. Consequently, the Schwartz kernels of 
$e^{-t \widetilde{D}^2}$ and $\widetilde{D} e^{-t \widetilde{D}^2}$ are indeed smooth
in the open smooth manifold $\mathscr{F} \times \mathscr{F}$ and hence also on
$\widetilde{M}\times \widetilde{M}$.
\end{proof}

We conclude the subsection with an integral representation of the $\Gamma$-trace in terms
of the Schwartz kernel for the operators $e^{-t \widetilde{D}^2}$ and 
$\widetilde{D} e^{-t \widetilde{D}^2}$. We proceed using the argument of 
Proposition \ref{lidskii-theorem} and establish the following result.

\begin{proposition}\label{lidskii-theorem-galois}
The operators $e^{-t \widetilde{D}^2}$ and 
$\widetilde{D} e^{-t \widetilde{D}^2}$ are $\Gamma$-trace class 
and their traces can be represented by integrals of their corresponding
Schwartz kernels
\begin{equation}
\begin{split}
\Tr_\Gamma \left(e^{-t \widetilde{D}^2}\right)&=
\int_{\mathscr{F}} \textup{tr}_p \, \left(e^{-t \widetilde{D}^2}\right) (p, p) \textup{dvol} (p) , \\
\Tr_\Gamma \left(\widetilde{D} e^{-t \widetilde{D}^2}\right)&=
\int_{\mathscr{F}}  \textup{tr}_p \, \left(\widetilde{D} e^{-t \widetilde{D}^2}\right) 
(p, p) \textup{dvol} (p).
\end{split} 
\end{equation}
\end{proposition}

\begin{proof}
Using the semi-group
property of the heat operator, we can write the operators $e^{-t \widetilde{D}^2}$ and 
$\widetilde{D} e^{-t \widetilde{D}^2}$ as products of two $\Gamma$-Hilbert Schmidt
operators
\begin{equation}
\widetilde{D} e^{-t \widetilde{D}^2} = (\widetilde{D} e^{-\frac{t}{2} \widetilde{D}^2})
( e^{-\frac{t}{2} \widetilde{D}^2}), \quad 
e^{-t \widetilde{D}^2} = (e^{-\frac{t}{2} \widetilde{D}^2})
( e^{-\frac{t}{2} \widetilde{D}^2}).
\end{equation}
In view of the integral formula in \eqref{trace-definition-galois}
and Remark \ref{HS-remark-galois}, it suffices to prove continuity of 
the Schwartz kernels of $e^{-t \widetilde{D}^2}$ and 
$\widetilde{D} e^{-t \widetilde{D}^2}$ at the diagonal of $\mathscr{F}\times \mathscr{F}$.
This follows from Corollary \ref{smooth-galois}.
\end{proof}

\subsection{Eta invariants on Galois coverings with edge singularities}\label{spectral-calculus}
Following  
Ramachandran \cite[(3.1.13)]{Ram} we know that there exists a tempered measure $m_\Gamma$
on $\R$ such that $\int_\R (1+|x|^\ell)^{-1} dm_\Gamma$ is finite for some positive integer $\ell \in \N$ and
\begin{align*}
&\textup{Tr}_\Gamma \, (\widetilde{D} e^{-t \widetilde{D}^2})
= \int_{\mathscr{F}} \textup{tr}_p(\widetilde{D} e^{-t \widetilde{D}^2}) (p, p) \textup{dvol} (p) 
= \int_\R x e^{-tx^2} dm_\Gamma(x),
\\ &\textup{Tr}_\Gamma \, (e^{-t \widetilde{D}^2}) = \int_{\mathscr{F}} \textup{tr}_p (e^{-t \widetilde{D}^2}) (p, p) \textup{dvol} (p)
= \int_\R e^{-tx^2} dm_\Gamma(x),
\end{align*}
The measure $m_\Gamma$ is defined by the functional  
$$\mathcal{S}(\R)\ni f\to \textup{tr}_\Gamma (f(\widetilde{D}))$$
which we know is  well defined. We can now estimate exactly as in Ramachandran 
\cite[(3.1.17)]{Ram} 
\begin{equation}\label{eta-est}
\begin{split}
&\left| \frac{1}{\Gamma(1/2)} \int_1^\infty t^{-1/2} \, \textup{tr}_\Gamma \, (\widetilde{D} e^{-t \widetilde{D}^2}) dt \right|
\\ &\leq \frac{1}{\Gamma(1/2)} \int_\R |x| e^{-x^2} \left( \int_1^\infty t^{-1/2} \, e^{-(t-1)x^2}  dt \right)  \, dm_\Gamma(x) 
\\ &\leq  \int_\R e^{-x^2} dm_\Gamma(x) = \textup{tr}_\Gamma \, (e^{- \widetilde{D}^2}).
\end{split}
\end{equation}
We can then associate to $\widetilde{D}$ an $\Gamma$ eta function 
\begin{align*}
\eta_\Gamma (\widetilde{D},s) = \frac{1}{\Gamma((s+1)/2)}
\left( \int_0^1 +  \int_1^\infty \right) t^{(s-1)/2} \, \textup{tr}_\Gamma \, (\widetilde{D} e^{-t \widetilde{D}^2}) dt,
\end{align*}
where the integral over $[1,\infty)$ is bounded in a neighborhood of $s=0$ by the 
estimate \eqref{eta-est}. The integral over $(0,1)$ is well-defined for $\Re(s) > (m+1)/2$
and admits a meromorphic extension to $\C$ by the short time asymptotics of 
$\textup{tr}_\Gamma \, (\widetilde{D} e^{-t \widetilde{D}^2})$ which follows from
Proposition \ref{trace-class-Galois} and Theorem \ref{exp1}. This yields the following a priori statement,
which parallels the statement of Propositions \ref{eta-main} and \ref{eta-main2}.

\begin{prop}\label{Galois-main-eta} 
\begin{enumerate}
\item[(i)] Assume that $m$ is even, and $(\widetilde{D},b)$ are of same parity,
i.e. either $\widetilde{D}$ and $b$ are both even, or $\widetilde{D}$ and $b$ are both odd.
Then the $\Gamma$-eta invariant $\eta_\Gamma(\widetilde{D})$ is well-defined.  \medskip

\item[(ii)] Assume that $m$ is odd, and $(\widetilde{D},b)$ are of same parity.
Then $\Gamma$-eta function $\eta_\Gamma(\widetilde{D},s)$ has a simple pole at $s=0$ with the residue
being local over $M$ in the sense of Theorem \ref{trace-coefficients-flat}.
\medskip

\item[(iii)] Assume that $m$ is even and $(\widetilde{D},b)$ is of opposite parity, i.e. 
either $\widetilde{D}$ is even and $b$ is odd, or $\widetilde{D}$ is odd and $b$ is even.
Then the $\Gamma$-eta function $\eta_\Gamma(\widetilde{D},s)$ has a simple pole at $s=0$ 
with the residue being a sum of an integral over $\mathscr{F}$ with an integrand that is local over $\mathscr{F}$,
and an integral over $\mathscr{F}_B$ with an integrand that is local over the edge $\mathscr{F}_B$
in the sense of Theorem \ref{trace-coefficients-flat}. \medskip

\item[(iv)] Assume that $m$ is odd and $(\widetilde{D},b)$ is of opposite parity.
Then the $\Gamma$-eta function $\eta_\Gamma(\widetilde{D},s)$ 
may admit a second order pole singularity
at $s=0$. The Laurent coefficient of $s^{-2}$ in the Laurent expansion of $\eta_\Gamma(\widetilde{D},s)$ at $s=0$ 
is an integral over $\mathscr{F}_B$ with an integrand that is local over the edge $\mathscr{F}_B$ as in \textup{(iii)}.
The residue is of the same structure as the residue in \textup{(iii)}. \end{enumerate}

The case where the edge $B = \bigcup\limits_{i=1}^k B_i$ is a union of connected components $B_i$ of dimension $b_i$,
is dealt with a combination of the cases above. \medskip

\end{prop}

In the special case of $\widetilde{D}$ being the signature or the spin Dirac operator,
we employ Proposition \ref{trace-class-Galois} and Theorem \ref{exp2} to obtain a much stronger statement; 
as for the case of Proposition \ref{eta-main2}, this improved statement will be important when
we will discuss the Cheeger-Gromov rho invariant.

\begin{prop}\label{eta-main-Galois2}
Let $\widetilde{D}$ be either the signature or the spin Dirac operator. Then
\begin{enumerate}
\item[(i)] the $\Gamma$-eta invariant $\eta_\Gamma (\widetilde{D})$ is identically zero if $m$ is even.
\item[(ii)] the $\Gamma$-eta invariant $\eta_\Gamma (\widetilde{D},s)$ has at most a first order pole singularity at $s=0$, if $m$ is odd.
Then the residue at $s=0$ is local over $B$ in the sense of Theorem \ref{trace-coefficients-flat}.
\end{enumerate}
\end{prop}

\section{The index theorem for manifolds with edge singular boundary}\label{index-section}

Let $(M,g)$ be a compact incomplete manifold with 
edge singularities and an admissible Riemannian metric.
Assume $M$ is boundary of another compact admissible edge manifold $X$.
We  equip $X$ with a product metric $(du^2+g)$ near the boundary, where 
$u\in [0,\varepsilon)$ is the inward normal coordinate. 
Consider the closed double manifold $X_d = X\cup_M (-X)$. By the product assumption, 
the Riemannian metric on $X$ extends smoothly to a Riemannian metric on the closed double. 
\medskip

Consider an allowable Dirac-type operator $\mathbb{D}$ acting between sections of Hermitian vector
bundles $E$ and $F$ over $X$. Assume that over the collar $M\times [0,\varepsilon)$ of the boundary, 
$\mathbb{D}$ takes a special form 
\begin{align}\label{boundary}
\mathbb{D} = \sigma \left( \frac{\partial}{\partial u} + D \right),
\end{align}
where $u\in [0,\varepsilon)$ is the inward normal coordinate and $\sigma$ is a bundle
isomorphism $E\restriction M \to F \restriction M$. Here, the tangential operator $D$ is a formally self-adjoint 
operator acting on sections of $E \restriction M$. Consider its closed 
double $X_d$ and, assuming that the vector bundles $E$ and $F$ are product near the boundary, extend the bundles $(E,F)$ and the Dirac operator 
$\mathbb{D}$ canonically to $X_d$. Assume that the Dirac operator $\mathbb{D}$
extended to the closed double $X_d$ satisfies the geometric Witt condition
and is even or odd. We will continue under this assumption throughout the paper. 

\subsection{Index theory for manifolds with edge singular boundary}

We fix a closed domain of $\mathbb{D}$ by putting Atiyah-Patodi-Singer boundary 
conditions $P_+(D)$ at $\partial X=M$, defined in terms of the positive spectral projection of $D$. 
We do not need to impose boundary conditions at the singular strata of $X$ due to the geometric Witt condition. 
More precisely, we single out the subspace of elements in $\dom_{\max}(\mathbb{D})$ that are 
smooth up to the boundary. We denote them by $\dom^\infty_{\max}(\mathbb{D})$. 
This defines the core domain
\begin{align}\label{core-domain}
\dom^c_+(\mathbb{D}) := \{u \in \dom_{\max}^\infty(\mathbb{D}) \mid (P_+(D)) u|_{\partial X} = 0\}.
\end{align}
We fix the closed extension of $\mathbb{D}$ with domain $\dom(\mathbb{D})$ defined as the 
graph-closure of the core domain $\dom^c_+(\mathbb{D})$ in $L^2(X, E^*\oplus F)$.

\begin{prop}\label{index-trace}
$\mathbb{D}$ is Fredholm with index 
\begin{align*}
\textup{index} \, \mathbb{D} = \textup{Tr} \, e^{-t\mathbb{D}^*\mathbb{D}} 
- \textup{Tr} \, e^{-t\mathbb{D}\mathbb{D}^*}. 
\end{align*}
\end{prop}

\begin{proof}
Consider first the model situation with $X_0:=M\times [0,\infty)$.
We define a Dirac operator $\mathbb{D}_0$ on $X_0$ by \eqref{boundary}.
As in case of $\mathbb{D}$, we fix a closed domain for $\mathbb{D}_0$ by putting 
Atiyah-Patodi-Singer boundary conditions $P_+(D)$ at $\partial X_0=M$. 
Since $D$ is essentially self-adjoint with discrete spectrum, the heat kernels $H_1$
and $H_2$ for $\mathbb{D}^*_0\mathbb{D}_0$ and $\mathbb{D}_0\mathbb{D}^*_0$,
respectively, can be constructed explicitly using spectral eigenspace decomposition on $M$, as in 
\cite[p. 10-11]{APSa}. Writing $(s,p), (s',p')\in M\times [0,\infty)$ for two copies of coordinates 
on $X_0$, we infer from \cite[(2.20)]{APSa} that the heat kernels $H_{j}(t, s, p, s', p'), j=1,2,$ are bounded by
\begin{align}\label{estimate-APS}
\frac{3}{2\sqrt{\pi t}} \left(\textup{tr} \, e^{-\frac{t}{2}D^2} (p, p) + 
\textup{tr} \, e^{-\frac{t}{2}D^2} (p', p')\right) e^{-\frac{(s-s')^2}{4t}}.
\end{align}
From here one easily concludes that for any smooth cutoff functions 
$\phi, \psi \in C^\infty(X_0)$ with compact support the kernels $\phi H_j \psi, j=1,2,$ are
Hilbert Schmidt in $L^2(X_0)$. \medskip

The heat kernels $H_1, H_2$ solve the heat equation for $\mathbb{D}^*\mathbb{D}$ 
and $\mathbb{D}\mathbb{D}^*$ near the boundary $\partial X$, respectively, and can be corrected 
to an exact solution using the heat kernels from the interior of $X$ in the usual way. The heat kernel 
parametrices in the interior of $X$ may be constructed microlocally in the same way as in 
\S \ref{microlocal-section}, and are Hilbert Schmidt for $t>0$ when multiplied from both sides 
with smooth cutoff functions that are compactly supported in the interior of $X$. \medskip

Then the heat operators for $\mathbb{D}^*\mathbb{D}$ and 
$\mathbb{D}\mathbb{D}^*$ must be Hilbert Schmidt
in $L^2(X)$ for $t>0$. Consequently the heat operators
for $\mathbb{D}^*\mathbb{D}$ and $\mathbb{D}\mathbb{D}^*$
are trace class by the semigroup property of the heat kernel. 
Thus both operators have discrete spectrum with finite multiplicities by the spectral 
theorem for compact operators. Hence $\mathbb{D}$ and $\mathbb{D}^*$ have finite dimensional kernels
and $\mathbb{D}$ is Fredholm. The formula for index now follows. 
\end{proof}

\subsection{Eta invariants on singular boundaries of edge spaces}\label{section-eta-bdry}

Now we employ an argument by Atiyah-Patodi-Singer to prove well-definement of the
eta invariant of the tangential operator $D$ appearing in the formula \eqref{boundary}, 
under certain additional assumptions. 

\begin{thm}\label{eta-exists}
Assume that $X$ satisfies the geometric Witt condition and that for the dimension $b$ of each edge in $X$, 
at least one of the numbers $(m+1-b)$ and $b$ is odd. Then the eta-invariant $\eta(D)$
of the tangential operator $D$ in \eqref{boundary} is well-defined. \footnote{
Note that this result is stronger than the one stated in Propositions \ref{eta-main} and \ref{eta-main2}.}
\end{thm}
 
\begin{proof}
We follow the argument of Atiyah-Patodi-Singer, which we also employ below in the proof
of Theorem \ref{index-main}. We still write $\mathbb{D}$ for the Dirac operator on the closed double.
Evenness or oddness of $\mathbb{D}$ guarantees that the heat kernels 
$\exp (-t\mathbb{D}^*\mathbb{D})$ and $\exp (-t\mathbb{D}\mathbb{D}^*)$
both lie in the even subcalculus. In \cite[(3.4)]{APSa} the authors define a function 
\begin{align}
F(q,t) := \textup{tr} \left(\exp (-t\mathbb{D}^*\mathbb{D}) - 
\exp (-t\mathbb{D}\mathbb{D}^*)\right) (q,q)
\end{align}
on the double manifold $q \in X_d$. 
The short time asymptotic expansion of $F(q,t)$ follows from \eqref{H-exp}
and in particular obtain ($b_i$ denote the dimensions of the various
edge singular strata in the double manifold $X_d$)
\begin{equation}\label{F-exp}
F(t) := \int_{X} F(q,t) \sim_{t\to 0} \sum_{\ell = 0}^\infty A'_\ell t^{\ell - \frac{m+1}{2}} 
+ \sum_i \sum_{\ell = 0}^\infty B'_\ell t^{\ell - \frac{b_i}{2}}
+ \sum_i \sum_{\ell \in \mathfrak{I}_i} C'_\ell t^{\ell - \frac{b_i}{2}} \log t.
\end{equation}
where $\mathfrak{I}_i = \varnothing$ if $(m+1-b_i)$ is odd and 
$\mathfrak{I}_i=\N_0$ if $(m+1-b_i)$ is even. Using the notation introduced in the 
proof of Proposition \ref{index-trace}, we also define according to \cite[(2.23)]{APSa}
\begin{align}
K(t):= \int_0^\infty \int_M \textup{tr} \left(\exp (-t\mathbb{D}_0^*\mathbb{D}_0) - 
\exp (-t\mathbb{D}_0\mathbb{D}_0^*)\right) (s,p, s,p) \, ds \, \textup{dvol}_g(p). 
\end{align}
Duhamel principle implies in view of Proposition \ref{index-trace} that for any 
$N\in \N$
\begin{align}\label{F-duhamel}
\textup{index} \, \mathbb{D} = \textup{Tr} \, e^{-t\mathbb{D}^*\mathbb{D}} 
- \textup{Tr} \, e^{-t\mathbb{D}\mathbb{D}^*} = K(t) + F(t) + O(t^N), \ t \to 0.
\end{align}
In this context, the central observation of Atiyah-Patodi-Singer in \cite[(2.25)]{APSa} is that 
\begin{align}
\int_0^\infty \left( K(t) + \frac{1}{2} \dim \ker D \right) t^{s-1} dt = 
- \frac{\Gamma \left(s+\frac{1}{2}\right)}{2s\sqrt{\pi}} \eta(D, 2s).
\end{align}
This relation implies that the residue of $\eta(D, 2s)$ at $s=0$
is given by twice the coefficient of $\log(t)$ in the short time asymptotics
of $K(t)$. The short time asymptotics of $K(t)$ follows from \eqref{F-exp}
by the relation \eqref{F-duhamel}. Consequently, we conclude
\begin{align}
\res\limits_{s=0} \eta(D,2s) = - 2 \sum_i C'_{\frac{b_i}{2}}.
\end{align}
The assumptions on the dimensions $b_i$ imply that the coefficients 
$C'_{\frac{b_i}{2}}$ are always trivial and hence the eta invariant
$\eta(D)$ is well-defined. 
\end{proof}

We conclude with a remark on the singularity structures allowed in $X$.
In fact, we may allow $X$ to admit a certain class of iterated higher order singularities. 
More precisely, $X$ may admit an isolated conical singularity with an exact conical metric 
and the link equal a simple edge space, or more generally be a stratified space
with an essentially self-adjoint Dirac operator $D$ with discrete spectrum as assumed in 
\S \ref{iterated-heat-trace}. Then, the heat kernel near such a singularity is given explicitly in terms of Bessel
functions, as explained in \S \ref{iterated-heat-trace}. Then \eqref{H-exp} and in particular \eqref{F-exp} 
still holds without any $C'_*$ coefficients coming from the conical singularity.  
Hence the proof of Theorem \ref{eta-exists} above carries over to this slightly extended setting, 
if the Dirac operator $\mathbb{D}$ is geometric, i.e. if $\mathbb{D}^*\mathbb{D}$ and $\mathbb{D}\mathbb{D}^*$ are the
direct sum components of the Hodge de Rham or the spin Laplacians. Note that general allowable Dirac-type operators
cannot be included, since in this case the corresponding heat kernels cannot be described explicitly in terms of Bessel
functions. 
 \medskip

This remark has the following non-trivial consequence.
First notice that such a space $X$ with such an additional isolated conical singularity and boundary given by $M$
always exists by setting $X=\mathscr{C}(M)$ where the cone extends from $\{0\}$ (the tip) to $3$, the boundary.
Thus  $X^{{\rm reg}}= (0,3] \times M$ with a smooth Riemannian metric given by $(du^2+ u^2 g)$
over $(0,1] \times M$ and by $(du^2 + g)$ over $(2, 3] \times M$.  The boundary of $X$ 
is precisely $M$ and the metric is by construction a product near the boundary and 
a straight conic metric in $(0,1]\times M$.\medskip

If the Dirac operator $D$ is equal to the signature or the spin operator, 
then there exists a geometric $\mathbb{D}$ on $X$. We need to argue in what way 
the geometric Witt condition of $D$ on $M$  implies the geometric Witt condition for $\mathbb{D}$
on $X$. First, however, we need to make sense of the geometric Witt condition
on a stratified space of depth $>$ 1. The cohomological and geometric Witt condition for the signature
operator are well-known for stratified spaces of depth $>$ 1, see \cite{signature-package};
the geometric Witt condition in the spin case for spin  stratified spaces of depth $>$ 1
would inductively require  the spectrum of the operator on all of the links
to have  empty intersection with the interval $(-1/2, 1/2)$. We shall not give this definition here (it 
clearly requires the precise definition of the domains along the links) but content ourselves with the present very simple case: thus we simply require that $(M,g)$, the link of the tip of the cone,
has a Dirac operator which is essentially self-adjoint and satisfies the geometric Witt condition 
and that the operator induced on the links associated to the singular points $\{(b,u)\in B\times (0,3]\}$,
these links are closed compact manifolds $F$ endowed with a Riemmannian metric,
also satisfy the geometric Witt condition.

\medskip

Now, if $M$ is odd dimensional, and satisfies the cohomological Witt condition, then so does $X$.
Indeed, as already remarked, $X$ has two singular strata, the tip of the cone, with link $M$ itself, and the set 
$\{(b,t), b\in B, t\in (0,3]\}$, with link $F$ (the link of $b\in B$ in $M$). By assumption $F$ is either
of odd dimension or of even dimension with no cohomology in middle degree; moreover $M$ is odd 
dimensional and therefore Witt as a depth-one stratified space. Hence, all the links of $X$
satisfy the cohomological Witt condition, as stated.
If $M$ satisfies the geometric Witt condition, 
then it certainly satisfies the cohomological Witt condition
and hence we know that $X$ satisfies the cohomological Witt condition;
thus, by suitably scaling the metric 
on the cross section, $X$ satisfies the geometric Witt condition too.  Summarizing: if $(M,g)$
satisfies the geometric Witt condition for the signature operator, then, up to scaling of $g$, so does $X$.
\medskip

In the setting of the spin Dirac operator $D$ we argue as follows. First of all, 
as already remarked, we need to assume that the singular stratum $B$ in $M$ is spin. We fix a spin structure
associated to the horizontal metric on $B$ and consequently get a unique spin structure
for the vertical tangent bundle of $\partial M \to B$ endowed with the vertical 
Riemmannian metric. By Lichnerowicz formula
and by  \cite[Theorem 1.3]{Albin-Jesse}, a non-negative
somewhere positive scalar curvature on $(M,g)$ ensures that the Dirac operator $D$
 on $(M,g)$ satisfies the geometric Witt condition and, in addition, that the kernel of $D$ is trivial, see \cite[Theorem 1.3]{Albin-Jesse}. 
Notice that the scalar curvature might be singular
as a function of the radial variable along the cone over a point $b \in B$.
Assume now that 
the scalar curvature on $(M,g)$ 
is not only positive but in fact  bounded away from zero:
there exists  a positive real number $k_M> 0$ such that the scalar curvature is greater or equal to  $k_M$
on every point of $M$. Then by a suitable scaling of the metric and by Lichnerowicz formula
we can arrange that the intersection of the $L^2$-spectrum of $D_M$ with $(-1/2,1/2)$ is empty.
This means that by making $k_M$  large enough through the scaling,  we can ensure 
the geometric Witt condition  for the link associated to  the tip of the cone in $X$ (which is $M$).
By our assumption, the geometric Witt condition is also satisfied on the links of the stratum
 $\{(b,t), b\in B, t\in (0,3]\}$. 
Note that the scaling of $g$ does not change the eta invariant of $D_M$.
 Consequently, we obtain the following 
corollary by using such specific $X$.

\begin{cor}\label{eta-exists-geometric}
Assume that $M$ is odd dimensional with each edge in $M$ of even dimension. 
Let $D$ be the signature operator, satisfying the 
geometric Witt condition. Then the eta-invariant $\eta(D)$  is well-defined. If $D$ is the spin Dirac operator
satisfying the geometric Witt condition and the scalar curvature on $M$ 
is positive and bounded away from zero, then the eta-invariant $\eta(D)$  is well-defined.
\end{cor}

\begin{proof}
In view of the preceding argument, it remains to check the dimensional 
restrictions in the neighborhood of boundary. Note that an edge in $M$
of dimension $b$ gives rise to an edge in the cone $X= (0,1] \times M$ 
of dimension $b+1$. Consequently, if $m$ is odd and $b$ is even, $(m-b)$ is odd and hence $(m+1-(b+1))$ 
and $(b+1)$ are odd. The dimensional 
conditions in Theorem \ref{eta-exists} are satisfied. 
\end{proof}

\subsection{Index formula for manifolds with edge singular boundary}

Recall, we consider an allowable Dirac-type operator $\mathbb{D}$ acting between sections of Hermitian vector
bundles $E$ and $F$ over $X$. We assume that over the collar $M\times [0,\varepsilon)$ of the boundary, 
$\mathbb{D}$ takes a special form \eqref{boundary}. We have proved above that $\mathbb{D}$, equipped
with APS boundary conditions is Fredholm and also established that the eta-invariant for 
its tangential operator $D$ is well-defined. We now can prove the following index formula. 

\begin{thm}\label{index-main}
Let $X$ be even-dimensional. Assume that $\mathbb{D}^2$ is even, so that the heat kernel of $\mathbb{D}^2$ on $X_d$ lies in the 
even subcalculus. Assume further that the dimension of each edge singularity in $X_d$ is odd.
Then the Fredholm index of the Dirac operator $\mathbb{D}$ on $X$ 
with APS boundary conditions is given by the index formula 
\begin{align}
\textup{index} \, \mathbb{D} = \left(\int_X a_0 + \int_B b_0 \right)- \frac{\dim \ker D+ \eta(D)}{2},
\end{align}
where $B$ denotes the union of edge singularities of $X$, 
the sum of the two integrals is the constant term in the short time asymptotic 
expansion of the trace of $\exp (-t\mathbb{D}^*\mathbb{D}) - \exp (-t\mathbb{D}\mathbb{D}^*)$.
Moreover we have the following characterization of the integrands $a_0$ and $b_0$.
\begin{enumerate}
\item $a_0$ is local over $X$ in the sense of Theorem \ref{trace-coefficients-flat}.
In particular, $a_0$ is the same as in the classical case. \footnote{For example if $\mathbb{D}$ is 
the spin Dirac operator acting between the sections of positive and negative spinor bundles, 
then $a_0$ is given by the maximal degree part of $\widehat{A}$-polynomial applied to the curvature  tensor of $X$.}
\medskip

\item $b_0$ is local over $B$ in the sense of Theorem \ref{trace-coefficients-flat}.
\end{enumerate}
\end{thm}

\begin{proof}
Theorem \ref{eta-exists} asserts the existence
of the eta-invariant $\eta(D)$. The relation \cite[(3.8)]{APSa} then implies that
\begin{align}
\textup{index} \, \mathbb{D} + \frac{\dim \ker D+ \eta(D)}{2}
\end{align}
equals the constant term in the short time asymptotics \eqref{F-exp}.
The statement now follows from a geometric characterization of coefficients 
similar to Theorem \ref{trace-coefficients-flat} with 
\begin{align}
\int_M a_0 := A'_{\frac{m+1}{2}}, \quad \int_B b_0 := \sum_i B'_{\frac{b_i}{2}}.
\end{align}
The integrand $a_0$ is the same as in the classical case. Indeed, $A'_{\frac{m+1}{2}}$ is local 
in the sense of Theorem \ref{trace-coefficients-flat} and hence the integrand 
$a_0(p)$ depends functorially only on a finite numbers of jets of the full symbol of $\mathbb{D}$ at $p \in M$. 
In particular, if an open neighborhood of $p \in M$ is isometrically identified 
with an open neighborhood of $p'$ in a smooth closed manifold $M'$, then $a_0(p)$ equals to the 
corresponding coefficient $a'_0(p)$ in the corresponding index formula on $M'$. 
\end{proof}

\section{Index theorem for Galois coverings with edge singular boundary}

As in the previous section we consider a compact incomplete manifold $(M,g)$ with 
edge singularities and an admissible Riemannian metric. Consider an essentially self-adjoint Dirac
operator $D$, satisfying the geometric Witt condition.
Assume $M$ is boundary of another compact admissible edge manifold $X$ with a product
type metric at the boundary. Consider an allowable Dirac-type $\mathbb{D}$ acting between sections of a Hermitian vector
bundle $E$ over $X$. We continue under the geometric 
Witt condition for $\mathbb{D}$ imposed in \S \ref{index-section}. \medskip

Consider a Galois covering $\widetilde{X}$ of $X$ with Galois group $\Gamma$. 
We denote its fundamental domain by $\mathscr{F}$. Its boundary is again a Galois covering
$\widetilde{M}$ of $M$. We denote the fibration of links over the edge singular stratum 
in $\widetilde{X}$ by $\widetilde{Y}$, which is again a Galois covering of the 
corresponding fibration in $X$ with fundamental 
domain $\mathscr{F}_Y$. \medskip 

We denote the lifts of $\mathbb{D}$ and $D$ to the coverings
$\widetilde{X}$ and $\widetilde{M}$ by $\widetilde{\mathbb{D}}$ and $\widetilde{D}$, 
respectively. Near the boundary $\partial \widetilde{X} = \widetilde{M}$, with the boundary 
defining function $\widetilde{u}$, the operator $\widetilde{\mathbb{D}}$ is related to $\widetilde{D}$ 
in terms of a bundle isomorphism $\widetilde{\sigma}$ on the lift of $E\restriction M$ to $\widetilde{M}$ by
\begin{align}
\widetilde{\mathbb{D}} = \widetilde{\sigma} \left( \frac{\partial}{\partial \widetilde{u}} + \widetilde{D} \right).
\end{align}

\subsection{Index theory for Galois coverings with edge singular boundary}

We may now proceed to establish $\Gamma$-Fredholmness of $\widetilde{\mathbb{D}}$
by modelling the argument of Roe \cite{Roe} and Ramachandran \cite{Ram}
in a singular setting. A related argument is given by Vaillant \cite{Vaillant}, however
rather with non-compact bases with cylindrical ends and hence with an $L^2$-version of an
index theorem on Galois coverings. However, in our setting a more straightforward 
argument is appropriate. \medskip

First of all we note that the concepts of $\Gamma$-Hilbert Schmidt and 
$\Gamma$-trace class operators introduced in \S \ref{subsect:von neumann}
carry over  verbatim to the setting of Galois coverings with regular boundary.
We may now state the following central proposition.

\begin{prop}\label{ideal-Galois}
The heat operators of $\widetilde{\Delta}_1:= \widetilde{\mathbb{D}}^* \widetilde{\mathbb{D}}$
and $\widetilde{\Delta}_2:= \widetilde{\mathbb{D}} \widetilde{\mathbb{D}}^*$, as well as the 
orthogonal projections $P_{\ker \widetilde{\mathbb{D}}}$ and $P_{\ker \widetilde{\mathbb{D}}^*}$ of $L^2(\widetilde{X})$
onto the kernel of $\widetilde{\mathbb{D}}$ and $\widetilde{\mathbb{D}}^*$ respectively, are $\Gamma$-trace class.
The $\Gamma$-equivariant Dirac operator $\widetilde{\mathbb{D}}$ is $\Gamma$-Fredholm,
i.e. admits a finite $\Gamma$-index
\begin{align}
\textup{index}_\Gamma \widetilde{\mathbb{D}} := \textup{Tr}_\Gamma (P_{\ker \widetilde{\mathbb{D}}}) 
-  \textup{Tr}_\Gamma (P_{\ker \widetilde{\mathbb{D}}^*}).
\end{align}
\end{prop}

\begin{proof}
We may argue as in Proposition \ref{index-trace} 
that the heat operators of $\widetilde{\Delta}_{1}$ and $\widetilde{\Delta}_2$
are $\Gamma$-trace class. Using the Duhamel principle it suffices to study the heat kernels
of the corresponding operators on the infinite half-cylinder $\widetilde{M}\times [0,\infty)$ 
and the double Galois covering $\widetilde{X}_d:= \widetilde{X} 
\cup_{\widetilde{M}} (-\widetilde{X})$. The latter space
is smooth across the join, since the Riemannian metric on 
$\widetilde{X}$ is assumed to be product near
the boundary $\widetilde{M}\times \{u=0\}$. \medskip

Arguing as in Theorem \ref{trace-class-Galois}, we find that the 
heat operators on the double Galois covering $\widetilde{X}_d:= \widetilde{X} \cup_{\widetilde{M}} 
(-\widetilde{X})$ are $\Gamma$-trace class for each fixed $t>0$ with a short time asymptotics as $t\to 0$
of the same structure as in the compact setting, cf. \eqref{H-exp}. \medskip

By Theorem \ref{trace-class-Galois} the heat operator 
of $\widetilde{D}^2$ on $\widetilde{M}$
is $\Gamma$-trace class, and consequently the heat 
operators of $\widetilde{\Delta}_{1}$ and $\widetilde{\Delta}_2$,
viewed as operators on the infinite half-cylinder $\widetilde{M}\times [0,\infty)$, are $\Gamma$-trace 
class near the boundary. Indeed, consider cutoff functions $\phi, \psi \in C^\infty_0[0,\varepsilon)$
such that $\textup{supp} \, \phi \subset \textup{supp} \, \psi$ and $\textup{supp} \, \phi \cap \textup{supp} (d\psi/du)
= \varnothing$. Extend $\phi, \psi$ trivially to functions on $X$ and let $\widetilde{\phi}, \widetilde{\psi}$ 
denote the lifts of $\phi, \psi$ to the Galois covering $\widetilde{X}$, respectively. Then, like in 
\eqref{estimate-APS}, compare also \cite[p.344-345]{Ram}, we may estimate for any $k\in \{1,2\}$
the $\Gamma$-trace of $\widetilde{\phi} \circ e^{-t\widetilde{\Delta}_k} \circ \widetilde{\psi}$ by the 
$\Gamma$-trace of $\widetilde{D}^2$ on $\widetilde{M}$. \medskip

Consequently $e^{-t\widetilde{\Delta}_k} \in \calC_1$ for any $k\in \{1,2\}$. 
We can write the projections $P_{\ker \widetilde{\mathbb{D}}}$ and $P_{\ker \widetilde{\mathbb{D}}^*}$ as compositions 
with the corresponding heat operators 
\begin{align}
P_{\ker \widetilde{\mathbb{D}}} = e^{-t\widetilde{\Delta_1}} \circ P_{\ker \widetilde{\mathbb{D}}}, 
\quad P_{\ker \widetilde{\mathbb{D}}^*} = e^{-t\widetilde{\Delta_2}} \circ P_{\ker \widetilde{\mathbb{D}}^2}.
\end{align}
Since the trace class operators $\calC_1$
form a two-sided ideal in the von Neumann algebra $\mathscr{A}_\Gamma (\widetilde{X})$, 
we conclude that the projections are 
$\Gamma$-trace class as well. Consequently the $\Gamma$-dimensions 
$\dim_\Gamma \ker \widetilde{\mathbb{D}}= \textup{Tr} (P_{\ker \widetilde{\mathbb{D}}})$
and $\dim_\Gamma \ker \widetilde{\mathbb{D}}^* = \textup{Tr} (P_{\ker \widetilde{\mathbb{D}}^*})$ are finite and 
$\widetilde{\mathbb{D}}$ is $\Gamma$-Fredholm with index
\begin{align}
\textup{index}_\Gamma \widetilde{\mathbb{D}} := \dim_\Gamma \ker \widetilde{\mathbb{D}}
-  \dim_\Gamma \ker \widetilde{\mathbb{D}}^*.
\end{align}
\end{proof}

Repeating \cite[Lemma 15.11]{Roe}, we deduce the McKean Singer index formula.

\begin{prop}  The Dirac operator $\widetilde{\mathbb{D}}$ is $\Gamma$-Fredholm
with the index $\textup{ind}_\Gamma \widetilde{\mathbb{D}} $ given by a $\Gamma$-super trace 
\begin{align*}
\textup{ind}_\Gamma \widetilde{\mathbb{D}} 
= \textup{Str}_\Gamma \left(e^{-t\widetilde{\mathbb{D}}^2}\right) 
:= \textup{Tr}_\Gamma \left(e^{-t(\widetilde{\mathbb{D}})^*\widetilde{\mathbb{D}}}\right) 
- \textup{Tr}_\Gamma \left(e^{-t \widetilde{\mathbb{D}}\widetilde{\mathbb{D}}^*}\right) 
\end{align*}
\end{prop}

\subsection{Eta invariants on boundaries of singular Galois coverings}\label{eta-galois-section}

Following the Atiyah-Patodi-Singer argument outlined in the 
proof of Theorem \ref{eta-exists}, we obtain the following corresponding
existence theorem on Galois coverings.

\begin{thm}\label{eta-exists-Galois}
Assume the Galois covering $\widetilde{M}\to M$ is boundary of some Galois covering
$\widetilde{X} \to X$ with $\partial X = M$ and product metric and bundle structures near the boundary.
Assume that for the dimension $b$ of each edge in $X$, 
at least one of the numbers $(m+1-b)$ and $b$ is odd. 
We finally assume that $\mathbb{D}$ is an allowable
even or odd Dirac-type operator. Then the $\Gamma$ eta-invariant $\eta_\Gamma(\widetilde{D})$
of the Dirac operator $\widetilde{D}$ on $\widetilde{M}$ is well-defined.
\end{thm}

\begin{proof}
We follow the argument of Atiyah-Patodi-Singer, cf. proof of Theorem \ref{eta-exists}. 
We  equip the Galois covering space $\widetilde{X}$ with a product metric $(du^2+g)$ near the boundary,
where $u\in [0,\varepsilon)$ is the inward normal coordinate. 
Consider the double manifold $\widetilde{X}_d:= \widetilde{X}
\cup_{\widetilde{M}} (-\widetilde{X})$. By the product assumption, 
the Riemannian metric on $X$ extends smoothly to a Riemannian 
metric on the double. The Dirac $\widetilde{\mathbb{D}}$ defines 
an operator on the double, which we denote as $\widetilde{\mathbb{D}}$ 
again. \medskip

Let us now consider the heat operators associated to $\widetilde{\mathbb{D}}$
as a Dirac operator on the double manifold $\widetilde{X}_d$. By the same argument as in Corollary 
\ref{smooth-galois}, we infer the  continuity of the Schwartz kernels for 
$\exp (-t\widetilde{\mathbb{D}}^*\widetilde{\mathbb{D}})$ and 
$\exp (-t\widetilde{\mathbb{D}}\widetilde{\mathbb{D}}^*)$ at the diagonal 
and hence a representation of their traces by integrals of their respective
Schwartz kernels at the diagonal along the fundamental domain, as in 
Proposition \ref{lidskii-theorem-galois}. \medskip

Evenness or oddness of $\widetilde{\mathbb{D}}$ guarantees that the heat kernels 
$\exp (-t\widetilde{\mathbb{D}}^*\widetilde{\mathbb{D}})$ and 
$\exp (-t\widetilde{\mathbb{D}}\widetilde{\mathbb{D}}^*)$
both lie in the even subcalculus. In \cite[(3.4)]{APSa} the authors define a function 
\begin{align}
F(p,t) := \textup{tr} \left(\exp (-t\widetilde{\mathbb{D}}^*\widetilde{\mathbb{D}}) - 
\exp (-t\widetilde{\mathbb{D}}\widetilde{\mathbb{D}}^*)\right) (p,p)
\end{align}
on the double manifold $\widetilde{X}_d$. The short time asymptotic expansion of $F(p,t)$ follows from 
Theorem \ref{trace-class-Galois} and \eqref{H-exp}
and in particular obtain ($b_i, i \in I$ denote the dimensions of the various
edge singular strata in the double manifold $\widetilde{X}_d$)
 \begin{equation}
 F(t) := \int_{\mathscr{F}} F(p,t) \sim_{t\to 0} \sum_{\ell = 0}^\infty A'_\ell t^{\ell - \frac{m+1}{2}} 
+ \sum_{i \in I} \sum_{\ell = 0}^\infty B'_\ell t^{\ell - \frac{b_i}{2}}
+ \sum_{i \in I} \sum_{\ell \in \mathfrak{I}_i} C'_\ell t^{\ell - \frac{b_i}{2}} \log t.
\end{equation}
where $\mathfrak{I}_i = \varnothing$ if $(m+1-b_i)$ is odd and 
$\mathfrak{I}_i=\N_0$ if $(m+1-b_i)$ is even. In the present context, the central observation 
of Atiyah-Patodi-Singer in \cite[(3.8)]{APSa}, worked out in detail in the proof of Proposition 
\ref{eta-exists} and carried over to the non-compact
setting of Galois coverings, is that
\begin{align}
\res\limits_{s=0} \eta_\Gamma (\widetilde{D},2s) = - 2 \sum_{i \in I} C'_{\frac{b_i}{2}}.
\end{align}
Assumptions on the dimensions $b_i$ imply that the coefficients 
$C'_{\frac{b_i}{2}}$ are always trivial and hence the eta invariant
$\eta_\Gamma (\widetilde{D})$ is well-defined. 
\end{proof}

\subsection{Atiyah-Patodi-Singer index theorem on Galois coverings}

We may now proceed as in \cite[Theorem  7.1.1]{Ram} to deduce an index
formula for the $\Gamma$-index of $\widetilde{\mathbb{D}}$, where the index theorem
on Galois coverings has been established following the standard argument once the diagonalization
theorem of Browder-Garding is employed.

\begin{thm}\label{index-main-galois}
Assume that $\widetilde{X}$ is an even-dimensional admissible
edge space with dimension $b$ of each edge singularity being odd.  Assume $\widetilde{\mathbb{D}}^2$ is even, 
so that the heat kernel of $\widetilde{\mathbb{D}}^2$ on 
the double Galois covering $\widetilde{X}_d$ lies in the 
even subcalculus. Then the $\eta_\Gamma(\widetilde{D})$ 
eta invariant is well-defined by restricting the eta function 
\begin{align*}
\eta_\Gamma(\widetilde{D},s) = \frac{1}{\Gamma((s+1)/2)}
\int_0^\infty t^{(s-1)/2} \, \textup{tr}_\Gamma \, (\widetilde{D} e^{-t\widetilde{D}^2}) dt
\end{align*}
to $s=0$ and is related to the $\Gamma$-Fredholm index of $\widetilde{\mathbb{D}}$ by the 
index formula
\begin{align}
\textup{ind}_\Gamma \widetilde{\mathbb{D}} = \left( \int_{\mathscr{F}} \widetilde{a}_0 +  \int_{\mathscr{F}_B} \widetilde{b}_0 \right)- 
\frac{\dim \ker_\Gamma \widetilde{D} + \eta_\Gamma(\widetilde{D})}{2},
\end{align}
where the sum of two integrals is the constant term in the short time asymptotic 
expansion of the super trace $\textup{Str}_\Gamma \left(e^{-t\widetilde{\mathbb{D}}^2}\right)$
on the double Galois covering $\widetilde{X} \cup_{\widetilde{M}} (-\widetilde{X})$.
Moreover we have the following characterization of the integrands $\widetilde{a}_0$ and $\widetilde{b}_0$:
\begin{enumerate}
\item $\widetilde{a}_0$ is local over the interior $\mathscr{F}$ in the sense 
of Theorem \ref{trace-coefficients-flat}. \medskip

\item $\widetilde{b}_0 $ is local over the edge $\mathscr{F}_B$ 
in the sense of Theorem \ref{trace-coefficients-flat}.

\end{enumerate}
\end{thm}

\section{Existence of the APS and Cheeger-Gromov rho invariants}

\subsection{Existence of the Atiyah-Patodi-Singer rho invariants for allowable Dirac operators}

Let us first consider the Atiyah-Patodi-Singer rho invariant on 
an edge singular manifold $(M,g)$ with an admissible edge metric.
Let $E$ and $F$ denote two flat vector bundles of same rank induced by two 
unitary representations $\A$ and $\beta$ of the fundamental group $\pi_1(\overline{M})$, 
respectively, where $\overline{M}$ denotes the compact stratified space including its singular stratum.
Each vector bundle yields a twisted Dirac operator, which we denote
by $D_E$ and $D_F$, respectively. Assume both satisfy the geometric Witt condition.
Consider the difference of the 
associated eta functions.
\begin{align}\label{APS-definition}
\rho_{\A - \beta}(s,D) = \eta(s,D_E) - \eta(s,D_F). 
\end{align}  
In case each individual eta function is regular at $s=0$,
the rho function is regular at $s=0$ and its value at zero,
the Atiyah-Patodi-Singer (APS) rho invariant, obviously 
exists. In case the individual eta invariants are not well-defined,
we are in the setup where the eta functions may admit a pole of order $2$ at $s=0$ and 
the Laurent coefficients of $s^{-2}$ and the residues 
at $s=0$ are induced by a coefficients $A_*, B_*$ and $C_*$ in the asymptotic expansions in Theorem 
\ref{exp1}, which are integrals of terms $a_*, b_*$ and $c_*$, respectively, that are local over the interior $M$ and the edge $B$. 
We indicate the a priori dependence of the coefficients on the
vector bundles by an additional upper index. By flatness of the vector 
bundles $E$ and $F$ over $\overline{M}$, these local coefficients coincide: $a^E_* = a^F_*$, 
$b^E_* = b^F_*$ and $c^E_* = c^F_*$. Hence their contribution
to the Laurent coefficient of $s^{-2}$ and the residue of $\rho_{(E,F)}(s,D)$ at $s=0$ cancels. We conclude

\begin{thm}\label{theo:rho-aps}
Assume $(M^m,g)$ is an edge manifold with an admissible edge metric
and an edge singularity at $B^b$. Let $D$ be an allowable Dirac-type operator.
Let $\A$ and $\beta$ be two unitary representations of the fundamental group 
$\pi_1(\overline{M})$. Assume that both twisted Dirac operators 
 satisfy the geometric Witt condition.
Then the rho function $\rho_{\A- \beta}(s,D)$ is well-defined at $s=0$.
The corresponding APS rho invariant is then defined by 
$$\rho_{\A- \beta}(D) := \rho_{\A- \beta}(0,D).$$
\end{thm}
 
 \subsection{Existence of the Atiyah-Patodi-Singer rho invariants for geometric Dirac operators}

If $D$ is the signature operator on a Witt space  or the spin Dirac operator
 for an incomplete edge metric of uniform positive scalar curvature, then 
 the statement in Theorem \ref{theo:rho-aps} can be improved.
 Observe first that if  $D$ is the signature operator, then
the cohomological Witt condition $H^{\frac{f}{2}}(F)=0$ holds for $D$ as well as for its twisted versions:
indeed, the twisting bundles
are trivial when restricted to the links  \footnote{recall that
the representations $\alpha $ and $\beta$ are representations of the fundamental group
of the singular pseudomanifold $\overline{M}$}. Once we have the cohomological 
Witt condition we can always  strengthen it to its geometric
version via scaling of the metric on the links of the edge fibration and rescaling of the Hermitian metric 
on the twisting bundle, so that the absolute value of the eigenvalues of the vertical operators is scaled up. 
 In this case the conclusion of Theorem \ref{theo:rho-aps} holds assuming that $D$ satisfies 
 the cohomological Witt condition only. 
 The advantage is that the cohomological Witt condition
 does not involve the two representations. \medskip  

In case $D$ is the spin 
Dirac operator and we assume in addition that the singular stratum $B$ is spin and  that we have  
uniform positive scalar curvature. Then we see that the twisted spin Dirac operator
satisfies the geometric Witt condition up to scaling of the metric: indeed, the Lichnerowicz formula
for twisted operators involve an additional curvature term but since our bundles $E_\alpha$ and $E_\beta$
are flat, this additional term is not present and we can argue as usual in order to
obtain the invertibility of the twisted operators. Hence, by scaling, we see that 
the  twisted operators
satisfy the geometric
Witt condition. In fact, by Lichnerowicz formula the width of the gap in the spectrum of
$D$ and its twisted version will be the same.
We conclude that Theorem \ref{theo:rho-aps} holds for the spin Dirac operator assuming that we have  
uniformly positive scalar curvature. The advantage, also in this case, is that this condition involves the metric only and not the representations $\alpha$ and $\beta$.

\subsection{Existence of the Cheeger-Gromov rho invariants}

Let us now consider the Cheeger-Gromov rho invariant on 
an admissible edge manifold $(M,g)$ and its Galois covering $\widetilde{M}$
with Galois group $\Gamma$ and fundamental domain $\mathscr{F}$.
We continue in the notation of \S \ref{eta-galois-section} and 
consider the difference of the 
associated eta functions.
\begin{align}\label{rho-definition-again}
\rho_\Gamma(D, s) := \eta_\Gamma(\widetilde{D},s) - \eta(D,s).
\end{align}

By Proposition \ref{trace-class-Galois} the short time asymptotics
of the trace for $De^{-tD^2}$ and the $\Gamma$-trace for $\widetilde{D}e^{-t\widetilde{D}^2}$ coincide. 
Consequently we obtain the following result.

\begin{thm}
Assume $(M^m,g)$ is an edge manifold with an admissible edge metric
and an edge singularity at $B^b$. Let $D$ be an allowable Dirac-type operator
satisfying the geometric Witt condition. Let $\widetilde{D}$ be its lift to 
the Galois covering $\widetilde{M}$. Then the rho function $\rho_\Gamma(D, s)$ is well-defined at $s=0$
and the corresponding Cheeger-Gromov rho invariant can be defined by 
$$\rho_\Gamma(D) := \rho_\Gamma(D, 0)$$ 
\end{thm}

\section{Stability properties of rho-invariants}\label{section-bordism}

\subsection{Rho invariants of bordant metrics of positive scalar curvature} 

\begin{defn}\label{bordism-def}
Let $(M_1,g_1)$ and $(M_2,g_2)$ be two odd-dimensional incomplete edge spaces with admissible edge metrics $g_1$
and $g_2$ of \emph{uniform positive scalar curvature}. We assume that both $(M_1,g_1)$ and $(M_2,g_2)$ are spin. 
We assume that the singular strata $B_1$ and $B_2$ are also spin.
Finally, we assume the existence of two classifying maps $r_1 : \overline{M}_1\to B\Gamma$
and $r_2: \overline{M}_2\to B\Gamma$.
 We call the two triples $(M_1,r_1,g_1)$ and $(M_2,r_2,g_2)$  
\emph{bordant} if there exists an even-dimensional spin edge space $\overline{W}$ 
with boundary $\partial \overline{W} = \overline{M}_1 \sqcup (-\overline{M}_2)$
equipped on its regular part with a positive scalar curvature metric $G$ 
which is product near the boundary and restricts to 
$g_1$ at the boundary component $M_1$, and to $g_2$ at the boundary component $(-M_2)$.
We also assume the existence of a  classifying map $R: \overline{W}\to B\Gamma$ restricting to $r_1$ and $r_2$ at the boundary. 
As we shall need to apply the APS index theorem we restrict the class of bordisms $W$ to those with edges of odd dimension.
\end{defn}

Consider Theorem \ref{index-main} in the special case of $D$ being the spin Dirac operator. 
We indicate the two spin Dirac operators associated to $(M_1,g_1)$ and $(M_2,g_2)$
by the corresponding lower index $1$ and $2$. 
Pulling back the universal bundle over $B\Gamma$ with the two classifying maps we obtain two
Galois coverings with base respectively $M_1$ and $M_2$.
We  consider two unitary representations 
$\alpha$ and $\beta$ of $\Gamma$ of the same dimension. 
We twist the operators with the corresponding flat 
vector bundles $E_\A$ and $E_\beta$ and indicate dependence of the operators on the representations by an upper 
script $\A$ and $\beta$. By the Weitzenb\"ock formula and the positive scalar curvature assumption,  the index of $\mathbb{D}$
is zero and we obtain from \ref{index-main}
\begin{equation}\label{witt-a-b}
\begin{split}
&\eta(D^\A_1) - \eta(D^\A_2) = 2\left(\int_X a^\A_0 + \int_B b^\A_0 \right) ,\\
&\eta(D^\beta_1) - \eta(D^\beta_2) = 2\left(\int_X a^\beta_0 + \int_B b^\beta_0 \right) \,.\end{split}
\end{equation}
We recall that the integrands $a^{\gamma}_0$
and $b^\gamma_0$ with $\gamma \in \{\A, \beta\}$ are independent of the 
 flat vector bundles.
 Consequently, we find by subtracting the two equalities in \eqref{witt-a-b}
\begin{align}
\rho_{\alpha-\beta} (D_1) =  \rho_{\alpha-\beta} (D_2).
\end{align}

A similar argument applies in the case of Cheeger-Gromov rho invariants,
using Theorem \ref{index-main-galois} in the special case of $\widetilde{D}$ being the spin Dirac operator.
We summarise all this in the following

\begin{thm}
Let  $(M_1,r_1:M_1\to B\Gamma,g_1)$ and $(M_2,r_2:M_2\to B\Gamma,g_2)$  
be positively bordant as in Definition \eqref{bordism-def}.
For the  APS and the Cheeger-Gromov rho invariants we have
\begin{equation}
\begin{split}
\rho_{\alpha-\beta}(D_1) &=  \rho_{\alpha-\beta} (D_2),\\
\rho_{\Gamma}(\widetilde{D}_1) &=  \rho_{\Gamma}(\widetilde{D}_2).
\end{split}
\end{equation}
\end{thm}

\subsection{Metric independence of the signature rho invariants}

We begin with a metric independence result for the 
signature APS invariant. Let $\overline{M}$ be a stratified space of depth one. We assume that 
$\overline{M}$ is Witt. Let $g_1$ and $g_2$ 
be two admissible metrics on the open interior $M$ and let $g_\mu$ 
be a smooth family of admissible edge metrics on $M$ joining $g_1$ and $g_2$. 
We denote the signature operators associated to $g_\mu$  by $D_\mu$.
Up to scaling we can and we shall assume that the signature operators
$D^F_\mu$ on the links satisfy the {\bf geometric} Witt condition 
(they certainly satisfy the cohomological Witt condition).

\begin{thm}\label{independence-theorem}
 Let $\A$ and $\beta$
be any two unitary representations of $\pi_1(\overline{M})$ in $U(\ell)$. Then, under the above hypothesis
\begin{equation*}
 \rho_{\A - \beta} (D_1) = \rho_{\A - \beta} (D_2).
\end{equation*}
\end{thm}

\begin{proof}
We shall prove that the derivative of the rho-invariant associated to $g_\mu$ is equal to $0$ on $(0,1)$.
Denote the flat vector bundles defined by $\A$ and $\beta$ by $E^\A$ and $E^\beta$, respectively.
Denote the operators twisted by $E^\gamma, \gamma\in \{\A, \beta\}$, by an upper subscript $\gamma$. 
We denote the heat operators of $(D_\mu^\A)^2$ and $(D_\mu^\beta)^2$ by $H_\mu^\A$ and $H_\mu^\beta$, 
respectively. Using Theorem \ref{trace-coefficients-flat} we find
\begin{align}
T(t, \mu):= \Tr \left(D_\mu^\A H_\mu^\A \right) - \Tr \left(D_\mu^\beta H_\mu^\beta \right) = O(t^{\infty}), 
\ \textup{as} \ t\to 0.
\end{align}
Since the spectrum of $(D_\mu^\A)^2$ and $(D_\mu^\beta)^2$ is discrete, we also
conclude $T(t, \mu) = O(t^{-\infty})$ as $t\to +\infty$. 
Consequently, the APS rho invariant is given explicitly by evaluating the expressions 
\eqref{APS-definition} and \eqref{eta-heat} at $s=0$
\begin{align}
\rho_{\A - \beta} (D_\mu) = \frac{1}{\sqrt{\pi}} \int_0^\infty T(t, \mu) \frac{dt}{\sqrt{t}}.
\end{align}
In order to establish a variational formula for the rho invariant, we need to ensure that 
the operators act on the same Hilbert space. Hence, we consider the natural isometry
for a fixed $\mu_0 \in I$ and any $\gamma\in \{\A, \beta\}$
\begin{align}
T^\gamma_\mu: L^2(M, {}^{ie}\Lambda^* T^*M \otimes E^\gamma, g_\mu)
\to L^2(M, {}^{ie}\Lambda^* T^*M \otimes E^\gamma, g_{\mu_0}).
\end{align}
By construction, conjugating the Hodge star operator of $(M,g_{\mu_0})$
by the isometry $T^\gamma_\mu$ gives the Hodge star operator of $(M, g_\mu)$.
Consequently, $T^\gamma_\mu$ maps the domain of the twisted signature operator $D^\gamma_\mu$
to the domain of $D^\gamma_{\mu_0}$. We write $\mathscr{D}^\gamma_\mu:= T^\gamma_\mu \circ 
D^\gamma_\mu \circ (T^\gamma_\mu)^{-1}$ with fixed domain $\dom (D^\gamma_{\mu_0})$,
and $\mathscr{H}^\gamma_\mu:= T^\gamma_\mu \circ H^\gamma_\mu \circ (T^\gamma_\mu)^{-1}$. 
This allows us to rewrite $T(t,\mu)$ in terms of operators acting on a fixed Hilbert space
$L^2(M, {}^{ie}\Lambda^* T^*M \otimes E^\gamma, g_{\mu_0})$
\begin{align}
T(t, \mu) = \Tr \left(\mathscr{D}_\mu^\A \mathscr{H}_\mu^\A \right) - 
\Tr \left(\mathscr{D}_\mu^\beta \mathscr{H}_\mu^\beta \right).
\end{align}
We now compute the variation of $T(t,\mu)$ in $\mu$ using the semi-group property
of the heat kernel and conclude, justifying interchange of integration and differentiation 
by the dominated convergence theorem (we also denote derivatives in $\mu$ evaluated at $\mu_0$ by an upper dot),
that the following holds
\begin{equation}\label{mu0}
\begin{split}
\left. \frac{d}{d\mu} \right|_{\mu_0} T(t, \mu) 
&= \Tr \left(\dot{\mathscr{D}}_\mu^\A \mathscr{H}_\mu^\A \right) - 
\Tr \left(\dot{\mathscr{D}_\mu^\beta} \mathscr{H}_\mu^\beta \right) \\
&- t \Tr \left(\mathscr{D}_\mu^\A \left[ \left. \frac{d}{d\mu} \right|_{\mu_0} 
\left(\mathscr{D}_\mu^\A \right)^2 \right] \mathscr{H}_\mu^\A \right)
\\ &+ t \Tr \left(\mathscr{D}_\mu^\beta \left[ \left. \frac{d}{d\mu} \right|_{\mu_0} 
\left(\mathscr{D}_\mu^\beta \right)^2 \right] \mathscr{H}_\mu^\beta \right).
\end{split}
\end{equation}
Note that by construction, $\mathscr{H}_\mu^\gamma$ maps to the 
domain of $(\mathscr{D}^\gamma_\mu)^2$ that is given by $\dom (D^\gamma_{\mu_0})^2$ 
and is in particular independent of $\mu$. Hence the compositions above are well-defined.
Evaluating the derivatives of $\mathscr{D}_\mu^\A$ and $\mathscr{D}_\mu^\beta$
explicitly in terms of the isometries $T_\mu^\gamma$ we find exactly as in \cite[(3.7)]{MV}
($\gamma \in \{\A, \beta\}$)
\begin{equation*}
\begin{split}
\Tr \left(\dot{\mathscr{D}}_\mu^\gamma \mathscr{H}_\mu^\gamma \right)
&=\Tr \left(\dot{T_\mu^\gamma} \, D_\mu^\gamma \, H_\mu^\gamma \, (T^\gamma_\mu)^{-1}\right)
+\Tr \left(T_\mu^\gamma \, \dot{D_\mu^\gamma} \, H_\mu^\gamma \, (T^\gamma_\mu)^{-1}\right)
\\ &+\Tr \left(T_\mu^\gamma \, D_\mu^\gamma \, \dot{T_\mu^\gamma} \, T_\mu^\gamma 
\, H_\mu^\gamma \, (T^\gamma_\mu)^{-1}\right).
\end{split}
\end{equation*}
Employing commutativity of bounded operators under the trace, we conclude
\begin{align}\label{mu1}
\Tr \left(\dot{\mathscr{D}}_\mu^\gamma \mathscr{H}_\mu^\gamma \right)
= \Tr \left( \dot{D_\mu^\gamma} \, H_\mu^\gamma \right),
\end{align}
where we now view the composition $\dot{D_\mu^\gamma} \, H_\mu^\gamma $
as a differential expression $\dot{D_\mu^\gamma}$ applied to the Schwartz kernel 
of $H_\mu^\gamma$. Similar arguments prove
\begin{align}\label{mu2}
\Tr \left(\mathscr{D}_\mu^\gamma \left[ \left. \frac{d}{d\mu} \right|_{\mu_0} 
\left(\mathscr{D}_\mu^\gamma \right)^2 \right] \mathscr{H}_\mu^\gamma \right)
= 2 \Tr \left(\dot{D_\mu^\gamma} \, (D_\mu^\gamma)^2 H_\mu^\gamma \right).
\end{align}
Plugging \eqref{mu1} and \eqref{mu2} into \eqref{mu0}, and using the fact that by 
the Hodge theorem, the kernels of $(D_\mu^\A)^2$ and $(D_\mu^\beta)^2$ are independent
of $\mu$, precluding spectral flow at zero, we obtain after
justifying interchange of integration and differentiation by the dominated convergence theorem
\begin{equation*}
\begin{split}
\left. \frac{d}{d\mu} \right|_{\mu_0} \rho_{\A - \beta} (D_\mu) &= 
\frac{1}{\sqrt{\pi}} \int_0^\infty \left. \frac{d}{d\mu} \right|_{\mu_0} T(t, \mu) \frac{dt}{\sqrt{t}}
\\ &= \frac{1}{\sqrt{\pi}} \int_0^\infty \left( \Tr \left( \dot{D_\mu^\A} \, H_\mu^\A \right)
-  \Tr \left( \dot{D_\mu^\A} \, H_\mu^\A \right)\right) \frac{dt}{\sqrt{t}} \\
&- \frac{2}{\sqrt{\pi}} \int_0^\infty \left( \Tr \left( \dot{D_\mu^\A} \, \frac{d}{dt} H_\mu^\A \right)
-  \Tr \left( \dot{D_\mu^\A} \, \frac{d}{dt} H_\mu^\A \right)\right) \, \sqrt{t} \, dt.
\end{split}
\end{equation*}
Using integration by parts we obtain after cancellations
\begin{equation}\label{rho-variation}
\begin{split}
\left. \frac{d}{d\mu} \right|_{\mu_0} \rho_{\A - \beta} (D_\mu) 
= \frac{1}{\sqrt{\pi}} \int_0^\infty \frac{d}{dt} \left( \sqrt{t} \Tr \left( \dot{D_\mu^\A} \, H_\mu^\A \right)
-  \sqrt{t} \Tr \left( \dot{D_\mu^\A} \, H_\mu^\A \right)\right) \, dt = 0,
\end{split}
\end{equation}
where the last equality uses the fact that $\Tr \left( \dot{D_\mu^\A} \, H_\mu^\A \right)
-  \Tr \left( \dot{D_\mu^\A} \, H_\mu^\A \right)$ is vanishing to infinite order as $t$ goes either
zero or to infinity by the Hodge theorem on Witt spaces and by Theorem \ref{trace-coefficients-flat}.
\end{proof}

Next we tackle the Cheeger-Gromov rho invariant. 
We follow the very detailed proof of their result 
\cite{Cheeger-Gromov} given
in Roy \cite{indrava} and Azzali-Wahl \cite{azzali-wahl}

\begin{thm}\label{CG-independence}
Under the same assumptions as in Theorem \ref{independence-theorem}
we also have
\begin{equation*}
 \rho_\Gamma (D_1) = \rho_\Gamma (D_2).
\end{equation*}
\end{thm}

\begin{proof}
We continue in the notation set in the previous Theorem \ref{independence-theorem}
and indicate the operators on the Galois covering by an upper tilde. The twisting flat vector 
bundle is fixed in the current setting and hence is not indicated notationally here. The $\Gamma$-trace
is denoted by $\Tr_\Gamma$. Using Proposition \ref{trace-class-Galois} we observe
\begin{align}
T_\Gamma(t, \mu):= \Tr_\Gamma \left(\widetilde{D}_\mu 
\widetilde{H}_\mu \right) - \Tr \left(D_\mu H_\mu \right) = O(t^{\infty}), 
\ \textup{as} \ t\to 0.
\end{align}
Due to discreteness of the spectrum of $(D_\mu)^2$ and due to \eqref{eta-est}, 
$t^{-1/2} T_\Gamma(t, \mu)$ is integrable in $t\in [1,\infty)$. Together with the infinite
vanishing of $T_\Gamma(t, \mu)$ as $t\to 0$, we may write the Cheeger-Gromov 
invariant by evaluating the expressions for \eqref{rho-definition-again} at $s=0$
\begin{align}
\rho_\Gamma (D_\mu) = \frac{1}{\sqrt{\pi}} \int_0^\infty T_\Gamma (t, \mu) \frac{dt}{\sqrt{t}}.
\end{align}
The formulae \eqref{mu0}, \eqref{mu1} and \eqref{mu2} for the variation of 
$T (t, \mu)$ employ only the trace class property of the operators, the fact that 
bounded operators commute under the trace and the semi-group property of the 
heat kernel. Since these properties continue to hold under the Gamma-trace and the
heat kernel on the Galois covering, we conclude by repeating the same arguments 
as in Theorem \ref{independence-theorem}
\begin{equation}
\begin{split}
\frac{d}{d\mu} T_\Gamma(t, \mu) 
&= \Tr_\Gamma \left( \dot{\widetilde{D}}_\mu \, \widetilde{H}_\mu \right) - 
\Tr \left(\dot{D}_\mu H_\mu \right) \\
&- 2t \Tr_\Gamma \left(\dot{\widetilde{D}}_\mu \, (\widetilde{D}_\mu)^2 \widetilde{H}_\mu \right)
+ 2t \Tr \left(\dot{D}_\mu \, (D_\mu)^2 H_\mu \right)\end{split}
\end{equation}
Exactly as in \eqref{rho-variation} we conclude using integration by parts 
\begin{equation}\label{rho-variation-Galois}
\begin{split}
\left. \frac{d}{d\mu} \right|_{\mu_0} \rho_{\Gamma} (D_\mu) 
= \frac{1}{\sqrt{\pi}} \int_0^\infty \frac{d}{dt} \left( \sqrt{t} \Tr_\Gamma \left( \dot{\widetilde{D}}_\mu \, 
\widetilde{H}_\mu \right)
-  \sqrt{t} \Tr \left( \dot{D}_\mu \, H_\mu \right)\right) \, dt,
\end{split}
\end{equation}
provided the limits of $\sqrt{t} \Tr_\Gamma \left( \dot{\widetilde{D}}_\mu \, 
\widetilde{H}_\mu \right)$ and $\sqrt{t} \Tr \left( \dot{D}_\mu \, H_\mu \right)$
as $t\to 0$ and as $t\to \infty$ are well-defined. We now study these limits.
Using the Duhamel principle worked out in Proposition \ref{trace-class-Galois},
the short time asymptotic behaviour of the two traces coincides and hence 
\begin{align}\label{limit-zero}
\lim_{t\to 0} \sqrt{t} \left( \Tr_\Gamma \left( \dot{\widetilde{D}}_\mu \, 
\widetilde{H}_\mu \right) - \Tr \left( \dot{D}_\mu \, H_\mu \right) \right) = 0.
\end{align}
For the large time behaviour we proceed on the Galois covering as in \S 
\ref{spectral-calculus}. Note that $\dot{\widetilde{D}}_\mu \, 
\widetilde{H}_\mu = \dot{*}_\mu d \, \widetilde{H}_\mu$, where $*_\mu$ 
denotes the Hodge star operator associated to $g_\mu$ and $d$ 
is the exterior derivative. Thanks to the classical inequalities in the von Neumann algebras 
we can estimate
\begin{align*}
\left| \Tr_\Gamma \left( \dot{\widetilde{D}}_\mu \, 
\widetilde{H}_\mu \right)\right|  =  \left| \Tr_\Gamma \left( \dot{*}_\mu d \, 
\widetilde{H}_\mu \right) \right| &\leq \|\dot{*}_\mu *_\mu^{-1}\| \Tr_\Gamma \left| *_\mu d \, 
\widetilde{H}_\mu \right| \\ &= \|\dot{*}_\mu *_\mu^{-1}\| \Tr_\Gamma \left| \widetilde{D}_\mu \, 
\widetilde{H}_\mu \right| \leq C \Tr_\Gamma \left| \widetilde{D}_\mu \, 
\widetilde{H}_\mu \right| 
\end{align*}
for some constant $C>0$ independent of $\mu$ and $t$. Recall
the tempered measure $m_\Gamma$ relative to $\widetilde{D}_\mu$, introduced in \S \ref{spectral-calculus}.
Consider the tempered measure $n_\Gamma$ relative to $\widetilde{D}_\mu^2$.  We can then write
\begin{align*}
\Tr_\Gamma \left| \widetilde{D}_\mu \, 
\widetilde{H}_\mu \right| = \int_0^\lambda \sqrt{x} e^{-tx} dn_\Gamma(x) 
+ \int_\lambda^\infty \sqrt{x} e^{-tx} dn_\Gamma(x) =: I_1(t) + I_2(t).
\end{align*}
Let us denote by $\widetilde{E}_x$
of $\widetilde{D}_\mu^2$. 
Note that the function $\sqrt{x} e^{-tx}$ has its maximum at $x= 1/ (2t)$
with value $e^{-1/2}/\sqrt{2t}$. Thus we can estimate $I_1(t)$ as follows
\begin{align*}
I_1(t) = \int_0^\lambda \sqrt{x} e^{-t x^2} dn_\Gamma(x) 
&\leq  (2t)^{-1/2} e^{-1/2} \int_0^\lambda dn_\Gamma(x) 
\\ &\leq  (2t)^{-1/2} e^{-1/2} \Tr_\Gamma ((1-E_0)E_\lambda (1-E_0)).
\end{align*}
We estimate $I_2(t)$ exactly as in $\eqref{eta-est}$ and obtain
\begin{align*}
I_2(t) \leq t^{-1/2} e^{-(t-1)\lambda} \Tr_\Gamma \widetilde{H}_\mu(1). 
\end{align*}
Note that $\sqrt{t}\, I_2(t)$ converges to zero as $t\to \infty$. 
From here we conclude 
\begin{align*}
\lim_{t\to \infty} \left| \sqrt{t} \Tr_\Gamma \left( \dot{\widetilde{D}}_\mu \, 
\widetilde{H}_\mu \right) \right| \leq \lim_{t\to \infty}  \sqrt{t} (I_1(t) + I_2(t))
 \leq (2e)^{-1/2} \Tr_\Gamma ((1-E_0)E_\lambda (1-E_0)).
\end{align*}
Since the left hand side of the inequality is independent of $\lambda$, 
we may take $\lambda \to 0$, and noting by normality of the Gamma-trace
that $\Tr_\Gamma ((1-E_0)E_0 (1-E_0))=0$, we finally obtain
\begin{align}\label{limit-infinity}
\lim_{t\to \infty} \sqrt{t} \Tr_\Gamma \left( \dot{\widetilde{D}}_\mu \, 
\widetilde{H}_\mu \right)  = 0.
\end{align}
Similar argument holds for $\sqrt{t} \Tr \left( \dot{D}_\mu \, H_\mu \right)$
with integrals replaced by sums due to discreteness of the spectrum. 
In view of \eqref{limit-zero} and \eqref{limit-infinity} we conclude from 
\eqref{rho-variation-Galois}
\begin{equation}
\begin{split}
\left. \frac{d}{d\mu} \right|_{\mu_0} \rho_{\Gamma} (D_\mu) 
&= \frac{1}{\sqrt{\pi}} \lim_{t\to \infty} \left( \sqrt{t} \Tr_\Gamma \left( \dot{\widetilde{D}}_\mu \, 
\widetilde{H}_\mu \right) -  \sqrt{t} \Tr \left( \dot{D}_\mu \, H_\mu \right)\right) \\
&- \frac{1}{\sqrt{\pi}}\lim_{t\to 0} \left( \sqrt{t} \Tr_\Gamma \left( \dot{\widetilde{D}}_\mu \, 
\widetilde{H}_\mu \right) -  \sqrt{t} \Tr \left( \dot{D}_\mu \, H_\mu \right)\right) = 0.
\end{split}
\end{equation}

\end{proof}

\subsection{Stratified diffeomorphism invariance of the signature rho invariants}
The metric invariance of the APS and the Cheeger-Gromov rho invariant implies immediately
their stratified diffeomorphism invariance. Indeed, 
let $\overline{f}:\overline{M}\to \overline{M}^\prime$ be a stratified
diffeomorphism between two edge spaces and let $f:M\to M^\prime$ be  the induced diffeomorphism on 
the respective regular parts. Fix an admissible edge structure $g^\prime$ on $M^\prime$ and consider
its pull-back $f^* g^\prime$. This is an admissible edge structure on $M$; the operation of
pull-back defines a unitary isomorphism $U: L^2(M^\prime, g)\to L^2 (M,f^*g)$ and, more generally,
a unitary isomorphism $U: L^2 (M^\prime, {}^{ie}\Lambda^* T^* (M^\prime))\to L^2 (M, {}^{ie}\Lambda^* T^* M)$.
One check easily that the signature operator associated to $f^* g^\prime$ is obtained by conjugating the
signature operator on $M^\prime$ through the unitary isomorphism $U$; in formulae, and adopting a very precise
notation for the signature operator $D$, we have:
$$D^{f^*g^\prime}_M= U^{-1} \circ D^{g^\prime}_{M^\prime} \circ U$$
Consequently, by functional calculus,
$$ D^{f^*g^\prime}_M \exp (- t (D^{f^*g^\prime}_M)^2)=U^{-1} \circ D^{g^\prime}_{M^\prime} \exp (-t 
(D^{g^\prime}_{M^\prime})^2)\circ U$$
which  implies the equality 
\begin{equation}\label{equality-of-eta}
\eta (D^{f^*g^\prime}_M)= \eta (D^{g^\prime}_{M^\prime})
\end{equation}
when one of the two is well defined.
This implies easily that if  $\alpha^\prime$ and $\beta^\prime$ are  two  representations of $\pi_1 (\overline{M}^\prime)$ into
$U(\ell)$ and 
$\alpha:=  \alpha^\prime\circ \overline{f}_*$, $\beta  := \beta^\prime  \circ  \overline{f}_* $ then
\begin{equation}\label{stratified-invariance-aps}
\rho_{\alpha-\beta} (D^{f^*g^\prime}_M)= \rho_{\alpha^\prime- \beta^\prime} (D^{g^\prime}_{M^\prime})
\end{equation}
Since both members are metric independent we conclude that for the signature operator $D$ the following equality holds:
\begin{equation}\label{stratified-invariance-aps-bis}
\rho_{\alpha-\beta} (D_M)= \rho_{\alpha^\prime- \beta^\prime} (D_{M^\prime})
\end{equation}
A similar argument applies to two Galois $\Gamma$-coverings 
$$ \overline{M}_\Gamma\to \overline{M}\,,\quad\quad  \overline{M}_\Gamma^\prime \to \overline{M}^\prime$$
endowed with a 
$\Gamma$-equivariant stratified diffeomorphism $$\overline{f}_\Gamma:
\overline{M}_\Gamma\to   \overline{M}_\Gamma^\prime\,.$$
Let $f_\Gamma: \widetilde{M}\to \widetilde{M}^\prime$ be the induced $\Gamma$-equivariant diffeomorphism
between the two regular strata.
We fix a $\Gamma$-equivariant edge structure $\tilde{g}^\prime$ on $\widetilde{M}^\prime$ and we consider
$f_\Gamma^*(\tilde{g}^\prime)$ on $\widetilde{M}$; we denote by
$ \widetilde{D}^{\tilde{g}^\prime}$ and  $\widetilde{D}^{f_\Gamma^*\tilde{g}^\prime}$
the two corresponding signature operators.
Then,
proceeding as above, one proves in addition to \eqref{equality-of-eta} the following equality:
\begin{equation}\label{equality-of-gammaeta}
\eta_{\Gamma} (\widetilde{D}^{f_\Gamma^*\tilde{g}^\prime})= \eta_\Gamma (\widetilde{D}^{\tilde{g}^\prime})
\end{equation}
whenever one of the two is well defined.
We thus conclude, by metric independence, that
\begin{equation}
\rho_{\Gamma} (\widetilde{D}_{\widetilde{M}})=\rho_{\Gamma} (\widetilde{D}_{\widetilde{M}^\prime})\,.
\end{equation}

\section{Open problems and future research directions}

In this final section we wish to highlight some open problems and future research 
directions which are strongly connected to the discussion presented here. We plan to tackle 
these issues in future projects.

\subsection{Residue of the eta-function at zero}
The residue of the eta-function at zero is a sum of a term coming from the interior and a term coming from the edge singularity.
The term coming from the interior is an integral of a local quantity over $M$. 
The term coming from the edge $B$ is an integral over $B$ of a term that is global in the fibres of
the edge fibration $\phi: \partial M \to B$ and local along $B$. 
It would be interesting to understand the structure of the interior and the edge contribution 
more explicitly. Another related open question is if the APS index theorem still holds (with some
additional correction terms), if the residue of the eta-function does not vanish.

\subsection{Heat kernel analysis for Cheeger boundary conditions}
Throughout our arguments we have posed the geometric Witt condition to ensure 
essential self-adjointness of the corresponding Dirac operators. In case the geometric 
Witt condition is not satisfied, one may pose Cheeger boundary conditions, which have 
been introduced by the first named author, jointly with Albin, Leichtnam and Mazzeo 
in \cite{Cheeger-spaces}. One would like to extend our discussion to the signature operator equipped 
with Cheeger boundary conditions; this  requires a microlocal construction of the corresponding
 heat kernel.

\subsection{Analytic torsion on singular Galois coverings}
There has been intensive research on $L^2$-invariants, in particular
analytic torsion on Galois coverings, cf. the incomplete list of references 
\cite{An1, An2, An3, An4}. Our analysis here lays the groundwork for 
defining $L^2$-analytic torsion on Galois coverings of simple incomplete 
edge spaces. 

\subsection{Signature formula on simple edge spaces}
The two Atiyah-Patodi-Singer index theorems established here may be applied to the special 
case of the signature operator in order to derive a signature formula on edge manifolds
with boundary and an $L^2$-signature formula on on $\Gamma$-edge manifolds with boundary
(satisfying, of course, the Witt condition).
As in the work of Atiyah-Patodi-Singer and in the work of Valliant 
\cite{Vaillant} and L\"uck-Schick \cite{Lueck-Schick},
there is an additional Hodge theoretic argument in order to pass from an index formula of APS type
to a true signature formula. This is ongoing work of the authors.

\def\cprime{$'$}
\providecommand{\bysame}{\leavevmode\hbox to3em{\hrulefill}\thinspace}
\providecommand{\MR}{\relax\ifhmode\unskip\space\fi MR }
\providecommand{\MRhref}[2]{%
  \href{http://www.ams.org/mathscinet-getitem?mr=#1}{#2}
}
\providecommand{\href}[2]{#2}

\end{document}